\documentclass[11pt]{amsart}

\usepackage[T1]{fontenc}

\usepackage{mathrsfs}
\usepackage{amsmath}
\usepackage{amssymb}
\usepackage{amsthm}
\usepackage{tikz-cd}
\usepackage{systeme}
\usepackage[margin=1in]{geometry}
\usepackage{url}
\usepackage{mathtools}
\usepackage{bm}
\usepackage{calligra}
\usepackage{multicol}
\usepackage{subfig}
\usepackage{float}
\usepackage[colorlinks = true,
linkcolor = blue,
urlcolor  = purple,
citecolor = blue,
anchorcolor = blue]{hyperref}

\newcommand{\red}{\color[rgb]{0.7,0,0}}

\newcommand{\leqp}{%
  \mathrel{\raisebox{-0.5ex}{$\scriptscriptstyle($}}%
  \leq
  \mathrel{\raisebox{-0.5ex}{$\scriptscriptstyle)$}}%
}

\graphicspath{ {pictures/} }

\theoremstyle{plain}
\newtheorem{theorem}{Theorem}[section]
\newtheorem{corollary}[theorem]{Corollary}
\newtheorem{proposition}[theorem]{Proposition}
\newtheorem{lemma}[theorem]{Lemma}
\newtheorem{conjecture}[theorem]{Conjecture}

\theoremstyle{definition}
\newtheorem{definition}[theorem]{Definition}
\newtheorem{remark}[theorem]{Remark}

\newtheorem{example}[theorem]{Example}
\newtheorem{question}[theorem]{Question}

\numberwithin{equation}{theorem}

\newcommand{\ZZ}{\boldsymbol{\mathbb{Z}}}

\newcommand{\XX}{\mathbb{P}^1 \times \mathbb{P}^1}

\newcommand{\cc}[1]{\mathcal{#1}}

\newcommand{\bb}[1]{\mathbb{#1}}
\newcommand{\OO}{\mathcal{O}}
\newcommand{\VV}{\mathcal{V}}
\newcommand{\PP}{\mathbb{P}}

\newcommand{\surj}{\twoheadrightarrow}
\newcommand{\inj}[1]{\xhookrightarrow{#1}}

\DeclarePairedDelimiter\floor{\lfloor}{\rfloor}
\newcommand{\codim}{\textnormal{codim}}
\newcommand{\ext}{\textnormal{ext}}
\newcommand{\Ext}{\textnormal{Ext}}
\newcommand{\Hom}{\textnormal{Hom}}
\newcommand{\Pic}{\textnormal{Pic}}
\newcommand{\DLP}{\textnormal{DLP}}
\newcommand{\rk}{\textnormal{rk}}
\newcommand{\gr}{\textnormal{gr}}
\newcommand{\Char}{\textnormal{Char}}
\newcommand\dual{\raise0.2ex\hbox{$\scriptscriptstyle\vee$}}
\newcommand{\sslash}{\mathbin{/\mkern-6mu/}}
\newcommand*{\sheafhom}{\mathscr{H}\text{\kern -3pt {\calligra\large om}}\,}
\newcommand{\bv}{{\bf v}}
\newcommand{\bw}{{\bf w}}
\newcommand{\be}{{\bf e}}

\begin{document}

\sloppy

\date{\today}
\author[D. Pedchenko]{Dmitrii Pedchenko}
\address{Department of Mathematics, The Pennsylvania State University, University Park, PA 16802}
\email{dzp5326@psu.edu}


\title{The Picard group of the moduli space of sheaves on a quadric surface}
\thanks{During the preparation of this article the author was partially supported by the NSF FRG grant DMS-1664303}

\begin{abstract}
In this paper, we study the Picard group of the moduli space of semistable sheaves  on a smooth quadric surface. We polarize the surface by an ample divisor close to the anticanonical class. We focus especially on moduli spaces of sheaves of small discriminant, where we observe new and interesting behavior. Our method relies on constructing certain resolutions for semistable sheaves  and applying techniques of geometric invariant theory to the resulting families of sheaves.\end{abstract}

\maketitle

\tableofcontents

\section{Introduction} In this paper, we are concerned with the calculation of the Picard group of the moduli space of semistable sheaves on the quadric surface $\XX$.

Let $Y$ be a smooth complex projective surface and let $H$ be an ample divisor on $Y$. Consider the moduli space of sheaves $M(\bv)$ parameterizing $S$-equivalence classes of $H$-Gieseker semistable sheaves with Chern character $\bv$ on $Y$. These moduli spaces have been intensively studied over the years, but many basic questions about their geometry remain open. For that matter, calculating the Picard group  is the first necessary step towards understanding the birational geometry of these spaces.

The starting point of our investigations is \cite{Dr88} where Dr\'{e}zet computes the Picard group of the moduli space of semistable sheaves on $\PP^2$.  Recall that for a smooth surface $Y$ the \emph{total slope} and \emph{discriminant} of a Chern character $\bv \in K(Y)$ of rank $r$ are defined by $$\nu=\frac{c_1}{r}, \quad  \ \Delta=\frac{1}{2}\nu^2-\frac{ch_2}{r}.$$

Let $Y=\PP^2$. By \cite{DL85}, there is a fractal-like curve $\DLP(\nu)$ in the $(\nu, \Delta)$-plane, which we call the \emph{Dr\'{e}zet-Le Potier curve}, such that the moduli space $M(\bv)$ is positive dimensional if and only if $\Delta \geq \DLP(\nu)$, where $\bv=(r, \nu, \Delta)$. The Dr\'{e}zet-Le Potier curve is comprised of branches $\DLP_{E}(\nu)$ indexed by all exceptional bundles $E$ on $\PP^2$ where for each branch the value $\DLP_E(\nu)$ is calculated using the numerical invariants of $E$.  Dr\'{e}zet \cite{Dr88} shows that 
\begin{enumerate}
\item If $\Delta>\DLP(\nu)$, then  $\Pic(M(\bv)) \cong \ZZ^2$,

\item If $\Delta=\DLP(\nu)$, then  $\Pic(M(\bv)) \cong \ZZ$.
\end{enumerate}
This way, the Picard number of $M(\bv)$ is determined by the position of $\bv=(r,\nu, \Delta)$ relative to the $\DLP$-curve.

Now, let $Y=\XX$ with a generic polarization $H$ close to the anticanonical one. Rudakov \cite{Rudakov94}  constructed a fractal-like surface $\DLP(\nu)$ in the $(\nu, \Delta)$-space, \emph{the Dr\'{e}zet-Le Potier surface}, such that again the moduli space $M(\bv)$ is positive dimensional if and only if $\Delta\geq \frac{1}{2}$ and $\Delta \geq \DLP(\nu)$, where $\bv=(r, \nu,\Delta)$. This surface is also comprised of branches $\DLP_{E}(\nu)$ indexed by all exceptional bundles $E$ on $\XX$ with $r(E)<r$ where for each branch the value $\DLP_{E}(\nu)$ is calculated using the numerical invariants of $E$. The main new feature compared  with the $\PP^2$ case is that now there exist integral Chern characters with positive dimensional moduli space $M(\bv)$ which lie on the \emph{intersection of two} branches of the $\DLP$-surface.

This way, we see that for $\XX$  character $\bv=(r, \nu, \Delta)$ with positive dimensional moduli space $M(\bv)$ can be positioned in three different ways with respect to the $\DLP$-surface: (1) $\bv$  lies above the $\DLP$-surface, (2) $\bv$  lies on a single branch of the $\DLP$-surface, and (3) $\bv$ lies on the intersection of two branches of the $\DLP$-surface.

Our key finding is that contrary to the $\PP^2$ case the Picard number $\rho$ of $M(\bv)$ is \emph{not} determined only by the position of $\bv$ relative to the $\DLP$-surface. The main results of this paper are  summarized in the following theorem.

\begin{theorem}[{See Theorems \ref{Main theorem} and \ref{Main theorem 2}}] Let ${\bf v}=(r, \nu, \Delta)\in K(\XX)$ be a character with $r \geq 2$ and $\Delta\geq \frac{1}{2}$. 

\begin{enumerate}\item If    $\bv'=(r, \nu, \Delta-\frac{1}{r})$ lies above the $\DLP$-surface, then $\rho(M(\bv))=3$.

\item If $\bv$ lies on a single branch of the $\DLP$-surface, then $$\rho(M(\bv))=2 \ \text{or} \ \rho(M(\bv))=1.$$

\item If ${\bf v}$ lies on the intersection of two branches of the $\DLP$-surface, then $\rho(M(\bv))=1$.
\end{enumerate}

Furthermore, if $\bv$ is a primitive character, then $\Pic(M(\bv))$ is a free abelian group of rank $\rho$. 
\end{theorem}

Let us explain the dichotomy in case (2) of the above theorem in greater detail. We split Chern characters into  two groups, calling the characters in the first group \emph{good} characters and the characters in the second group \emph{bad} characters, see Definition \ref{bad character}. For a good character $\bv$ lying on a single branch of the $\DLP$-surface the Picard number of $M(\bv)$ is equal to $2$, see Theorem \ref{Main theorem}. On the other hand, we construct  infinite sequences of bad Chern characters lying on a single branch of the $\DLP$-surface for which the Picard number of the moduli space drops to $1$, see Examples \ref{bad example}, \ref{bad example 2}. Moreover, we show that $\rho(M(\bv))=1$ for any bad character lying on a single branch of the $\DLP$-surface given by a \emph{line bundle}, see Theorem \ref{Main theorem 2}. When $\bv$ is a bad character of the smallest rank,  the moduli space $M(\bv)$ turns out to be isomorphic to a projective space, see Example \ref{bad example 3} and Question \ref{interesting question}. 

We emphasize that determining which statement of the above theorem applies to a given character ${\bf v}=(r, \nu, \Delta)$ is a finite computational procedure and therefore can be implemented on a computer: both the computation of $\DLP(\nu)$ and determining whether ${\bf v}$ is a good or a bad character are finite computational procedures.

The proofs of the above results rely on constructing resolutions for semistable sheaves of a given character $\bv$ and applying techniques of Mumford's geometric invariant theory to the resulting families of sheaves described by such resolutions. In the case of $\PP^2$ the most powerful tool for constructing resolutions of semistable sheaves is the Beilinson-type spectral sequence coming from a choice of a full exceptional collection (see \cite[\S 5]{67c4a1ff6f134980a50953a83ab223e7} for a detailed analysis).  The main difficulty is that  full exceptional collections on $\XX$ require four exceptional bundles instead of three in the case of $\PP^2.$ As a result, for most characters $\bv$ writing the associated Beilinson-type spectral sequence   no longer gives a resolution of a semistable sheaf $\VV$ of character $\bv$ on $\XX$ as a (co)kernel in a short exact sequence.  To circumvent this difficulty we instead use the so-called Gaeta-type resolutions of Coskun and Huizenga constructed in \cite{coskun_huizenga_2018}.

We conjecture that for all characters $\bv$ with positive-dimensional moduli space the Picard number of $M(\bv)$ is fully determined by the relative position of $\bv$ with respect to the $\DLP^{<r}$-surface and by whether the character is good or bad (see  Question \ref{question} and Conjecture \ref{conjecture}).

Finally, let us survey the previous results on   the Picard group of the moduli space of semistable sheaves on $\XX$. The Picard group of $M(\bv)$ was studied by Nakashima  \cite{Nak93} and Qin \cite{Qin92}  for characters $\bv=(r, \nu, \Delta)=(r, c_1, \chi) \in K(\XX)$ satisfying $r=2$ and $c_1 \cdot F=1$,
by Yoshioka \cite{Yoshioka1995} for characters $\bv$ satisfying $r=2$, and by Yoshioka  \cite{Yosh96} for characters $\bv$ satisfying $c_1 \cdot F=0$ with the asymptotic polarization $H_m=E+mF, \ m \gg 0$ (here $E$ and $F$ are the standard generators of $\Pic(\XX)$).   Yoshioka {\cite{Yos96}}
also  computed equivariant Picard groups of certain spaces closely related to the moduli space $M(\bv)$. The current paper is a natural continuation of this line of work. 

\subsection*{Organization of the paper} In \S\ref{prelim}, we recall the preliminary facts and survey the known results concerning vector bundles and moduli spaces of sheaves on $\XX$ needed in the rest of the paper.  

Sections \S3-5 form the technical core of the paper. In \S\ref{Shatz section}, we study the Shatz stratification of complete families of sheaves on $\XX$. We also prove the irreducibility of families parameterizing sheaves with a fixed Harder-Narasimhan filtration, which we later use to prove the irreducibility of Shatz strata in complete families of vector bundles admitting a Gaeta-type resolution. In \S\ref{Group act and Gaet}, we establish basic facts about group actions in the context of Gaeta-type resolutions. In \S\ref{main section}, we calculate the Picard group of $M(\bv)$ under the assumption that $\bv$ is a good character.

Finally in \S\ref{last section}, we study $\Pic(M(\bv))$ for bad characters $\bv$ that lie on a single branch of the $\DLP$-surface given by a line bundle. 

\subsection*{Acknowledgments} We would like to thank Jack Huizenga for his support, encouragement and helpful discussions throughout the project. We would also like to express gratitude to K\={o}ta Yoshioka, Jean-Marc Dr\'{e}zet and Daniel Levine for valuable discussions.

\section{Preliminaries}\label{prelim} 
In this section we recall basic facts, previous results and constructions concerning moduli spaces of sheaves that will be used in the rest of the paper. We will denote an arbitrary variety or an arbitrary projective surface by $Y$, while $X$ will be always reserved for the quadric surface $X=\XX$.

\subsection{Chern characters} Given a torsion-free sheaf $\VV$ on a surface $Y$ and an ample divisor $H$, the \emph{total slope} $\nu$, the $H$-\emph{slope} $\mu_H$ and the \emph{discriminant} $\Delta$ are defined by $$\nu(\VV)=\frac{ch_1(\VV)}{ch_0(\VV)}, \quad \mu_H(\VV)=\frac{ch_1(\VV)\cdot H}{ch_0(\VV)}, \quad \Delta(\VV)=\frac{1}{2}\nu^2 - \frac{ch_2(\VV)}{ch_0}.$$These quantities depend only on the Chern character of $\VV$ and not on the particular sheaf. Given a Chern character ${\bf v}\in K(Y)$, we define its total slope, $H$-slope and discriminant by the same formulae. We will often record Chern characters by the rank, total slope and discriminant. Note that one can recover the Chern classes from this data.

\subsection{Stability} We refer the reader to \cite{huizenga2017birational}, \cite{HL10} and \cite{LP97} for more detailed discussions. Let  $Y$ be a surface and $H$ be an ample divisor on it. A torsion-free coherent sheaf $\VV$ is called $\mu_H$-\emph{(semi)stable} (or \emph{slope} (semi)stable) if every proper subsheaf $0 \neq  \cc{W} \subsetneq \VV$ of smaller rank satisfies $$\mu_H(\cc{W}) \leqp \mu_H(\VV).$$Define the $H$-\emph{Hilbert polynomial} $P_{H,\VV}(m)$ and the \emph{reduced $H$-Hilbert polynomial} $p_{H,\VV}(m)$ of a torsion-free sheaf $\VV$ by $$P_{H,\VV}(m)=\chi(\VV(mH)), \ p_{H,\VV}(m)=\frac{P_{\VV}(m)}{r(\VV)}.$$A torsion-free sheaf $\VV$ is $H$-\emph{(semi)stable} (or \emph{Gieseker} (semi)stable) if for every proper subsheaf $\cc{W} \subset \VV$, we have $$p_{H,\cc{W}}(m) \leqp p_{H,\VV}(m) \ \text{for} \ m \gg 0.$$ Slope stability implies Gieseker stability and Gieseker semistability implies slope semistability.

Every torsion-free sheaf $\VV$ admits a \emph{Harder-Narasimhan} filtration with respect to both $\mu_H-$ and $H$-semistability, that is there is a finite filtration $$0=\VV_0 \subset \VV_1 \subset \VV_2 \subset ... \subset \VV_n=\VV,$$such that the quotients $\cc{W}_i=\VV_i/\VV_{i-1}$ are $\mu_H$ (respectively, $H$-Gieseker) semistable and $$\mu_H(\cc{W}_i)>\mu_H(\cc{W}_{i-1}) \ \ \ (\text{respectively,} \ p_{H,\cc{W}_i}(m) > p_{H,\cc{W}_{i-1}}(m) \ \text{for} \ m \gg 0)$$for $1 \leq i \leq n$. The Harder-Narasimhan filtration is unique. A semistable sheaf further admits a \emph{Jordan-H\"{o}lder} filtration into stable sheaves. Two semistable sheaves are called $S$-\emph{equivalent} if they have the same associated graded objects with respect to a Jordan-H\"{o}lder filtration. 

Our main object of study will be the moduli space $M_{H}({\bf v})$ parameterizing $S$-equivalence classes of $H$-Gieseker semistable sheaves of character ${\bf v}$ on $Y$. We refer the reader to \cite[\S 4.3]{HL10} for the details about the construction of $M_H(\bv)$ and its basic properties. 
 
\subsection{Choosing the polarization}\label{Polarization} For our purposes, we would like to work with a locally factorial moduli space. After recalling some definitions and results  from \cite[\S4.C]{HL10}, we show that if $Y$ is rational surface other than $\PP^2$, then it is always possible to vary the polarization $H$ slightly so that $M_H(\bv)$ becomes locally factorial.

Let $Y$ be a smooth projective surface. The intersection pairing defines a bilinear form on $Num(Y)$ and the Hodge Index Theorem implies that the extension of this bilinear form to $Num_{\bb{R}}(Y)$ defines the Minkowski metric on $Num_{\bb{R}}$ (that is the signature of the form is $(1,N)$).
Define the positive cone as $$K_+:=\{y \in Num_{\bb{R}}(Y) \ | \ y \cdot y >0 \ \text{and} \ y \cdot H >0 \ \text{for some ample divisor } H\},$$and note that it contains the positive span of ample divisors as an open subcone $Amp(Y)$. Since we can think of a polarization given by an ample divisor as a ray $\bb{R}_{>0} H \subset K_+$, it is convenient to introduce $\cc{H}$ as the set of rays in $K_+$. This set becomes a hyperbolic manifold if we make the identification $\cc{H} \cong \{ H \in K_+ \ | \ H\cdot H=1 \}$.

\begin{definition}[{\cite[Definition 4.C.1]{HL10}}] Let $r \geq 2$ be an integer and $\Delta >0$ a real number\footnote{Note that the definition of the discriminant $\Delta$ we are using in this paper differs from the definition of discriminant $\tilde{\Delta}$ in \cite{HL10}: $ \tilde{\Delta}=2r^2\Delta$. That is why some formulas in this subsection differ by a factor $2r^2$ compared to the formulas in \cite[\S 4.C]{HL10}. }. A class $\xi \in Num(X)$ is of {\it type} $(r, \Delta)$ if $-\frac{r^4}{2} \Delta \leq \xi \cdot \xi <0$.  The {\it wall} defined by $\xi$ is the real hypersuface $$W_{\xi} :=\{\bb{R}_{>0}H \in \cc{H} \ | \ \xi \cdot H=0 \} \subset \cc{H}.$$
\end{definition}

When $r\geq 2$ and $\Delta>0$, Lemma 4.C.2 in \cite{HL10} asserts that the set of walls of type $(r, \Delta)$ is locally finite in $\cc{H}$. It is therefore always possible to choose $H$ to not lie on any wall of type $(r, \Delta)$ by a small perturbation. 
\begin{lemma}[{\cite[Lemma 4.C.3]{HL10}}]\label{Huybrechts-Lehn} Let $H$ be an ample divisor, let $\VV$ be a $\mu_H$-semistable sheaf of rank $r$ and discriminant $\Delta$ on $Y$ and let $\VV' \subset \VV$ be a subsheaf of rank $r'$, $0<r'<r$, with $\mu_H(\VV')=\mu_H(\VV)$. Then $\xi:=r c_1(\VV')-r'c_1(\VV)$ satisfies $$\xi \cdot H=0 \ \text{and} \ -\frac{r^4}{2} \Delta \leq \xi \cdot \xi \leq 0$$and $$\xi \cdot \xi=0 \iff \xi =0.$$
\end{lemma}

From this we can prove that if $H$ does not lie on a wall, then the quotients in a Jordan-H\"{o}lder filtration all have the same numerical invariants.
\begin{lemma}{\label{walls}}Given a Chern character ${\bf v}=(r, \nu, \Delta) \in K(Y)$ with $r \geq 2$, choose an ample divisor $H$ not on a wall of type $(r, \Delta)$. Then for any $\mu_H$-semistable sheaf $\VV$ of Chern character ${\bf v}$ and a subsheaf $\VV' \subset \VV$ of rank $r'$, $0<r'<r$, we have  $$\mu_H(\VV')=\mu_H(\VV) \iff \nu(\VV') = \nu(\VV).$$
\end{lemma}
\begin{proof} Suppose $\VV' \subset \VV$ with  $\mu_H(\VV')=\mu_H(\VV)$, but $\nu(\VV') \neq \nu(\VV)$, or equivalently ${\xi:=r c_1(\VV')-r'c_1(\VV) \neq 0}$. By Lemma \ref{Huybrechts-Lehn}, we get that $\xi$ satisfies $-\frac{r^4}{2} \Delta \leq \xi \cdot \xi < 0$. Since $\xi \cdot H=0$, we obtain that $H$ lies on a wall $W_{\xi}$ of type $(r,\Delta)$, contradicting our choice of $H$. 
\end{proof}
\begin{corollary}{\label{JH under generic polarization}} Given a Chern character ${\bf v}=(r, \nu, \Delta) \in K(X)$ with $r \geq 2$, choose an ample divisor $H$ not on a wall of type $(r, \Delta)$. Then for any  $H$-semistable sheaf $\VV$ of Chern character ${\bf v}$ its Jordan-H\"{o}lder factors $\gr_i(\cc{V})$ satisfy
$$\nu(\gr_i(\VV))=\nu \quad \text{and} \quad \Delta(\gr_i(\VV))=\Delta.$$
\end{corollary}
\begin{proof} By the definition of a Jordan-H\"{o}lder filtration $0 \subset F_1 \subset F_2 \subset ... \subset F_l=\VV$, we have ${\mu_H(F_i)=\mu_H(\VV)}$.  Then by Lemma $\ref{walls}$ we get $\nu(F_i)=\nu$. We apply the "seesaw" property of the total slope to the short exact sequence $$0 \to F_{i-1} \to F_i \to \gr_i(\VV) \to 0$$to get $\nu(\gr_i(\VV))=\nu$. The statement about the discriminants then follows from the equality of reduced $H$-Hilbert polynomials $p_{H,  \gr_i(\VV)}=p_{H, \VV}$ and Riemann-Roch. 
\end{proof}

Now, let $(Y,H)$ be a polarized rational surface with $K_Y \cdot H<0$. Dr\'{e}zet \cite{drezet:hal-01175951} calls a point in $M_H(\bv)$ a \emph{type 2} point if the corresponding $S$-equivalence class $$[\VV_1 \oplus ... \oplus \VV_k]$$satisfies $$\nu_i\neq \nu_j \ \text{for some} \ 1 \leq i,j \leq k.$$The other points are called \emph{type 1} points.  Dr\'{e}zet shows in \cite[Theorem C]{drezet:hal-01175951}  that the moduli space $M_{H}({\bf v})$ is \emph{not locally factorial} at \emph{type 2} points.

Suppose further that  $Y$ is a rational surface other than $\PP^2$.  Then there is a morphism $Y \to \bb{P}^1$ such that the generic fiber is $\PP^1$. Let $F$ be the class of a fiber. Yoshioka {\cite{Yos96}} shows that if  $(K_Y +F)\cdot H<0$ and $\Delta(\bv)>\frac{1}{2}$, then  $M_H({\bf v})$ is locally factorial at points of type 1.
In light of Corollary \ref{JH under generic polarization}, we 
conclude that under these assumptions $M_H({\bf v})$ is locally factorial whenever $H$ is not on a wall of type $(r(\bv), \Delta(\bv))$.

\subsection{The Donaldson homomorphism.}\label{Don homo} The Donaldson homomorphism will be our main tool for constructing line bundles on the moduli space. We briefly recall the construction while referring the reader to \cite[\S 8.1]{HL10} and \cite[\S 18.2]{LP97} for full details. 

Let $\cc{U}/S$ be a flat family of semistable sheaves of Chern character $\bv$ on a smooth variety $Y$ parameterized by a variety $S$, and let $p: S \times Y \to S$ and $q:S \times Y \to Y$ be the two projections. The {\it Donaldson homomorphism} $\lambda_{\cc{U}}: K(Y) \to \Pic (S)$ is described as the composition
$$K(Y) \stackrel{q*}{\longrightarrow} K^0(S \times Y) \stackrel{\cdot [\cc{U}]}{\longrightarrow}K^0(S\times Y) \stackrel{p_{!}}{\longrightarrow}K^0(S) \stackrel{\det}{\longrightarrow} \Pic (S).$$
Functorial properties of $\lambda_{\cc{U}}$ are summarized in the following lemma.

\begin{lemma}[{\cite[Lemma 8.1.2.]{HL10}} and {\cite[Lemma 18.2.1]{LP97}}] {\label{computation}} Let $\lambda_{\cc{U}}: K(Y) \to \Pic(S)$ be the Donaldson homomorphism constructed above. 
\begin{enumerate}
\item If $\cc{U}$ is an $S$-flat family and $f: S' \to S$ a morphism, then for any ${\bf u}\in K(Y)$ one has ${\lambda_{f_{Y}^*\cc{U}} ({\bf u})=f^*\lambda_{\cc{U}}({\bf u})}$.
\item If $S$ is equipped with an action of an algebraic group $G$ and $\cc{U}$ is a $G$-linearized family over $S$, then $\lambda_{\cc{U}}$ factors through the group $\Pic^G(S)$ of isomorphism classes of $G$-linearized line bundles on $S$.
\item If $0 \to \cc{U}' \to \cc{U} \to \cc{U}'' \to 0$ is a short exact sequence of $S$-flat families of $G$-linearized coherent, sheaves then $\lambda_{\cc{U}}({\bf u})=
 \lambda_{\cc{U}'}({\bf u})\otimes \lambda_{\cc{U}''}({\bf u})$ in $\Pic^G(S)$.
\end{enumerate}
\end{lemma}

Using the last property we can construct line bundles on the moduli space of (semi)stable sheaves $M_{H}({\bf v})$. Informally, realize $M_H({\bf v})$ as a (good) quotient $R\sslash G$ of a subvariety $R$ of a Quot scheme. The $G$-linearized universal family of quotient sheaves $\cc{U}/R$ gives a map $\lambda_{\cc{U}}: K(X) \to \Pic^G(R)$ and we want to descend the $G$-linearized line bundles in the image along the quotient map $R \to R \sslash G=M_H({\bf v})$. 

For this construction to work we, however, need to restrict the domain of $\lambda_{\cc{U}}$. Denote by ${\bf v}^{\perp} \subset K(Y)$ the complement of ${\bf v}$ with respect to the Euler pairing $\chi(\_\cdot\_)$. 
We then get the following theorem, which shows that the above construction always produces line bundles on the stable locus $M^s_H({\bf v})$ and is compatible with the universal property of the moduli space  $M^s_H({\bf v})$.

\begin{theorem} \cite[Theorem 8.1.5]{HL10} Let ${\bf v}$ be a class in $K(Y)$. Then there exists a group homomorphism $\lambda^s: {\bf v}^{\perp} \to \Pic (M^s_H({\bf v}))$ with the following property:  

If $\cc{U}$ is a flat $G$-linearized family of stable sheaves of class ${\bf v}$ parameterized by a $G$-scheme $S$, and if the classifying morphism $\phi_{\cc{U}}: S \to M^s_H({\bf v})$ is $G$-equivariant, then the following diagram commutes : 
\begin{center}\begin{tikzcd}
{\bf v}^{\perp} \arrow[r, "\lambda^s"] \arrow[d, hook] & \Pic(M^s_H({\bf v})) \arrow[d, "\phi_{\cc{U}}^*"] \\
K(Y) \arrow[r, "\lambda_{\cc{U}}"]          & \Pic^G(S).                  
\end{tikzcd} \end{center}
\end{theorem}
In general, for a polarized variety $(Y,H)$ one needs to further restrict the domain of the Donaldson homomorphism in order to obtain line bundles on the full locus $M_H({\bf v})$ (see the rest of \cite[Theorem 8.1.5]{HL10}). However, when $Y$ is a surface the analysis of the proof of \cite[Theorem 8.1.5]{HL10} shows that for a polarization which does not lie on a wall of type $(r({\bf v}), \Delta({\bf v}))$ we do not need to further shrink the domain.

\begin{proposition}{\label{Donaldson}} Let $Y$ be a smooth projective surface. Let ${\bf v}=(r,\nu, \Delta)$ be a class in $K(Y)$ and let $H$ be an ample divisor not lying on a wall of type $(r, \Delta)$. Then there exists a group homomorphism $\lambda: {\bf v}^{\perp} \to \Pic (M_H({\bf v}))$ with the following property:  

If $\cc{U}$ is a flat $G$-linearized family of $H$-semistable sheaves of class ${\bf v}$ parameterized by a $G$-scheme $S$, and if the classifying morphism $\phi_{\cc{U}}: S \to M_H({\bf v})$ is $G$-equivariant, then the following diagram commutes: \begin{center}
\begin{tikzcd}
{\bf v}^{\perp} \arrow[r, "\lambda"] \arrow[d, hook] & \Pic(M_H({\bf v})) \arrow[d, "\phi_{\cc{U}}^*"] \\
K(X) \arrow[r, "\lambda_{\cc{U}}"]          & \Pic^G(S).                  \end{tikzcd} 
\end{center}  
\end{proposition} 
\begin{proof} We follow the notation used in the proof of  \cite[Theorem 8.1.5]{HL10}. Let $R \stackrel{\pi}{\surj} M_H({\bf v})$ be the quotient morphism   in the GIT construction of the moduli space, where $R$ is a subvariety of the Quot scheme with a universal family of quotients $\cc{U}$. For ${\bf u} \in {\bf v}^{\perp}$, we would like to descend a $GL(V)$-linearized line bundle $L=\lambda_{\cc{U}}({\bf u})$ to $M({\bf v})$ along the quotient map $\pi$. According to the Descent Lemma \cite[Theorem 4.2.15]{HL10}, we need to make sure that for any point $[q:\cc{H} \surj F] \in R$ in a closed $GL(V)$-orbit the stabilizer $GL(V)_{[q]}$ acts trivially on the fiber $L|_{[q]}$ of $L$ over the point $[q]$. 

The orbit of $[q:\cc{H} \surj F] \in R$ is closed if and only if $F$ is a polystable sheaf. Thus $$F \cong \oplus_i (F_i \otimes W_i)$$ with distinct stable Jordan-H\"{o}lder factors $F_i$ and vector spaces $W_i$. The stabilizer of $[q]$ is then isomorphic to $Aut(F) \cong \prod_i GL(W_i)$, and an element $(A_1, ..., A_l) \in \prod_i GL(W_i)$ acts on the fiber $L|_{[q]}$ via multiplication by 
\begin{equation}{\label{exponents}} \prod_i \det(A_i)^{\chi({\bf u} \cdot [F_i])}.\end{equation}

Let $\bv_i=[F_i]$ and $r_i=r(F_i)$. According to Corollary \ref{JH under generic polarization}, for $H$ not on a wall of type $(r,\Delta)$ we get that $\bv_i=\frac{r_i}{r} {\bf v}$ for all $i$, and therefore, the exponents in \eqref{exponents} all vanish: $$\chi({\bf u} \cdot \bv_i)=\chi({\bf u} \cdot \frac{r_i}{r} {\bf v})=0 \ \text{for} \ {\bf u} \in {\bf v}^{\perp}.$$ It follows that $GL(V)_{[q]}$ acts trivially on $L|_{[q]}$. \end{proof}

When $Y$ is a rational surface other than $\PP^2$, Yoshioka {\cite{Yos96}}  analyzes the equivariant Picard group of a subvariety of the Quot scheme parameterizing a certain family of $\OO(0,1)$-prioritary sheaves (see the next subsection for a review of prioritary sheaves) and proves the following result as a consequence of this analysis. 
\begin{theorem}[{\cite[Corollary 3.4]{Yos96}}]{\label{Yoshioka surjectivity}} Let $Y$ be a rational surface other than $\PP^2$ and let ${{\bf v}=(r,\nu, \Delta)\in K(Y)}$ be a Chern character with $\Delta>1/2$. If $H$ is a generic polarization with $(K_{Y}+2F)\cdot H<0$ and if $M^{s}_{H}({\bf v})$ is not empty, then the Donaldson homomorphism gives a surjection $$\lambda: {\bf v}^{\perp} \surj \Pic (M_H({\bf v})).$$
\end{theorem}
Note that  Proposition \ref{Donaldson}  ensures that for a generic polarization $H$ the Donaldson homomorphism is defined as a map $\lambda: {\bf v}^{\perp} \to \Pic (M_H({\bf v}))$.
Therefore, under these assumptions we have a bound on the Picard number of $M_H(\bv)$ and the computation of $\Pic(M_{H}({\bf v}))$ boils down to the computation of the kernel of the Donaldson homomorphism.

\subsection{Prioritary sheaves}\label{Walter's stuff}
It is often the case that the sheaves in a complete family obtained by considering various resolutions and extensions enjoy an extra cohomological property which, in particular, makes the analysis of the locus of semistable sheaves in the family much more tractable.

\begin{definition} Let $\cc{L}$ be a line bundle on a projective surface $Y$. A torsion-free sheaf $\cc{V}$ on $Y$ is called $\cc{L}$-\emph{prioritary} if $$\Ext^2(\VV, \VV \otimes \cc{L}^{\dual})=0.$$ \end{definition}

Let $D$ be an effective Cartier divisor on a projective surface $Y$. Denote the stack of torsion-free sheaves on $Y$ and the stack of coherent sheaves on $D$ with fixed numerical invariants by $TF_Y(r, c_1, c_2)$ and $Coh_D(r, c_1 \cdot D)$ respectively. The next result shows that the restriction of $\cc{O}(D)$-prioritary sheaves from $Y$ to $D$ behaves nicely in families.

\begin{lemma}\cite[Lemma 4.]{Walter1993IrreducibilityOM}\label{Walter} If $\cc{V}$ is an $\OO(D)$-prioritary sheaf, then the restriction map $$TF_Y(r, c_1, c_2) \to Coh_D(r, c_1 \cdot D)$$is smooth (and therefore open) in a neighborhood of $[\VV]$.\end{lemma}

\subsection{The quadric surface}\label{quadric} We specialize some of the above discussion to the case $X=\XX$. 

The surface $X$ comes with two natural projections to the $\PP^1$ factors. Let $F$ denote the class $[pr_1^*(pt)]$ and $E$ denote the class $[pr_2^*(pt)]$. The Picard group of $X$ and the intersection pairing is then given by $$\Pic(X)=\ZZ E \oplus \ZZ F, \ \ E^2=F^2=0, \ \ E \cdot F =1.$$The canonical class of $X$ is $$K_X=-2E-2F.$$A divisor  class $H=aE+bF$ is ample if and only if $a,b>0$. For $m \in \bb{Q}$, we consider the $\bb{Q}$-divisor class $$H_m=E+mF.$$Note that every ample divisor on $X$ is an integer multiple of some $H_m$ with $m>0$.

For character ${\bf v}=(r, \nu, \Delta)=(r, \varepsilon E+\varphi F, \Delta)$ on $X=\XX$, the Riemann-Roch Theorem gives $$\chi({\bf v})=r(P(\nu)-\Delta),$$where $$P(\nu)=(\varepsilon+1)(\varphi+1).$$ Given two sheaves $\VV, \cc{W}$, let $\ext^i(\VV, \cc{W})$ denote $\dim \Ext^i(\VV, \cc{W})$. The Riemann-Roch Theorem says that $$\chi(\VV, \cc{W})=\sum_{i=0}^{2} (-1)^i \ext^i(\VV, \cc{W})=r(\VV)r(\cc{W})( P(\nu(\cc{W})-\nu(\VV))-\Delta(\VV)-\Delta(\cc{W}) ).$$

Note that on $X$ with an ample divisor $H$  every $H$-semistable  sheaf $\cc{V}$ of character ${\bf v}$ is both $\OO(1,1)$- and $\OO(0,1)$-prioritary:
\begin{equation*} \begin{aligned}
\Ext^2(\cc{V}, \cc{V}(-1,-1))&=\Hom(\cc{V}, \cc{V}(-1,-1))^{\dual}=0, \ \\
 \Ext^2(\cc{V}, \cc{V}(0,-1))&=\Hom(\cc{V}, \cc{V}(-2,-1))^{\dual}=0 \end{aligned} \end{equation*}by Serre duality and semistability. Thus, if we denote the stack of $\cc{L}$-prioritary sheaves by $\cc{P}_{\cc{L}}(\bv)$, then we have a chain of open substacks $$\cc{M}_{H}({\bf v}) \subset \cc{P}_{\OO(1,1)}({\bf v}) \subset \cc{P}_{\OO(0,1)}({\bf v}).$$

Walter's Theorem \cite{Walter1993IrreducibilityOM} asserts that the stack $\cc{P}_{\OO(0,1)}(\bv)$ is irreducible and smooth (if nonempty). This implies that the moduli space $M_{H}({\bf v})$ is irreducible as well. Furthermore, if $r({\bf v})\geq 2$, then the general member of $\cc{P}_{\OO(0,1)}(\bv)$ is locally free. Additionally, Walter shows that $M_H(\bv)$ is unirational. 

We also have the following useful result of Yoshioka.

\begin{theorem}[{\cite[Proposition 5.1]{Yoshioka1995}}]\label{Maiorana} Let $\bv=(r,\nu,\Delta)=(r, c_1, \chi) \in K(X)$ be a primitive $H$-semistable Chern character. 

If the polarization  $H$  satisfies $$\gcd (r, c_1 \cdot H,\chi)=1,$$ then $\Pic(M_H(\bv))$ is torsion-free.
\end{theorem}

 In this paper, we will be concerned with the calculation of the Picard group of the moduli space  $M_{H_m}({\bf v})$  of $H_m$-semistable sheaves on $X$ when $m \in \bb{Q}$ is sufficiently close to $1$:$$H_m=E+mF, \ m=1+\epsilon, \ 0<|\epsilon| \ll 1.$$ The reason for doing so is twofold. On the one hand, as we explained above  the genericity  assumption makes $M_{H_m}({\bf v})$ into a locally factorial variety with a known bound on the Picard number. On the other hand,  in the next subsection we recall that when $H_m$ is close to the anticanonical class, there is a complete classification of Chern characters ${\bf v}$ for which the moduli space $M_{H_m}({\bf v})$ is nonempty or positive-dimensional.

\subsection{Exceptional bundles and existence of semistable sheaves}\label{existence}Let $X=\XX$ polarized by an ample divisor $H$. The question of when $M_{H}({\bf v})$ is nonempty was studied by Rudakov in \cite{Rudakov94} and Coskun and Huizenga in \cite{coskun2019existence} (where they studied the existence question for all Hirzebruch surfaces). We follow \cite{coskun2019existence} in this subsection.
\begin{definition} A sheaf $\cc{V}$ on  $X$ is 
\begin{enumerate}
\item \emph{simple}, if $\Hom(\VV, \VV)=\bb{C}$;
\item \emph{rigid}, if $\Ext^1(\VV, \VV)=0$;
\item \emph{exceptional}, if it is simple, rigid, and $\Ext^2(\VV, \VV)=0$;
\item \emph{semiexceptional}, if it is a direct sum of copies of an exceptional sheaf.
\end{enumerate}

We call a character ${\bf v}\in K(X)$ of positive rank \emph{potentially exceptional} if $\chi({\bf v, v})=1$, and \emph{(semi)exceptional} if there is a (semi)exceptional torsion-free sheaf of character ${\bf v}$. We also say that character ${\bf v}$ is $H$-\emph{(semi)stable} (resp. $\mu_H$-\emph{(semi)stable}) if there is an $H$-(semi)stable (resp., $\mu_H$-(semi)stable) sheaf of character ${\bf v}$.
\end{definition}

Any exceptional torsion-free sheaf is locally free and $\mu_{-K_X}$-stable by \cite{Gorodentsev_1989} and therefore, remains $\mu_{H_m}$-stable for $m\in \bb{Q}$ sufficiently close to $1$ by the openness of slope stability in the polarization. We reproduce a part of \cite[Lemma 6.7]{coskun2019existence} that further characterizes (potentially) exceptional bundles and characters.

\begin{lemma}[{\cite[Lemma 6.7]{coskun2019existence}}]\label{exceptional stuff} Let ${\bf v} \in K(X)$ be a potentially exceptional character of rank $r$. 
\begin{enumerate}
\item The rank of $\bv$ is odd and the discriminant of ${\bf v}$ is $$\Delta=\frac{1}{2}-\frac{1}{2r^2}.$$
\item The character ${\bf v}$ is primitive.
 
\item If $m$ is generic and $\VV$ is an $H_m$-semistable sheaf of discriminant $\Delta(\VV)<\frac{1}{2}$, then $\VV$ is semiexceptional.
\end{enumerate}
\end{lemma}

Heuristically, $\mu_H$-stable exceptional bundles give strong bounds on the possible numerical invariants of $\mu_H$-semistable sheaves. In particular, if $E$ is a $\mu_H$-stable exceptional bundle and $\VV$ is a $\mu_H$-semistable sheaf with $$\frac{1}{2}K_X \cdot H \leq \mu_H(\VV)-\mu_H(E)<0,$$then $$\Hom(E, \VV)=0 \quad \text{and} \quad \Ext^2(E, \VV)=\Hom(\VV, E(K_X))^{\dual}=0$$by semistability and Serre duality. Therefore, $\chi(E,\VV)\leq0$.  By the Riemann-Roch Formula, this inequality can be viewed as a lower bound on $\Delta(\VV)$: $$\Delta(\VV)\geq P(\nu(\VV)-\nu(E))-\Delta(E).$$Likewise, if instead $$0<\mu_H(\VV)-\mu_H(E)\leq -\frac{1}{2}K_X\cdot H,$$then the inequality $\chi(\VV, E)\leq0$ provides a lower bound $$\Delta(\VV)\geq P(\nu(E)-\nu(\VV))-\Delta(E)$$ on $\Delta(\VV)$. 

This motivates the following definition.

\begin{definition}[{\cite[Definition 6.13]{coskun2019existence}}]{\label{DLP less than r}} For a $\mu_H$-stable exceptional bundle $E$,  define a function  $$\DLP_{H,E}(\nu) =
\left\{
        \begin{array}{ll}
                P(\nu-\nu(E))-\Delta(E)  &  \text{if} \ \frac{1}{2}K_X \cdot H \leq (\nu-\nu(E))\cdot H <0 \\
                P(\nu(E)-\nu)-\Delta(E)  &  \text{if} \ 0<  (\nu-\nu(E))\cdot H \leq -\frac{1}{2}K_X \cdot H \\
\max\{P(\pm(\nu-\nu(E)))-\Delta(E)\} & \text{if} \ (\nu-\nu(E))\cdot H=0
        \end{array}
\right.$$on the strip of slopes $\nu=\varepsilon E + \varphi F=(\varepsilon, \varphi)\in \bb{Q}^2$ satisfying $$|(\nu-\nu(E))\cdot H|\leq -\frac{1}{2}K_X\cdot H.$$
Let $\bb{E}_H$ be the set of $\mu_H$-stable exceptional bundles on $X$. Further define a function $$\DLP_H^{<r}(\nu)=\sup_{\substack{E\in\bb{E}_H \\ |(\nu-\nu(E))\cdot H| \leq -\frac{1}{2}K_X\cdot H \\ r(E)<r}} \DLP_{H,E}(\nu),$$where this time $\nu=(\varepsilon, \varphi)$ could be any point in $\bb{Q}^2$. We refer to the above functions as the \emph{Dr\'{e}zet-Le-Potier functions}, or $\DLP$-\emph{functions}, for short.
\end{definition}
One can see the graph of $\Delta=\DLP_{H_m}^{<r}(\varepsilon E + \varphi F)$ in the $(\varepsilon, \varphi, \Delta)$-space in Figure \ref{fig:example} below (for $m=1$). In the rest of the paper we will call the graph of $\Delta=\DLP_{H_m}^{<r}(\varepsilon E + \varphi F)$   the \emph{Dr\'{e}zet-Le-Potier surface}, the \emph{$\DLP^{<r}_{H_m}$- surface}, or the \emph{$\DLP^{<r}$-surface} for short. 

\begin{figure}[h]%
    \centering
    \subfloat[Top-down view. Reproduced from {\cite[Figure 5]{coskun2019existence}}.]{{\includegraphics[width=7.4cm]{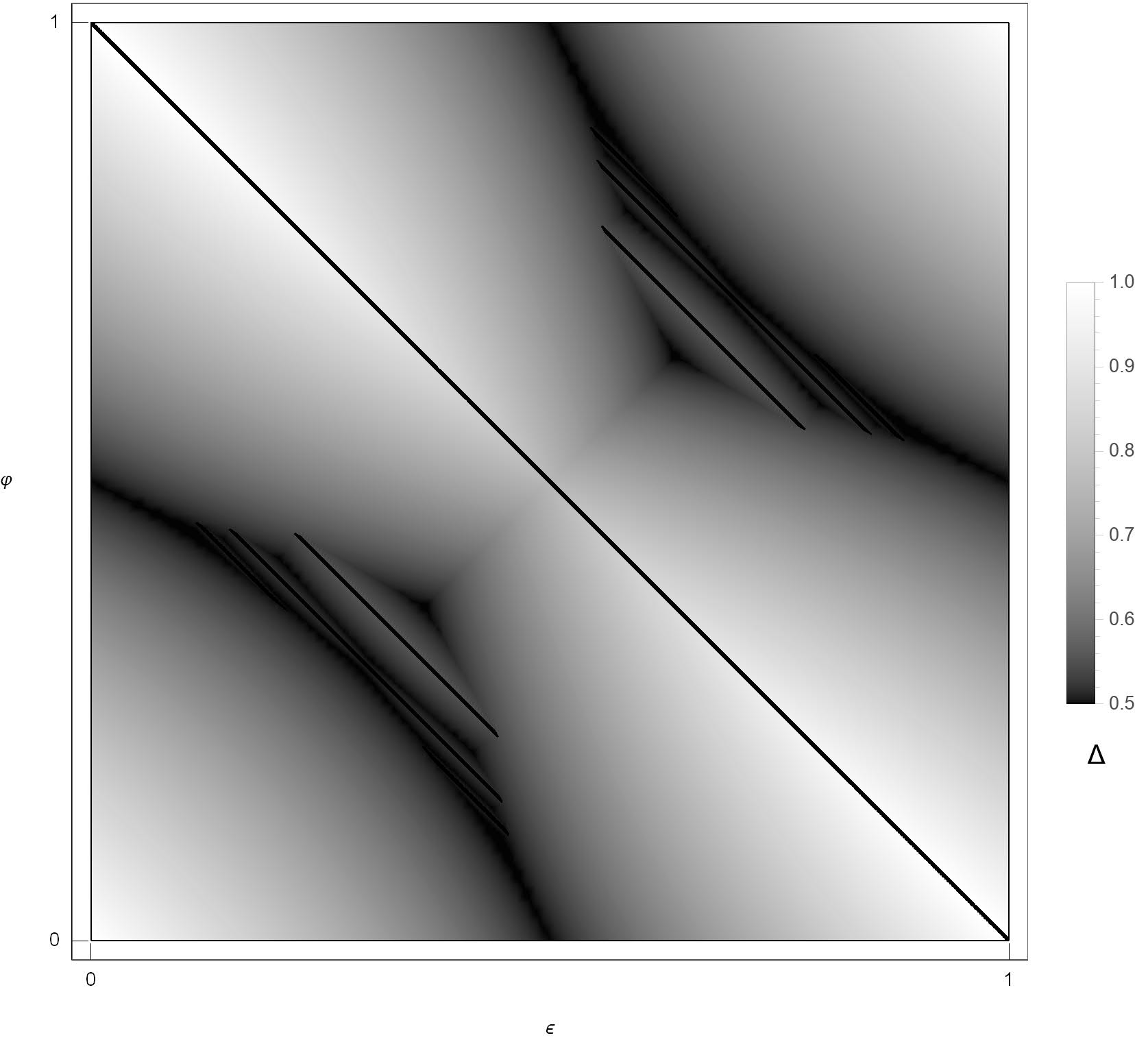} }}%
    \qquad
    \subfloat[View from the side]{{\includegraphics[height=6.6cm, width=7.4cm]{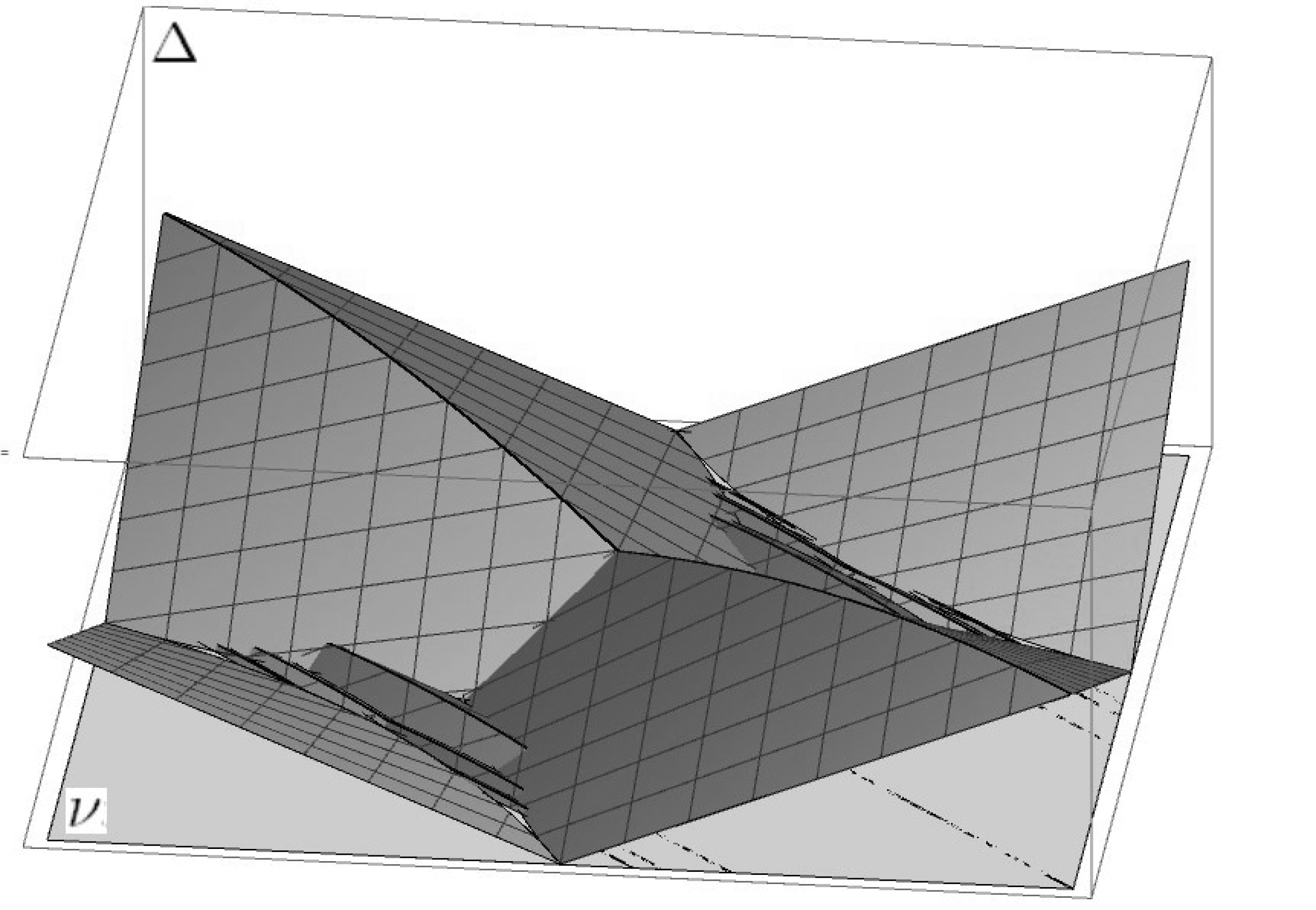} }}%
    \caption{The $\DLP^{<r}$-surface: graph of $\Delta=\DLP_{H_1}^{<8}(\varepsilon E+ \varphi F ).$}%
    \label{fig:example}%
\end{figure}
The discussion before Definition \ref{DLP less than r} shows that if there is a $\mu_H$-semistable sheaf of total slope $\nu$ and discriminant $\Delta$, then $\Delta\geq \DLP_{H}^{<r}(\nu)$. The next theorem shows that for a generic polarization close to the anticanonical class  such inequalities also provide sufficient conditions for the existence of Gieseker semistable sheaves of total slope $\nu$ and discriminant $\Delta$.  
\begin{theorem}[{\cite[Corollary 7.6]{coskun2019existence}}]{\label{existence theorem}} Let ${\bf v}=(r, \nu, \Delta)\in K(X)$ be a character with positive rank. Let $\epsilon \in \bb{Q}$ be sufficiently small (depending on $r$), $0 < |\epsilon| \ll 1$, and set $m=1+\epsilon$.
\enumerate
\item If ${\bf v}$ is potentially exceptional, then it is exceptional if and only if $$\Delta \geq \DLP_{H_m}^{<r}(\nu).$$

\item If ${\bf v}$ is not semiexceptional, there is an $H_m$-semistable sheaf of character ${\bf v}$ if and only if $$\Delta \geq \DLP_{H_m}^{<r}(\nu).$$ \end{theorem}

One can also easily tell when ${\bf v}$ is semiexceptional. Write ${\bf v}=N{\bf v'}$ with $N\in \bb{N}$ and a primitive character ${\bf v'}$  of rank $r'=r/N$. Then using (1) from the above theorem, we see that ${\bf v}$ is semiexceptional if and only if ${\bf v'}$ is potentially exceptional and $\Delta \geq \DLP_{H_m}^{<r'}(\nu)$. 

Taken together, these statements provide a finite inductive computational procedure for determining whether the moduli space $M_{H_m}({\bf v})$ is nonempty for a given character ${\bf v}$. Let us  remark that for a sufficiently small $\epsilon$ we actually have $\DLP^{<r}_{H_m}(\nu)=\DLP^{<r}_{H_1}(\nu)$ by {\cite[Lemma 7.8]{coskun2019existence}}, so one can keep using Figure \ref{fig:example} to gain insight into $\DLP^{<r}_{H_m}(\nu)$ for $m$ close to $1$.

Also note that since we are concerned with calculating the Picard group of the moduli space, we will only be interested in those characters ${\bf v}$ for which the moduli space $M_{H_m}({\bf v})$ is nonempty and positive dimensional. Recall (\cite[\S 4.5]{HL10}) that the expected dimension of the moduli space is given by$$\exp\dim M_{H_m}({\bf v})=r^2(2\Delta-1)+1.$$ This shows that the expected dimension is positive if and only if $\Delta\geq 1/2$. Lemma \ref{exceptional stuff} (1) implies that such characters are not semiexceptional.

Next, we introduce  useful terminology describing the position of character ${\bf v}$ relative to the $\DLP^{<r}$-surface. 

\begin{definition}\label{associated exceptional} Let ${\bf v}=(r, \nu, \Delta)\in K(X)$ be an $H_m$-semistable character with $\Delta \geq \frac{1}{2}$ and ${\Delta=\DLP^{<r}_{H_m}(\nu)}$, where $m=1+\epsilon$ and $\epsilon \in \bb{Q}$ is a  sufficiently small number depending on $r$, $0 < |\epsilon| \ll 1$.  

We say that an exceptional bundle $E$ is \emph{associated} to ${\bf v}$ if 
\begin{equation}
\begin{aligned} 
r(E)&<r,\\
|(\nu-\nu(E))\cdot H_m | & \leq -\frac{1}{2}K_X\cdot H_m, \quad \text{and} \\
 \ \Delta=\DLP^{<r}(\nu)&=\DLP_{H_m, E}(\nu).
\end{aligned}
\end{equation}

Character ${\bf v}=(r,\nu,\Delta)$ can be positioned in three different ways relative to the $\DLP^{<r}$-surface (see Figure \ref{fig:example2}):
\begin{enumerate}
\item If $\Delta>\DLP^{<r}_{H_m}(\nu)$, then we will say that character ${\bf v}$ lies \emph{above} the $\DLP^{<r}$-surface,
\item If $\Delta=\DLP^{<r}_{H_m}(\nu)$ and there is a single exceptional bundle $E$ associated to ${\bf v}$, then we will say that character ${\bf v}$ lies \emph{on a single branch} of the $\DLP^{<r}$-surface given by the exceptional bundle $E$, 
\item If $\Delta=\DLP^{<r}_{H_m}(\nu)$ and there are at least two different exceptional bundles $E_1,E_2$ associated to ${\bf v}$, then we will say that character ${\bf v}$ lies \emph{on the intersection of  branches} of the $\DLP^{<r}$-surface given by exceptional bundles $E_1$ and $E_2$.
\end{enumerate}
\end{definition}

\begin{figure}[H]%
    \centering
    \subfloat[$\bv$ above $\DLP^{<r}$-surface]{{\includegraphics[height=4.6cm, width=5.1cm]{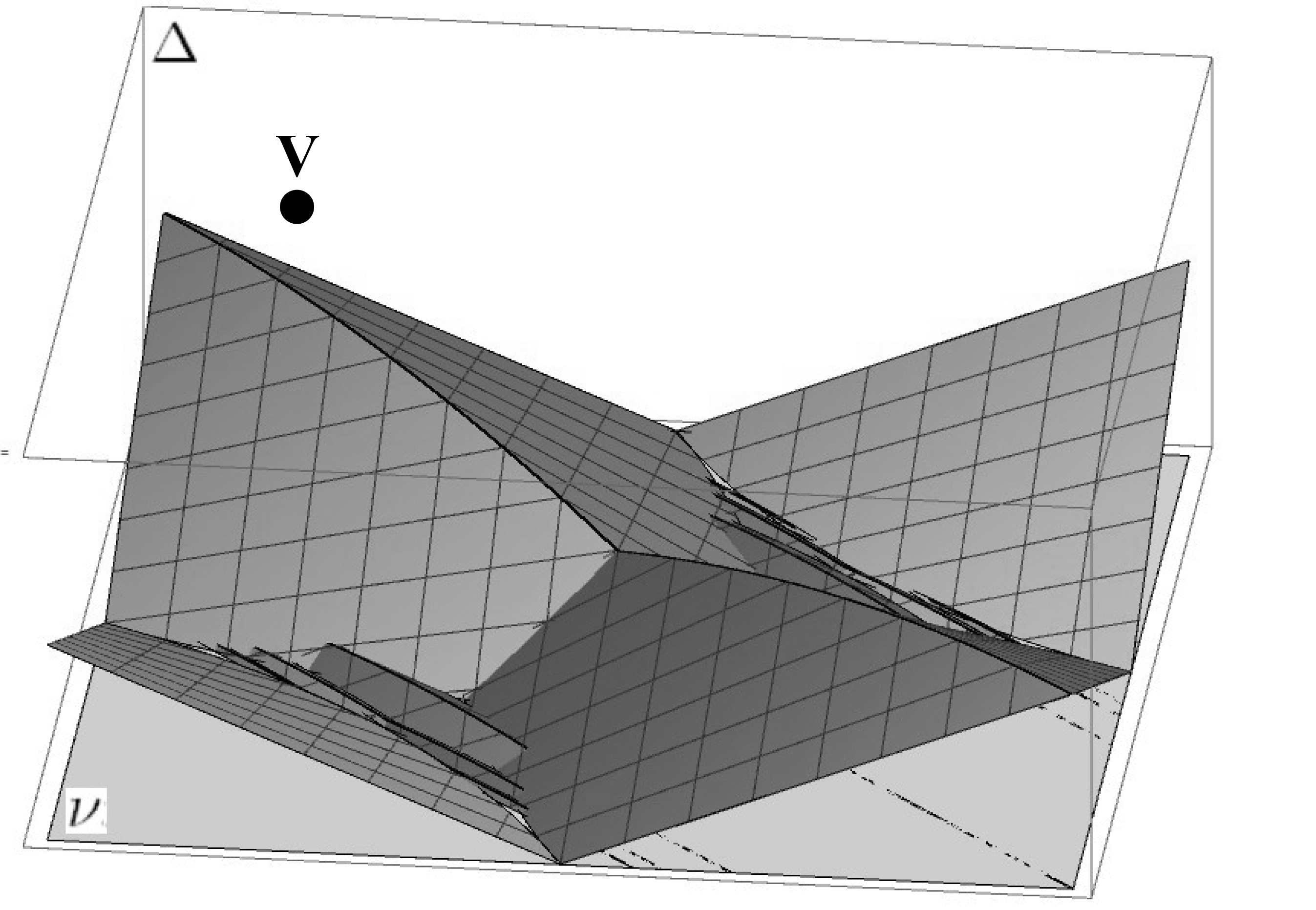} }}%
    \quad
    \subfloat[$\bv$ on a single branch of $\DLP^{<r}$-surface]{{\includegraphics[height=4.6cm, width=5.1cm]{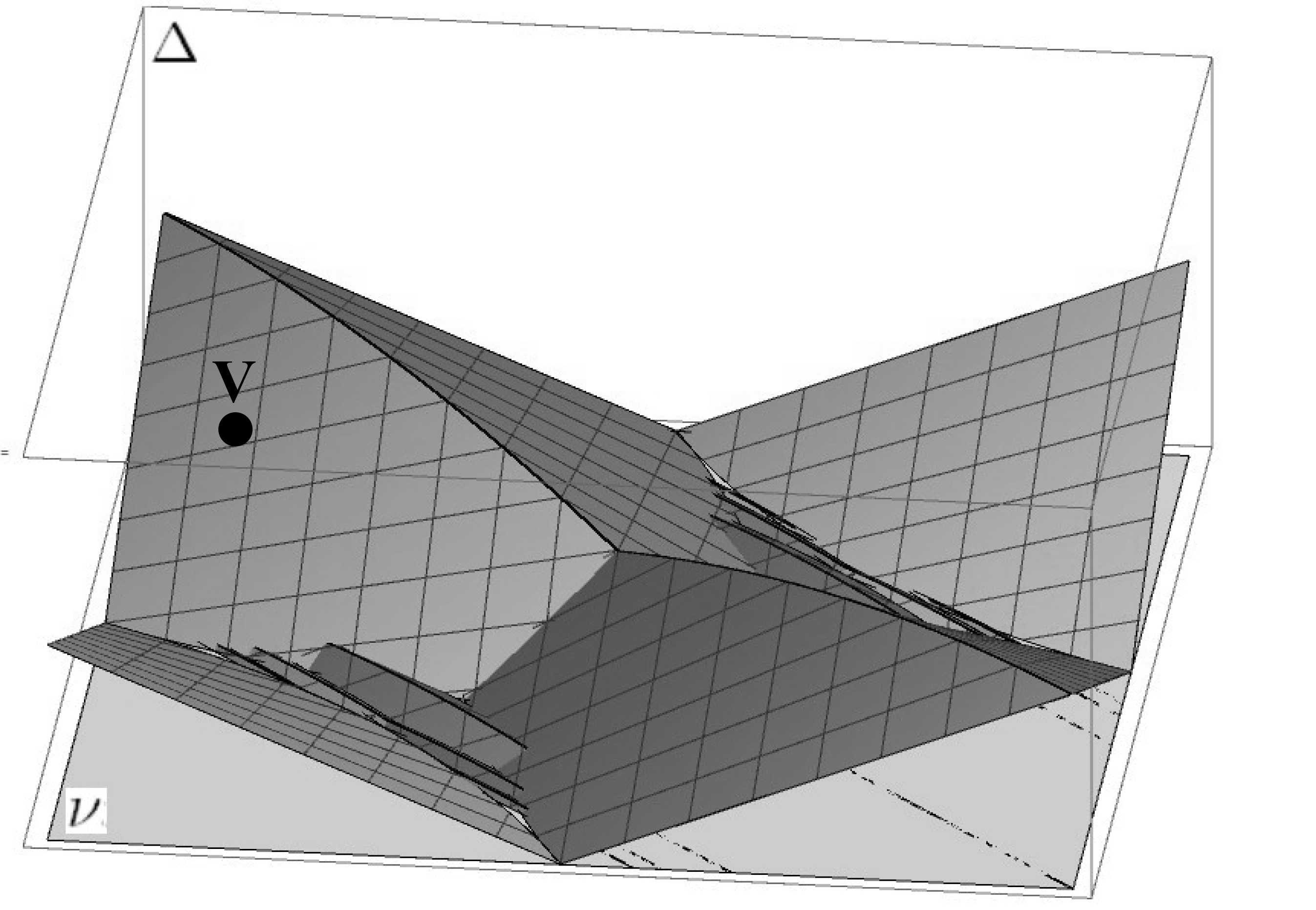} }}%
    \quad
     \subfloat[$\bv$ on the intersection of two branches of $\DLP^{<r}$-surface]{{\includegraphics[height=4.6cm, width=5.1cm]{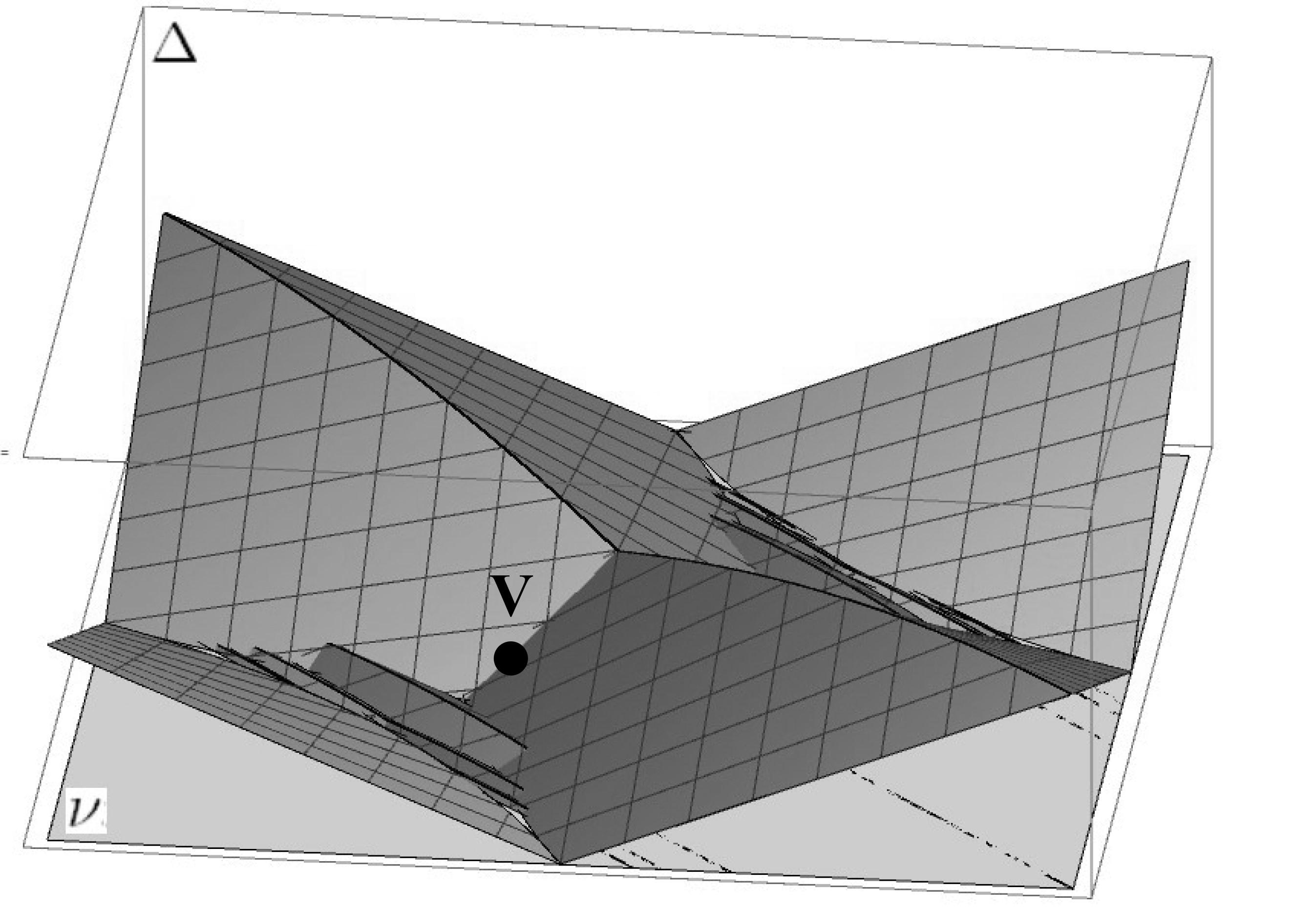} }}%
    \caption{Position of $\bv$ relative to the $\DLP^{<r}$-surface.}%
    \label{fig:example2}%
\end{figure}

We finish this subsection with a useful result about the existence of  stable bundles. 

\begin{proposition}[{\cite[Propositions 9.5 \& 9.6]{coskun2019existence}}]{\label{changing polarization}} Suppose $m\in \bb{Q}$ is generic and $\bv=(r,\nu,\Delta)$ is an integral Chern character.\begin{enumerate} \item If $\Delta>\frac{1}{2}$ and there are $H_m$-semistable sheaves of character ${\bf v}$, then  there are $H_m$-stable sheaves of of character ${\bf v}$.
\item If there are $H_m$-stable sheaves of of character ${\bf v}$, then there are $\mu_{H_m}$-stable sheaves of character ${\bf v}$. 

\end{enumerate}
    
\end{proposition}

\subsection{Gaeta-type resolutions}{\label{Gaeta section}} These resolutions are special resolutions of sheaves on $X$ by direct sums of line bundles. Their advantage is that they are simple enough to work with and provide unirational parameterizations of moduli spaces of sheaves. Gaeta-type resolutions were studied in  \cite{coskun_huizenga_2018} for all Hirzebruch surfaces $\bb{F}_e$, but we will only need the case  $X=\bb{F}_0=\XX$.

\begin{definition}{\label{Gaeta}} Let $L$ be a line bundle on $X$. An $L$-{\it Gaeta-type resolution} of a sheaf $\VV$ on $X$ is a resolution of $\VV$ of the form
\begin{equation}{\label{L Gaeta}}0 \to L(-1,-1)^{\alpha} \to L(-1,0)^{\beta} \oplus L(0,-1)^{\gamma} \oplus L^{\delta} \to \VV \to 0\end{equation}where $\alpha, \beta, \gamma, \delta$ are nonnegative integers. We say a sheaf $\VV$ has a {\it Gaeta-type resolution} if it admits an $L$-Gaeta-type resolution for some line bundle $L$. 
\end{definition}

The results of \cite[\S 4]{coskun_huizenga_2018} we will need are summarized in the following statement.

\begin{theorem}\label{Gaeta theorem} If ${\bf v}$ is a $\mu_H$-semistable Chern character with $\Delta({\bf v})\geq \frac{1}{4}$, then a general $\mu_H$-semistable sheaf $\VV$ admits a Gaeta-type resolution.
More specifically: 
\begin{enumerate} 
\item Suppose $L$ is a line bundle such that the inequalities 
\begin{equation}
 \label{Gaeta equation}
 \begin{aligned}
  \alpha=-\chi({\bf v}\otimes L^{\dual}(-1,-1)) \geq 0 \\
  \beta=-\chi({\bf v}\otimes L^{\dual}(-1,0)) \geq 0 \\
  \gamma=-\chi({\bf v}\otimes L^{\dual}(0,-1)) \geq 0 \\
  \delta=\chi({\bf v}\otimes L^{\dual}) \geq 0   
  \end{aligned} 
\end{equation}
are satisfied. Then not all of the integers in \eqref{Gaeta equation} are zero and  a general  $\mu_H$-semistable sheaf $\VV$ admits an $L$-Gaeta-type resolution with integers in \eqref{Gaeta equation} giving the exponents in  \eqref{L Gaeta}. 
\item A line bundle satisfying inequalities \eqref{Gaeta equation} always exists.
\end{enumerate}
\end{theorem}

Gaeta-type resolutions allow us to build complete families of $\cc{O}(1,1)$-prioritary sheaves.

\begin{proposition}{\label{Gaeta parameterizations}} Let $\alpha, \beta,\gamma,\delta$ be nonnegative integers satisfying $$r:=\beta+\gamma+\delta-\alpha >0.$$ For a line bundle $L$ consider the open subset$$U \subset \mathbb{H}:=\Hom \left(L(-1,-1)^{\alpha},  L(-1,0)^{\beta} \oplus L(0,-1)^{\gamma} \oplus L^{\delta} \right)$$ parameterizing injective sheaf maps with torsion-free cokernel.  For $\psi_u \in U$, let $\VV_{u}$ be the cokernel:
$$0 \to  L(-1,-1)^{\alpha} \xrightarrow{\psi_u} L(-1,0)^{\beta} \oplus L(0,-1)^{\gamma} \oplus L^{\delta} \to \VV_{u} \to 0.$$

If $r\geq 2$, then $U$ is nonempty, $\codim_{\mathbb{H}} \ (\mathbb{H}\setminus U) \geq 2$, and the family $\VV_{u}/U$ is a complete family of $\OO(1,1)$-prioritary sheaves.

\end{proposition}
\begin{proof} Only the statement about the codimension requires a proof as the other statements are  proved in  \cite[Theorem 2.10]{coskun_huizenga_2018}.

The statement about the codimension follows from the standard analysis of the incidence correspondence $$\Sigma:=\{(p, \psi) \ | \ \psi|_p \ \text{is not injective}\} \subset X \times \Hom \left(L(-1,-1)^{\alpha},  L(-1,0)^{\beta} \oplus L(0,-1)^{\gamma} \oplus L^{\delta} \right)$$using the fact that $$\sheafhom \left(L(-1,-1)^{\alpha},  L(-1,0)^{\beta} \oplus L(0,-1)^{\gamma} \oplus L^{\delta} \right)$$is a globally generated vector bundle. See \cite[pp. 238-239]{LP97} and the proof of \cite[Theorem 4.7]{DL85} for details.
\end{proof}
We finish this subsection by introducing the "dual version" of a Gaeta-type resolution. Specifically, this is a resolution of the form \begin{equation}\label{dual Gaeta}0 \to \VV  \to L(1,0)^{\alpha} \oplus L(0,1)^{\beta} \oplus L^{\gamma} \to L(1,1)^{\delta} \to 0. \end{equation} 
We have the following analogue of Proposition \ref{Gaeta parameterizations}.

\begin{proposition}\label{dual Gaeta family}Let $\alpha, \beta,\gamma,\delta$ be nonnegative integers satisfying $$r:=\alpha+\beta+\gamma-\delta >0.$$ For a line bundle $L$ consider the open subset$$U \subset \mathbb{H}:=\Hom \left(L(1,0)^{\alpha} \oplus L(0,1)^{\beta} \oplus L^{\gamma} ,  L(1,1)^{\delta} \right)$$ parameterizing surjective sheaf maps. For $\psi_u \in U$, let $\VV_{u}$ be the kernel:
$$0 \to  \VV_u \to  L(1,0)^{\alpha} \oplus L(0,1)^{\beta} \oplus L^{\gamma} \xrightarrow{\psi_u} L(1,1)^{\delta}\to 0.$$

 If $r\geq 2$, then $U$ is nonempty, $\codim_{\mathbb{H}} \ (\mathbb{H}\setminus U) \geq 2$, and the family $\VV_{u}/U$ is a complete family of $\OO(1,1)$-prioritary vector bundles.
\end{proposition}

\section{Study of the Shatz stratification}\label{Shatz section}

In this section we use the techniques from \cite[\S 1 \& 3]{DL85}, \cite[Chapter 15]{LP97} and \cite[\S 2.A]{HL10}  to detect strata of unstable sheaves of codimension one in complete families. 

\subsection{Generalities on the Shatz stratifiction}\label{spectral sequences} Given a complete family $\VV_t/T$ of torsion-free sheaves of character ${\bf v}$, we denote by $S_H(P_1,P_2,...,P_l) \subset T$ the Shatz stratum parameterizing sheaves $\VV_t$ with $H$-Harder-Narasimhan filtration having quotients with $H$-Hilbert polynomials $P_1,P_2,...,P_l$. If one further assumes that $T$ is smooth and for each $t\in T$ we have $\Ext^2(\VV_t, \VV_t)=0$, then the Shatz stratum $S_H(P_1, P_2,...,P_l)$ is a smooth locally closed subvariety of $T$ with the normal space at point $t \in S_H(P_1,P_2,..,P_l)$ given by $\Ext^1_+ (\VV_t, \VV_t)$. We refer the reader to   \cite[\S 1]{DL85} for the definition of $\Ext^1_+, \Ext^1_-$ and the general deformation theory of filtered sheaves. We instead review here the computational aspects.

For $t\in S_H(P_1,P_2,...,P_l)$ equip $\VV_t$ with its $H$-Harder-Narasimhan filtration with quotients $\gr_{1,t},...,\gr_{l,t}$.    Then there is a spectral sequence with $E_1$-term given by $$E_1^{p,q} = \begin{cases} \bigoplus_i \Ext^{p+q}(\gr_{i,t},\gr_{i-p,t}) &\textrm{if $p<0$}\\ 0 & \textrm{if $p\geq 0$}\end{cases}$$ which abuts on $\Ext_+^{p+q}(\VV_t,\VV_t)$ in degree $p+q$.  Similarly, there is a spectral sequence with $E_1$-term given by $$E_1^{p,q} = \begin{cases} 0 &\textrm{if $p<0$}\\ \bigoplus_i \Ext^{p+q}(\gr_{i,t},\gr_{i-p,t}) & \textrm{if $p\geq 0$}\end{cases}$$ which abuts on $\Ext_-^{p+q}(\VV_t,\VV_t)$ in degree $p+q$. 

For our purposes, it would be convenient to work with a slightly refined notion of a Shatz stratum. Note that since $\Pic(X)$ is a \emph{discrete} algebraic group scheme, for points $t$ \emph{within a connected component} of a Shatz stratum $S_H(P_1,P_2,...,P_l)$ the $H$-Harder-Narasimhan quotients $\gr_{1,t},\gr_{2,t},...,\gr_{l,t}$ of $\VV_t$ not only have the same $H$-Hilbert polynomials $P_1,P_2,...,P_l$, but also the same \emph{numerical invariants} ${\bf v}_1, {\bf v}_2,...,{\bf v}_l$.  Thus, $S_H(P_1,P_2,...,P_l)$ breaks up into a disjoint union of strata $S_H({\bf v}_1, \bv_2,...,\bv_l)$, where each $S_H({\bf v}_1, \bv_2,...,\bv_l)$ parameterizes sheaves $\VV_t$ with $H$-Harder-Narasimhan filtration having quotients with numerical invariants ${\bf v}_1, \bv_2, ..., \bv_l$. Later, when we use the notion of Shatz stratum we will have $S_H({\bf v}_1, \bv_2,...,\bv_l)$ in mind instead of $S_H(P_1,P_2,...,P_l)$. The  discussion above applies to $S_H({\bf v}_1, \bv_2, ..., \bv_l)$ equally well, and we conclude that when $\VV_t/T$ is a smooth complete family of torsion-free sheaves\ satisfying $\Ext^2(\VV_t, \VV_t)=0$ for each $t\in T$ the stratum $S=S_H({\bf v}_1, \bv_2,...,\bv_l)$ is a smooth locally closed subvariety of $T$ with  the normal space at point $t$ described as $$N_{S/T}|_t \cong\Ext^1_+(\VV_t,\VV_t).$$


\subsection{$\bf{\Delta_i=\frac{1}{2}}$-strata}Before we proceed with the estimates, let us introduce one more definition.  
Careful reading of \cite[Lemma 18.3.1]{LP97}, \cite[Proposition 2.4]{Dr88} and \cite[Lemma 4.8]{DL85} suggests that in the  $\PP^2$ case the codimension one Shatz strata occur in complete families of $\OO_{\PP^2}(1)$-prioritary sheaves  only  for characters ${\bf v}$ on the $\DLP_{\PP^2}$-curve and correspond to sheaves whose first or last Harder-Narasimhan quotient is semiexceptional. The next definition is created ad hoc to capture \emph{new} codimension one Shatz strata which did not exist in the $\PP^2$ case, but which  appear  in the $\XX$ case due to the presence of semistable Chern characters of discriminant $\frac{1}{2}$.
\begin{definition}{\label{ad hoc}}Let ${\bf v}=(r,\nu,\Delta)$ be an integral Chern character on $X=\XX$. Let $\VV_t/T$ be a complete family of torsion-free sheaves parameterized by a smooth variety $T$ with ${\bf v}(\VV_t)={\bf v}$.  We call Shatz stratum $S\subset T$ a $\mathit{\Delta_i=\frac{1}{2}}$-\emph{stratum} if $S$ parameterizes sheaves $\VV_t$ with the $H$-Harder-Narasimhan filtration of length $l=2$, $$S=S_H({\bf v}_1,\bv_2),$$ such that the numerical invariants $${\bf v}_1=(r_1, \nu_1, \Delta_1), {\bf v}_2=(r_2, \nu_2, \Delta_2)$$ of the $H$-Harder-Narasimhan quotients of $\VV_t$ satisfy  the following properties:
\begin{enumerate}

\item $\Delta_1, \Delta_2 \geq \frac{1}{2}$ with at least one $\Delta_i=\frac{1}{2}, \ i=1,2$,

\item $\nu_2-\nu_1=\frac{k}{r_1r_2}E-\frac{k}{r_1 r_2}F$ for some integer $k$ with $ 0<|k|\leq r_1r_2$,

\item $\chi({\bf v}_1,{\bf v}_2)=-1$.
\end{enumerate} 
\end{definition}

\subsection{Codimension of Shatz strata}
In this subsection, we present a study of Shatz stratification through a numerical analysis involving Riemann-Roch computations.

For the rest of this subsection we adopt the following convention. Consider a family $\VV_t/T$ of sheaves parameterized by a variety $T$. Suppose $\VV_t$ belongs to a Shatz stratum $S_{H_m}(\bv_1, \bv_2,...,\bv_l)\subset T$ with $H_m$-Harder-Narasimhan quotients $\gr_{1,t}, \gr_{2,t}, ... ,\gr_{l,t}$ having numerical invariants $\bv_1, \bv_2,...,\bv_l$. To improve readability we drop the subscript $t$ in $\gr_{i,t}$ if any confusion is unlikely. We further write $${\bf v}_i={\bf v}(\gr_i)=(r_i, \nu_i,\Delta_i).$$

We start with a couple of preparatory lemmas.
\begin{lemma}\label{preparatory}Let ${\bf v}=(r, \nu, \Delta)\in K(X)$ be a character with $r \geq 2$.  Let $\epsilon \in \bb{Q}$ be a sufficiently small number (depending on $r$), $0 < |\epsilon| \ll 1$, and set $m=1+\epsilon$. 

Consider a complete family $\VV_t/T$ of $\OO(1,1)$-prioritary sheaves with ${\bf v}(\VV_t)={\bf v}$ parameterized by a smooth variety $T$. If $\VV_t$ belongs to the Shatz stratum $S=S_{H_m}(\bv_1, \bv_2,...,\bv_l)$ and the  inequalities \begin{equation}\label{ineq}\mu_{\max, H}(\VV_t) - \mu_{\min, H}(\VV_t) \leq 2 \end{equation} are satisfied for $H=\OO(1,1), \OO(1,0)$ and $\OO(0,1)$, then $$|(\nu-\nu_i)\cdot H_m| \leq -\frac{1}{2}K_X \cdot H_m $$ and  $$\codim_{T,t}(S)=-\sum_{i<j} \chi(\bv_i,\bv_j).$$
\end{lemma}

\begin{proof}
By the above inequalities \eqref{ineq}, for any subsheaf $\cc{W}\subset \VV_t$ and any quotient $\VV_t \surj \cc{E}$ of  $\VV_t$ the difference of the total slopes $$\nu(\cc{W})-\nu(\cc{E})$$  lies in a bounded region of the $(\varepsilon, \varphi)$-plane of total slopes. Furthermore,  inequality \eqref{ineq} for ${H=\OO(1,1)}$ implies $$|(\nu - \nu_i)\cdot H_1|<2=-\frac{1}{2}K_X \cdot H_1.$$ It follows that since $m$ is close enough to $1$,  we have   \begin{equation}{\label{slope1}}|(\nu-\nu_i)\cdot H_m| \leq -\frac{1}{2}K_X \cdot H_m. \end{equation} 

Let  $\gr_{1},\gr_2,...,\gr_{l}$ be the quotients in the $H_m$-Harder-Narasimhan filtration of $\VV_t $.  
Since $$\codim_{T,t}(S)=\dim N_{S/T}|_t = \ext^1_+(\VV_t,\VV_t),$$we will use the spectral sequences for $\Ext^{\bullet}_+(\VV_t, \VV_t)$ and $\Ext^{\bullet}_{-}(\VV_t, \VV_t)$ from \S \ref{spectral sequences} to compute $\ext^1_+(\VV_t,\VV_t)$. 

Since $\Hom(\gr_{i}, \gr_{j})=0$ for $i<j$ by semistability, we see that $\Ext^0_+(\VV_t,\VV_t)=0.$ Likewise, by our bound \eqref{slope1} and semistability we have $$\Ext^2(\gr_{i}, \gr_{j})\cong \Hom(\gr_{j}, \gr_{i} \otimes K_X)^{\dual}=0 \quad \text{for any} \ \ i,j,$$so both $\Ext^2_+(\VV_t, \VV_t)=0$ and $\Ext^2_-(\VV_t,\VV_t)=0$. Therefore, the only nonzero terms in the spectral sequence for $\Ext^{p+q}_+(\VV_t, \VV_t)$ have $p+q=1$. We conclude $$  \ext^1_+(\VV_t, \VV_t)=\sum_{i<j}\ext^1(\gr_{i}, \gr_{j})=-\sum_{i<j} \chi(\gr_{i},\gr_{j})$$ with $\chi(\gr_{i}, \gr_{j})\leq 0$ for $i<j$. \end{proof}

\begin{lemma}{\label{slope distance}} Let $\VV_t/T$ be a complete family of $\OO(1,1)$-prioritary sheaves parameterized by a smooth variety $T$. Let $H$ be one of line bundles $\OO(1,1), \OO(1,0)$ or $\OO(0,1)$. Then the set of points $t\in T$ such that $$\mu_{\max,H}(\VV_t) - \mu_{\min,H}(\VV_t)>2$$
is a closed subset of codimension at least $2$ in $T$. 
\end{lemma}

\begin{proof}To show the result for $H=\OO(1,1)$, one follows the proof of \cite[Corollary 15.4.4.]{LP97}, replacing a line $d$ on $\mathbb{P}^2$ with a rational curve from the complete linear series $|\OO(1,1)|$ on $X=\XX$ and using Lemma \ref{Walter} together with  the $\OO(1,1)$-prioritariness of sheaves in the family.

For $H=\OO(1,0)$ or $\OO(0,1)$, we recall that $\OO(1,1)$-prioritariness implies both $\OO(1,0)$- and $\OO(0,1)$-prioritariness (see \cite[Lemma 3.1]{coskun2019existence}). The same argument as above applies in this case too.
 \end{proof}
 
The next two propositions describe codimension one Shatz strata in complete families of $\OO(1,1)$-prioritary sheaves.

\begin{proposition}{\label{Shatz}}{\label{no divisorial Shatz strata}}
 Let ${\bf v}=(r, \nu, \Delta)$ be a Chern character satisfying $$\Delta>\DLP^{<r}_{H_m}(\nu),$$where $m=1+\epsilon$ and $\epsilon \in \bb{Q}$ is a sufficiently small number (depending on $r$), $0< |\epsilon| \ll1.$

Consider  a complete family $\VV_t/T$ of $\OO(1,1)$-prioritary sheaves with ${\bf v}(\VV_t)={\bf v}$ parameterized by a smooth variety $T$. Then the $H_m$-Shatz strata of codimension $1$ in this family are given by (nonempty) $\Delta_i=\frac{1}{2}$-strata.
\end{proposition}

\begin{proof} Since the proof is rather long, we split it into several steps for the reader's convenience.

\emph{Step 1.} We start by making some preliminary reductions. By Lemma \ref{slope distance} we can pass to an open subset of points $t \in T$ where  \begin{equation}\label{esti} \mu_{\max,H}(\VV_t) - \mu_{\min,H}(\VV_t)\leq2,  \end{equation} for $H=\OO(1,1), \OO(1,0)$ and $\OO(0,1)$. 

Suppose $S:=S_{H_m}(\bv_1,\bv_2,...,\bv_l) \subset T$ is a nonempty Shatz stratum of codimension $1$ in $T$. By Lemma \ref{preparatory} for $t \in S$ we have \begin{equation}\label{esti10} \begin{aligned} |(\nu-\nu_i)\cdot H_m| \leq -\frac{1}{2}K_X \cdot H_m \ \text{and} \\ \codim_{T,t}(S)=-\sum_{i<j} \chi(\bv_i,\bv_j)=1.\end{aligned} \end{equation} This implies that if $\gr_{1},\gr_{2},...,\gr_{l}$ are the $H_m$-Harder-Narasimhan quotients of $\VV_t$, then   we  have $$\chi(\gr_{1}, \gr_{l})=0 \ \text{or} \ \chi(\gr_{1}, \gr_{l})=-1.$$

Below, we analyze various possibilities for what the numerical invariants $\bv_1, \bv_2, ..., \bv_l$ of the $H_m$-Harder-Narasimhan quotients $\gr_{1}, \gr_{2}, ..., \gr_{l}$ of $\VV_t$  could be and show that they must necessarily satisfy conditions (1)-(3) of Definition \ref{ad hoc}, i.e. $S$ must be a $\Delta_i=\frac{1}{2}$-stratum.

{\it Step 2.} Suppose that $\chi(\gr_1,\gr_l)=0$ holds. By the Riemann-Roch Theorem, this gives $$\Delta_1+\Delta_l=P(\nu_l-\nu_1).$$ Since $m$ is sufficiently close to $1$, we have that $P(\nu_l-\nu_1)\leq 1$ with equality holding only when $\nu_1=\nu_l$. 

If $\nu_1 \neq \nu_l$, we get that $$\Delta_1+\Delta_l=P(\nu_l-\nu_1)<1,$$ and therefore $\Delta_1<\frac{1}{2}$ or  $\Delta_l< \frac{1}{2}$. 

If $\nu_1=\nu_l$, then since $\gr_i$ are the quotients in the Harder-Narasimhan filtration, we must have that $\Delta_l >\Delta_1$. Since $P(\nu_l-\nu_1)=1$ in the case, we get $\Delta_1< \frac{1}{2}$. 

In both of these cases, Lemma \ref{exceptional stuff} (3)  implies that $\gr_1$ or $\gr_l$ is semiexceptional and we can follow the argument of \cite[Lemma 4.8]{DL85}. Here we deal with the case where $\gr_1$ is semiexceptional. The argument for when $\gr_l$ is semiexceptional is similar. Write $\gr_{1} \cong E^{ k}$ with exceptional bundle $E$.  Then we get \begin{equation*}{\label{est}}
 \chi({\bf v}_1, {\bf v})=\chi(E^{ k}, \VV_t)=\chi(\gr_1, \VV_t)=\chi(\gr_1,\gr_1)+\sum_{2<j} \chi(\gr_1,\gr_j)\geq \chi(\gr_1, \gr_1)-1 \geq 0. \end{equation*}
On the other hand, for a semistable $\VV$ of character ${\bf v}$ we have $\hom(E, \VV)=\ext^2(E,\VV)=0$, showing that $\chi({\bf v}_1,\bv)=k\cdot\chi(E, \VV)\leq 0$. Thus we have $\chi(E,\VV)=0$. Inequality \eqref{esti10} gives    $$(\nu(E)-\nu)\cdot H_m \leq -\frac{1}{2}K_X \cdot H_m,$$ and we get $\Delta=\DLP_{H_m, E}(\nu)$. This contradicts our assumption that character ${\bf v}$ lies above the $\DLP^{<r}$-surface.

{\it Step 3.} Now we know $\chi(\gr_1,\gr_l)=-1$. If one of $\gr_1,\gr_l$ is semiexceptional, we arrive to a contradiction in the same way as above.  At this point, we have shown that for points $t\in S$ the $H_m$-Harder-Narasimhan quotients of $\VV_t$ satisfy  $\chi(\gr_1,\gr_l)=-1$ and $\Delta_1, \Delta_l \geq \frac{1}{2}$.

Next, we show that $l=2$. Assume on the contrary that we have  $l\geq 3$.  Since $\chi(\gr_i,\gr_j)=0$ for $i<j, (i,j) \neq (1,l)$, we in particular have  $$\chi(\gr_1,\gr_i)=0 \implies P(\nu_i-\nu_1)-\Delta_1-\Delta_i=0 $$ for $1<i<l$. Since $P(\nu_i-\nu_1)\leq1$ and we cannot have $\nu_1=\nu_i, \ \Delta_1=\Delta_i$, we get that $\Delta_i<\frac{1}{2}$ and $\gr_i$ is semiexceptional, $\gr_i \cong E^{k}$. If $\mu_{H_m}(\gr_i) \geq \mu_{H_m}(\VV_t)$, then for a semistable $\VV$ of character ${\bf v}$ we have $$\chi({\bf v}_i, \bv)=\chi(\gr_i,\VV)=k \cdot \chi(E, \VV)\leq0,$$ because $\Delta \geq \DLP_{H_m}^{<r}(\nu)$. On the other hand,$$\chi({\bf v}_i,\bv)=\chi(\gr_i,\VV_t)=\sum_{j=1}^{i-1} \chi(\gr_i,\gr_j)+\chi(\gr_i,\gr_i)+\sum_{j=i+1}^{l}\chi(\gr_i,\gr_j)$$
$$\geq\sum_{j=1}^{i-1} \chi(\gr_j,\gr_i)+\chi(\gr_i,\gr_i)+\sum_{j=i+1}^{l}\chi(\gr_i,\gr_j)=\chi(\gr_i,\gr_i)>0,$$
which is a contradiction. Here we used  that for $j<i$ we have  $\mu_{H_m}(\nu_j)\geq \mu_{H_m}(\nu_i)$ and since $m$ was chosen to be close enough to $1$, this implies $\mu_{H_1}(\nu_j)\geq \mu_{H_1}(\nu_i)$. Therefore $$P(\nu_j-\nu_i)\geq P(\nu_i-\nu_j) \implies \chi(\gr_i,\gr_j) \geq \chi(\gr_j, \gr_i) .$$If $\mu_{H_m}(\gr_i) \leq \mu_{H_m}(\VV_t)$, we get a contradiction by instead comparing $\chi(\VV, \gr_i)$ for a semistable $\VV$ to $\chi(\VV_t, \gr_i)$.
 
\emph{Step 4.} So from now on we use that $l=2$ and $\Delta_1, \Delta_2 \geq \frac{1}{2}$. Expanding $\chi(\gr_1,\gr_2)=-1$ by the Riemann-Roch Theorem we get
\begin{equation}{\label{estimate}}\Delta_1+\Delta_2=P(\nu_2-\nu_1)+\frac{1}{r_1r_2}.\end{equation}
 Below we eliminate various cases for what the values of $\Delta_i$ and $\nu_2-\nu_1$ could be.

{\it Case 1.} Suppose that $\Delta_1, \Delta_2>\frac{1}{2}.$ Here we follow the method in \cite[Proposition 2.4]{Dr88}. In this case, the expected dimension of the moduli spaces $M(\bv_i)$ for $i=1,2$ is $$\exp \dim \ M(\bv_i)=r_i^2(2\Delta_i-1)+1\geq 2, $$which allows us to write \begin{equation}{\label{inequality}}\Delta_i \geq \frac{1}{2}+\frac{1}{2r_i^2} \quad \text{for} \ \ i=1,2.\end{equation} Using this estimate in equation \eqref{estimate}, we get
$$P(\nu_2-\nu_1)+\frac{1}{r_1r_2} \geq 1+\frac{1}{2} \left( \frac{1}{r_1^2}+\frac{1}{r_2^2} \right),$$ which simplifies to 
$$1-P(\nu_2-\nu_1)+\frac{1}{2} \left(\frac{1}{r_1}-\frac{1}{r_2} \right)^2 \leq 0.$$ Since $P(\nu_2-\nu_1)\leq 1$, we in fact have
$$1-P(\nu_2-\nu_1)=\frac{1}{r_1}-\frac{1}{r_2}=0,$$and$$ r_1=r_2,\nu_1=\nu_2.$$Comparing equations \eqref{estimate} and \eqref{inequality}, we get $\Delta_1=\Delta_2$. This is a contradiction because $\gr_1$ and $\gr_2$ are quotients in the Harder-Narasimhan filtration.

Next we rule out the cases where one of $\Delta_1, \Delta_2$ is equal to $\frac{1}{2}$, but condition (2) of Definition \ref{ad hoc} does not hold. We consider the case when $\Delta_1=\frac{1}{2}, \Delta_2 \geq \frac{1}{2}$; the case  $\Delta_1\geq \frac{1}{2},\Delta_2=\frac{1}{2}$ is dealt with similarly.

\emph{Case 2.} Assume $\Delta_1=\frac{1}{2}, \Delta_2\geq \frac{1}{2}, \nu_1=\nu_2$. Since $\Delta_1=\frac{1}{2}$, the rank $r_1$ must be even: $r_1=2\bar{r}_1$. Equation \eqref{estimate} gives \begin{equation}{\label{Delta}}\Delta_2=\frac{1}{2}+\frac{1}{2 \bar{r}_1r_2}=\frac{\bar{r}_1 r_2+1}{2\bar{r}_1 r_2}.\end{equation}

If one of $\bar{r}_1$ or $r_2$ is even, then the right hand side of equation \eqref{Delta} is an irreducible fraction. This means that after cancelling all the common factors in the numerator and the denominator of \begin{equation}{\label{discriminant}}\Delta_2=\frac{c_2(\gr_2)}{r_2}-\frac{r_2-1}{r_2^2} \cdot \frac{c_1(\gr_2)^2}{2}=\frac{c_2(\gr_2)r_2 - (r_2-1) \cdot (c_1(\gr_2)^2/2)}{r_2^2}\end{equation}the resulting denominator should be equal to $2\bar{r_1}r_2=r_1r_2$. This implies that $r_2=kr_1$ and character $(r_2, \nu_1=\nu_2, \frac{1}{2})$ is equal to $k \cdot{\bf v}_1$, so it is an integral Chern character.  We can, therefore, write $$\Delta_2=\frac{\bar{r}_1 r_2+1}{2\bar{r}_1 r_2}=\frac{1}{2}+\frac{N}{r_2}$$for some integer $N$. This gives equalities $$\bar{r}_1r_2+1=\bar{r}_1r_2+2N\bar{r}_1 \iff 2N\bar{r}_1=1,$$ which is impossible.

Now assume both $\bar{r}_1$ and $r_2$ are odd. Then we can cancel $2$ in the numerator and the denominator of \eqref{Delta} and the result will be an irreducible fraction with denominator $\bar{r}_1r_2$. This time, we see from equation \eqref{discriminant} that $\bar{r}_1$ divides $r_2$, so we can write $r_2= \bar{r}_1d$ for an odd integer $d$. Since both $\chi(\gr_1)$ and $\chi(\gr_2)$ are integers, we have$$\chi(\gr_1)=r_1 \cdot \left(P(\nu_1)-\frac{1}{2}\right)=2\bar{r}_1 P(\nu_1)-\bar{r}_1 =2\bar{r}_1 P(\nu_2)-\bar{r}_1  \in \ZZ  \quad (\text{recall}  \ \nu_1=\nu_2),$$
 
$$\chi(\gr_2)=r_2\left(P(\nu_2)-\frac{1}{2}-\frac{1}{r_1r_2}\right)=\bar{r}_1dP(\nu_2)-\frac{\bar{r}_1d}{2}-\frac{1}{2\bar{r}_1}=\frac{\bar{r}_1d(2\bar{r}_1P(\nu_2)-\bar{r}_1 ) - 1}{2\bar{r}_1}=$$
$$=\frac{\bar{r}_1d \cdot \chi(\gr_1)-1}{2\bar{r}_1} \in \ZZ.$$ The last expression implies $\bar{r}_1=1$ and
we get that ${\bf v}_1=(2, \nu_1, \frac{1}{2})$. Now, since $r_2$ is odd, we can write $\nu_1=\nu_2=\frac{*}{\text{odd}}E+\frac{*}{\text{odd}}F$ with both coefficients being irreducible fractions. But explicit analysis of the $\DLP^{<r}_{H_m}$-surface shows that $H_m$-semistable Chern characters with $(r, \nu, \Delta)=(2, \nu, \frac{1}{2})$ can only have $\nu=\frac{2k+1}{2}E+lF$ or $\nu=kE+\frac{2l+1}{2}F$ with $k,l \in \ZZ$. Thus we cannot have $\nu_1=\nu_2$ under these assumptions.

\emph{Case 3.}  We turn our attention to those cases where $\Delta_1=\frac{1}{2}, \Delta_2\geq \frac{1}{2}$ and $\nu_1\neq\nu_2$. We can explicitly write $$\nu_2-\nu_1=\left(\frac{a_2}{r_2}-\frac{a_1}{r_1}\right)E+\left(\frac{b_2}{r_2}-\frac{b_1}{r_1}\right)F=\frac{a}{r_1r_2}E+\frac{b}{r_1r_2}F, \quad a,a_i,b,b_i \in \ZZ,$$  so that \begin{equation*}{\label{Hilber}}P(\nu_2-\nu_1)=\left(1+\frac{a}{r_1r_2}\right)\left(1+\frac{b}{r_1r_2}\right)=1+\frac{a+b}{r_1r_2}+\frac{ab}{(r_1r_2)^2}.\end{equation*}
Further cases depend on the values of $a$ and $b$. Note that by the definition of the Harder-Narasimhan filtration, we cannot have $a>0$ and $b>0$ simultaneously. Moreover, since  $\Delta_2 \geq \frac{1}{2}$, we get an inequality \begin{equation}{\label{estimate33}}1+\frac{a+b+1}{r_1r_2}+\frac{ab}{(r_1r_2)^2}=P(\nu_2-\nu_1)+\frac{1}{r_1r_2}=\Delta_1+\Delta_2=\frac{1}{2}+\Delta_2 \geq 1,\end{equation}which we use to eliminate certain potential values of $a$ and $b$. 

{\it Case 3.1.} Assume $ a<0, b<0.$ Rewrite \eqref{estimate33} as \begin{equation}{\label{estimatetwo}}\frac{ab}{(r_1r_2)^2} \geq \frac{-a-b-1}{r_1r_2} \iff ab \geq (-a-b-1)r_1r_2\iff a(b+r_1r_2)\geq(-b-1)r_1r_2.\end{equation}

If $b+r_1r_2<0$, we get $a \leq\frac{-b-1}{b+r_1r_2}r_1r_2=\frac{-b-1}{-b-r_1r_2}(-r_1r_2)\leq -r_1r_2$, so that $a+r_1r_2 \leq 0$. But then$$\mu_{H_1}(\gr_1)-\mu_{H_1}(\gr_2)=\frac{-a-b}{r_1r_2}>2,$$
contradicting  \eqref{esti}.

If $b+r_1r_2=0$, then, since $r_1r_2 \geq 2$, the last inequality in \eqref{estimatetwo} reads as $$0=a(l+r_1r_2) \geq(-b-1)r_1r_l >0,$$which is a contradiction.

If $1 \leq b+r_1r_2 \leq r_1r_2-1$, the last inequality in \eqref{estimatetwo} reads as $$a\geq\frac{-b-1}{b+r_1r_2}r_1r_2 \geq 0,$$now contradicting the assumption that $a<0$. 

{\it Case 3.2.} Assume $a<0, b>0, |a|>b$ or  $a>0, b<0, |b|>a$.  We estimate the left hand side in \eqref{estimate33}:
$$1+\frac{a+b+1}{r_1r_2}+\frac{ab}{(r_1r_2)^2} \leq 1+\frac{ab}{(r_1r_2)^2}<1,$$ and the necessary condition \eqref{estimate33} does not hold.

{\it Case 3.3.} Assume $a=0, b<0$ or $a<0, b=0$. Since the calculations are symmetric in $a$ and $b$, we only treat $a=0, b<0$.  The left hand side of \eqref{estimate33} now reads as
$$1+\frac{b+1}{r_1r_2}.$$ 

If $b\leq -2$, then again condition \eqref{estimate33} does not hold.

If $b=-1$, we get from \eqref{estimate33} $$\Delta_1+\Delta_2=P(\nu_2-\nu_1)+\frac{1}{r_1r_2}=1+\frac{b+1}{r_1r_2}=1,$$which gives $\Delta_1=\Delta_2=\frac{1}{2}$. This implies that now both $r_1$ and $r_2$ are even, $r_1=2\bar{r}_1, r_2=2 \bar{r}_2$, and  $$\nu_2-\nu_1=\left(\frac{a_2}{r_2}-\frac{a_1}{r_1}\right)E+\left(\frac{b_2}{r_2}-\frac{b_1}{r_1}\right)F=\left(\frac{2\bar{r}_1a_2-2\bar{r}_2a_1}{r_1r_2}\right)E+\left(\frac{2\bar{r}_1b_2-2\bar{r}_2b_1}{r_1r_2}\right)F,$$ which is never equal to $$\frac{a}{r_1r_2}E+\frac{b}{r_1r_2}F=\frac{-1}{r_1r_2}F.$$

\emph{Case 3.4.} At this point observe that we have ruled out all cases for possible values of $a$ and $b$, except for the case $a<0, b>0, |a|=b$ or $a>0, b<0, a=|b|$. 

If $|a|=|b|>r_1 r_2$, then inequality \eqref{estimate33} does not hold. 

Finally, if  $|a|=|b|\leq r_1r_2$, then this case  corresponds precisely to conditions (1)-(3) of Definition \ref{ad hoc}. This shows that the nonempty Shatz stratum $S$ of codimension $1$ in $T$ must be a $\Delta_i=\frac{1}{2}$-stratum.
\end{proof}

The proof of Proposition \ref{Shatz} can be readily modified to give an analogous statement for   characters ${\bf v}=(r,\nu, \Delta)$, which  lie on a single branch of the Dr\'ezet-Le Potier surface. 
\begin{proposition}{\label{Shatz2}}Let ${\bf v}=(r, \nu, \Delta)$ be a Chern satisfying $$\Delta=\DLP^{<r}_{H_m}(\nu)$$with a single exceptional bundle $E$ associated to ${\bf v}$, where $m=1+\epsilon$ and $\epsilon \in \bb{Q}$ is a sufficiently small number (depending on $r$), $0< |\epsilon| \ll1.$ 

Consider  a complete family $\VV_t/T$ of $\OO(1,1)$-prioritary sheaves parameterized by a smooth variety $T$ with ${\bf v}(\VV_t)={\bf v}$. Then $H_m$-Shatz strata of codimension $1$ are given by
\begin{itemize}
\item the stratum parameterizing sheaves $\VV_t$ with $H_m$-Harder-Narasimhan filtration \begin{equation*}\begin{aligned} 0 \subset & \ E \subset \cc{V}_t  &(\text{when} \ \mu_{H_m}(E)\geq \mu_{H_m}({\bf v})), &\text{ or}\\ 0 \subset \cc{F}_1 \subset \VV_t, &\quad \cc{V}_t/ \cc{F}_1 \cong E  &(\text{when} \ \mu_{H_m}(E)\leq \mu_{H_m}({\bf v})), &\end{aligned}\end{equation*} 
\item $\Delta_i=\frac{1}{2}$-strata,
\end{itemize}when these strata are nonempty for the family $\cc{V}_t/T$.
\end{proposition}

\begin{proof} Suppose $S:=S_{H_m}(\bv_1, \bv_2, ...,\bv_l)$ is a Shatz stratum of codimension $1$. Repeating step 1  of the proof of Proposition \ref{Shatz},    we  have for a point  point $t \in S$ \begin{equation}\label{esti11} \begin{aligned} |(\nu-\nu_i)\cdot H_m| \leq -\frac{1}{2}K_X \cdot H_m \ \text{and} \\ \codim_{T,t}(S)=-\sum_{i<j} \chi(\bv_i,\bv_j)=1.\end{aligned} \end{equation} Again, we analyze various possibilities for what the numerical invariants $\bv_1, \bv_2, ..., \bv_l$ of the $H_m$-Harder-Narasimhan quotients $\gr_1,\gr_2,...,\gr_l$ of $\VV_t$  could be. 

First assume that both $\gr_1$ and $\gr_l$ are not semiexceptional. Inspecting step 2 of the proof of Proposition \ref{Shatz} we see that we cannot have $\chi(\gr_1,\gr_l)=0$. We conclude that  $\chi(\gr_1,\gr_l)=-1$,  $\Delta_1,\Delta_l \geq \frac{1}{2}$ and $\chi(\gr_i,\gr_j)=0$ for $i<j, (i,j) \neq (1,l) $ with  $\Delta_1,\Delta_l \geq \frac{1}{2}$. Now note  that in steps 3 and 4  of the proof of Proposition \ref{Shatz}, where we dealt with the same configuration for the numerical invariants, we only used that $\Delta \geq \DLP^{<r}_{H_m}(\nu)$. We conclude that the only type of codimension $1$ Shatz strata which arise in this case are the $\Delta_i=\frac{1}{2}$-strata.

Now, we deal with the case when one of $\gr_1$ or $\gr_l$ is semiexceptional. As before, we show the proof for $\gr_1$ being semiexceptional, with the  case  where $\gr_l$ is semiexceptional being similar. Assume $\gr_1 \cong F^k$, where $F$ is an exceptional bundle.  We get an inequality \begin{equation}{\label{est2}}
k\cdot \chi(F, {\bf v})=\chi(F^k, {\bf v})=\chi(\gr_1, {\bf v})=\chi(\gr_1,\gr_1)+\sum_{j=2}^{l} \chi(\gr_1,\gr_j)\geq \chi(\gr_1,\gr_1)-1 \geq 0. \end{equation} On the other hand, for an $H_m$-semistable $\VV$ of character $\bv$  we have $$\Hom(F, \VV)=\Ext^2(F,\VV)=0$$ by semistability and Serre duality. This gives $\chi(F, {\bf v})\leq 0$. Thus we obtain that all the inequalities in \eqref{est2} must be equalities. We get $\chi(\gr_1,\gr_1)=1$, which is only possible when $k=1$ and $\gr_1 \cong F$. We also have $\chi(F, {\bf v})=0$, which forces $F=E$ since by \eqref{esti11} $$(\nu(F)-\nu)\cdot H_m \leq -\frac{1}{2}K_X \cdot H_m,$$and we are assuming that ${\bf v}$ has a \emph{single} associated exceptional bundle $E$. We also remark that in this case $\gr_l$ cannot be semiexceptional, because then again ${\bf v}$ would have a second associated exceptional bundle different from $E$. Therefore, $\Delta_2\geq \frac{1}{2}$.

It remains to show that $l=2$. Assume that $l>2$. If $\chi(\gr_1,\gr_l)=-1$, then we follow the argument in Step 3 of the proof of Proposition \ref{Shatz} with minor modifications. Specifically, for $(i,j) \neq (1,l), \ i<j$, we have $$\chi(\gr_i, \gr_j)=0.$$ Taking $j=l$ and using Riemann-Roch, we get$$\chi(\gr_i,\gr_l)=P(\nu_l-\nu_i)-\Delta_i-\Delta_l=0.$$ We cannot have $P(\nu_l-\nu_i)=1$ and $\Delta_l=\frac{1}{2}$ because then $\nu_i=\nu_l, \ \Delta_i=\Delta_j$ and this contradicts the fact that $\bv_i, \bv_l$ are quotients in the Harder-Narasimhan filtration. It follows that $P(\nu_l-\nu_i)<1$ or $\Delta_l > \frac{1}{2}$. In both cases we get that $\Delta_i<\frac{1}{2}$ and $\gr_i$ is semiexceptional, $\gr_i \cong (F'')^k$. If $\mu_{H_m}(\gr_i) \geq \mu_{H_m}(\VV_t)$, then for a semistable $\VV$ of character ${\bf v}$ we have $$\chi({\bf v}_i, \bv)=\chi(\gr_i,\VV)=k \cdot \chi(F'', \VV)\leq0$$ because $\Delta \geq \DLP^{<r}_{H_m}(\nu)$. On the other hand,  
$$\chi({\bf v}_i,\bv)=\chi(\gr_i,\VV_t)=\sum_{j=1}^{i-1} \chi(\gr_i,\gr_j)+\chi(\gr_i,\gr_i)+\sum_{j=i+1}^{l}\chi(\gr_i,\gr_j)$$
$$\geq\sum_{j=1}^{i-1} \chi(\gr_j,\gr_i)+\chi(\gr_i,\gr_i)+\sum_{j=i+1}^{l}\chi(\gr_i,\gr_j)=\chi(\gr_i,\gr_i)>0,$$
where the inequality holds for the same reason as in Step 3 of Proposition \ref{Shatz}. This is a contradiction. If $\mu_{H_m}(\gr_i) <\mu_{H_m}(\VV_t)$, we get instead a contradiction by comparing $\chi(\VV,\gr_i)$ for a semistable $\VV$ to $\chi(\VV_t, \gr_i)$.

Finally, we eliminate $l>2$ under the condition $\chi(\gr_1,\gr_l)=0$. If we can find $i\neq 1,l$ with ${\chi(\gr_i,\gr_l)=0}$, we arrive to a contradiction in the same way as in the previous paragraph. If not, then $l=3$ and the only nonzero $\chi(\gr_i,\gr_j)$ with $i<j$ is $\chi(\gr_2,\gr_3)=-1$. But then $$\chi({\bf v}_1,\bv)=\chi(\gr_1, \VV_t)=\chi({\bf v}_1, \bv_1+\bv_2+\bv_3)=\chi({\bf v}_1,\bv_1)=\chi(\gr_1,\gr_1)=1>0,$$contradicting the fact that ${\bf v}$ is on the branch of the $\DLP$-surface given by $\gr_1 \cong E$: for a semistable $\VV$ of character ${\bf v}$ we have $$\chi({\bf v}_1,\bv)=\chi(E, \VV) = 0.$$
\end{proof}

The last two propositions motivate the following definition.
\begin{definition}\label{bad definition} Let $\bv=(r,\nu,\Delta)\in K(X)$ be an $H_m$-semistable Chern character, where $m=1+\epsilon$ for a sufficiently small (depending on $r$) number $\epsilon \in \bb{Q}, \ 0<|\epsilon|\ll 1$. 

We call  ${\bf v}$ a \emph{bad} character  if we can find a decomposition $${\bf v}=\bv_1+\bv_2,$$where $\bv_1, \bv_2$ are $H_m$-semistable Chern characters satisfying
\begin{enumerate}
\item $p_{H_m,\bv_1} > p_{H_m,\bv_2}$, where $p_{H_m,\bv_i}$ is the reduced $H_m$-Hilbert polynomial of $\bv_i$,

\item $\Delta_1, \Delta_2 \geq \frac{1}{2}$ with at least one $\Delta_i=\frac{1}{2}, \ i=1,2$,

\item $\nu_2-\nu_1=\frac{k}{r_1r_2}E-\frac{k}{r_1 r_2}F, \ \text{for some integer} \ k \ \text{with} \ 0<|k|\leq r_1r_2$,
\item  $\chi(\bv_1, \bv_2)=-1.$
\end{enumerate}
Otherwise, we call an $H_m$-semistable Chern character ${\bf v}$ a \emph{good} character.
\end{definition}

\begin{remark}\label{primitive} Note that a bad character ${\bf v}$ is always primitive. Indeed, if $\Delta_1=\frac{1}{2}$, then $$\chi({\bf v}_1, \bv)=\chi(\bv_1, \bv_1+\bv_2)=\chi({\bf v}_1, \bv_2)=-1.$$ If instead $\Delta_2=\frac{1}{2}$, then $\chi({\bf v}, \bv_2)=-1$.  
\end{remark}

The point of this notion is that by Definition \ref{ad hoc} and Propositions \ref{Shatz}, \ref{Shatz2} for good $H_m$-semistable characters $\Delta_i=\frac{1}{2}$-strata do not appear in smooth complete families of $\OO(1,1)$-prioritary sheaves. This way, for good characters the study of the Shatz stratification yields results that are similar  to the $\PP^2$ case. On the other hand, when ${\bf v}$ is a bad Chern character,  we get a potentially nonempty divisorial $\Delta_i=\frac{1}{2}$-stratum $S_{H_m}({\bf v}_1, {\bf v}_2)$  in smooth complete families of $\OO(1,1)$-prioritary sheaves for every decomposition ${\bf v}={\bf v}_1 + {\bf v}_2$ as in Definition \ref{bad definition}.
 
To demonstrate this phenomenon we give an example of a bad Chern character ${\bf v}$ and a smooth complete family of $\OO(1,1)$-prioritary sheaves of character ${\bf v}$ for which a $\Delta_i=\frac{1}{2}$-stratum is nonempty.

\begin{example}\label{bad character} Consider the character ${\bf v}=(r,\nu,\Delta)=(4,-\frac{1}{4}E-\frac{1}{4}F, \frac{9}{16})$. We have $$\Delta({\bf v})=\frac{9}{16}=\DLP^{<4}_{ H_m}\left(-\frac{1}{4}E-\frac{1}{4}F\right)=\DLP_{H_m,\OO}\left(-\frac{1}{4}E-\frac{1}{4}F\right) \ \text{for} \ m=1+\varepsilon, \ 0<\varepsilon \ll 1,$$so that ${\bf v}$ is $H_m$-semistable and the line bundle $\OO$ is associated to ${\bf v}$. One checks that conditions of Definition \ref{bad definition} are met for the decomposition $${\bf v}={\bf v}_1+{\bf v}_2, \  \ {\bf v}_1=(2, -\frac{1}{2}F, \frac{1}{2}), {\bf v}_2=(2, -\frac{1}{2}E, \frac{1}{2}),$$where the semistability of ${\bf v}_1, \bv_2$ follows from $$\DLP_{H_m}^{<2}\left(-\frac{1}{2}E\right)=\DLP_{H_m}^{<2}\left(-\frac{1}{2}F\right)=\frac{1}{2}.$$ This shows that ${\bf v}$ is an example of a bad Chern character. 

The Beilinson-type spectral sequence (see \cite[Proposition 5.1]{drezet:hal-01175951}) allows one to resolve any  $\mu_{H_m}$-semistable sheaf $\VV$ of character ${\bf v}$
as \begin{equation}{\label{main example}}0 \to \OO(-1,-1)^{2} \to \OO(-1,0)^{3} \oplus \OO(0,-1)^{3} \to \VV \to 0.\end{equation}
Note that this is precisely the $L$-Gaeta type resolution \eqref{L Gaeta} with $L=\OO$. Thus we consider the family $\VV_t/T$ of $\OO(1,1)$-prioritary sheaves admitting an $\OO$-Gaeta type resolution $$0 \to \OO(-1,-1)^{2} \xrightarrow{\psi_t} \OO(-1,0)^{3} \oplus \OO(0,-1)^{3} \to \VV_t \to 0,$$ where $$T\subset \mathbb{H}=\Hom \left(\OO(-1,-1)^{2},  \OO(-1,0)^{3} \oplus \OO(0,-1)^{ 3}  \right)$$ is the open subset parameterizing injective sheaf  maps with torsion-free cokernel. By  Proposition \ref{Gaeta parameterizations}, the subset $T$ is not empty, $\codim_{\bb{H}} (\bb{H} \setminus T)\geq 2$ and the family $\cc{V}_t/T$ is complete.  We conclude that any $H_m$-semistable $\VV \in M_{H_m}({\bf v})$ is equal to some $\cc{V}_t$ for $t\in T$. 


We demonstrate that $S_{H_m}({\bf v}_1, {\bf v}_2)$ is nonempty in this complete family as follows. Note that we also have $$\DLP_{H_1}^{<2}\left(-\frac{1}{2}E\right)=\DLP_{H_1}^{<2}\left(-\frac{1}{2}F\right)=\frac{1}{2}.$$ We then take $$F_1\in M_{H_1}\left(2, -\frac{1}{2}F, \frac{1}{2}\right) \ \text{and} \ F_2 \in M_{H_1}\left(2, -\frac{1}{2}E, \frac{1}{2}\right),$$ and consider their direct sum $$F_1 \oplus F_2,$$ which  is a $\mu_{H_1}$-semistable sheaf. Since the Beilinson-type spectral sequence is insensitive to small variations in the polarization, this  sheaf  is still resolved by \eqref{main example} and, therefore, appears as $\cc{V}_{\tau}$ for some $\tau\in T$. For the $H_m$-polarization it is, however, no longer semistable: $$\mu_{H_m}(F_1)>\mu_{H_m}(F_1 \oplus F_2).$$ 

Note that $F_1$ and $F_2$ are in fact $\mu_{H_1}$-stable ($\bv_i$ is primitive), and since slope stability is open in the polarization, they remain $\mu_{H_m}$-stable by our choice of $H_m$. It follows that the $H_m$-Harder-Narasimhan filtration of $\cc{V}_{\tau}$ is $$0 \subset F_1 \subset \VV_{\tau},$$ so that  $\VV_{\tau}$ belongs to  $S_{H_m}({\bf v}_1, {\bf v}_2)$ and this divisorial $\Delta_i=\frac{1}{2}$-stratum  is nonempty in $T$.
\end{example} 

\begin{example}\label{other bad characters} Generalizing the previous example, we can generate an infinite sequence of bad $H_m$-semistable Chern characters ${\bf w}_k$ such that analogous complete families of $\OO(1,1)$-prioritary sheaves of character ${\bf w}_k$ arising from $\OO$-Gaeta type resolutions all contain a non-empty $\Delta_i=\frac{1}{2}$-stratum.

Set ${\bf w}_1:={\bf v}={\bf v}_1+{\bf v}_2$, where ${\bf v}, {\bf v}_1, {\bf v}_2$ are the characters from the previous example. Inductively define \begin{equation}\label{decomp}\bw_k:=\bv_1+\bw_{k-1} \ \text{for} \ k\geq 2.\end{equation}

One inductively checks that for $m=1+\varepsilon, \ 0<\varepsilon \ll 1,$ the character $\bw_k$ is $H_m$-semistable and all the conditions of Definition \ref{bad definition} are satisfied for the decomposition of $\bw_k$ as in \eqref{decomp}. Below we list characters $\bw_k=(r_k, \nu_k, \Delta_k)$ for small values of $k$ and plot their total slopes in the $(\varepsilon, \varphi)$-plane along with the top-down projection of various branches of the $\DLP$-surface (compare to Figure \ref{fig:example} (A)):

\begin{multicols}{2}
\begin{table}[H]
\begin{tabular}{c c c }

\\
\\
\\
\\

$k$ &  &  $ (r_k, \nu_k,\Delta_k)$ \\ \hline
$1$ &  & $(4, -1/4E-1/4F, 9/16)$ \\
$2$ & &  $(6,-1/6E -1/3F, 5/9)$ \\
$3$  &  & $(8, -1/8E-3/8F, 35/64)$ \\
$4$ & &$(10, -1/10E-2/5F, 27/50)$ \\
$5$ & &$(12, -1/12E-5/12F, 77/144)$  
\end{tabular}
\end{table}
\columnbreak
\begin{figure}[H]
    {{\includegraphics[width=7.5cm]{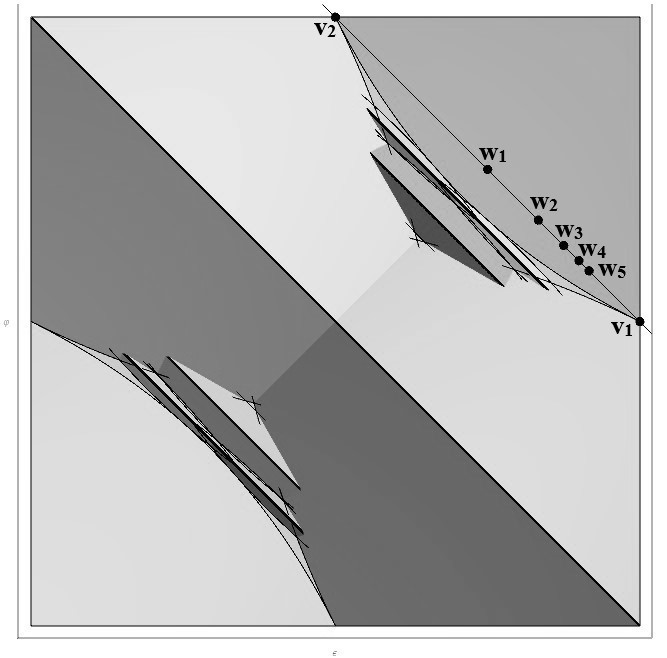} }}   
   \end{figure}
\end{multicols}

As in the previous example, the same Beilinson-type spectral sequence (\cite[Proposition 5.1]{drezet:hal-01175951}) allows one to resolve any  $\mu_{H_m}$-semistable sheaf $\cc{W}$ of character ${\bf w}_k$
as \begin{equation*}0 \to \OO(-1,-1)^{\alpha} \to \OO(-1,0)^{\beta} \oplus \OO(0,-1)^{\gamma} \to \cc{W} \to 0\end{equation*}for some positive integers $\alpha, \beta, \gamma$. Arguing as above, one considers the complete family $\cc{W}_t/T$ of  $\OO(1,1)$-prioritary sheaves of character $\bw_k$ admitting an $\OO$-Gaeta type resolution $$0 \to \OO(-1,-1)^{\alpha} \xrightarrow{\psi_t} \OO(-1,0)^{\beta} \oplus \OO(0,-1)^{\gamma} \to \cc{W}_t \to 0,$$ where $$T\subset \mathbb{H}=\Hom \left(\OO(-1,-1)^{\alpha},  \OO(-1,0)^{\beta} \oplus \OO(0,-1)^{\gamma}  \right)$$ is the open subset parameterizing injective sheaf maps, and one shows that the $\Delta_i=\frac{1}{2}$-stratum $S_{H_m}(\bv_1, \bw_{k-1})$ is nonempty.

Finally, note that $$\chi(\OO, \bv_1)=\chi(\OO,\bv_2)=0 \implies \chi(\OO, \bw_k)=0 $$ and $$\Delta(\bw_k)=\DLP^{<r(\bw_k)}_{H_m}(\nu(\bw_k))=\DLP_{H_m,\OO}(\nu(\bw_k)),$$so that the characters $\bw_k$ all lie on the branch of the $\DLP$-surface given by the line bundle $\OO$ (depicted as the upper-right circular sector in the picture above).  
\end{example}

\begin{example}\label{other bad characters 2} Furthermore, we can repeat the constructions of the previous two examples starting with any pair of of primitive characters of the form $\bv_1=(r, \varphi E +\varepsilon F, \frac{1}{2}), \bv_2=(r, \varepsilon E +\varphi F, \frac{1}{2})$ that lie on the branch of the $\DLP$-surface given by the line bundle $\OO$ $$\mu_{H_m}(\OO) > \mu_{H_m}(\bv_i), \quad \DLP^{<r}_{H_m}(\nu(\bv_i))=\DLP_{H_m, \OO}(\nu(\bv_i))$$ and  satisfy conditions (1), (3) and (4) of Definition \ref{bad definition}. This way, for each choice of $\bv_1,\bv_2$ as above we will get new  infinite sequences $${\bf w}_1:={\bf v}={\bf v}_1+{\bf v}_2, \quad \bw_k:=\bv_1+\bw_{k-1} \ \text{for} \ k\geq 2$$of bad Chern characters with a nonempty $\Delta_i=\frac{1}{2}$-stratum in a similarly constructed family $\cc{W}_t/T$.  

Here is a computer-generated list of such characters $\bv_i$ for small rank $r(\bv_i)$:
\footnotesize
\begin{table}[H]
\begin{tabular}{c| c | c | c}
$r(\bv_i)$ &  $\bv_1$& $\bv_2$ & $\bw_1=\bv_1+\bv_2$\\ \hline
$2$ & $(2, -1/2F, -1/2)$ & $(2, -1/2E, -1/2)$ & $(4, -1/4E-1/4F, 9/16)$ \\
$12$ & $(12, -1/4E-1/3F, 1/2)$& $(12, -1/3E-1/4F, 1/2)$ & $(24, -7/24E-7/24F, 289/576)$ \\
$70$  &    $(70,-2/7E-3/10F,1/2)$ &  $(70,-3/10E-2/7E,1/2)$ & $(140,-41/140E-41/140F,9801/19600)$ \\
$408$&  $(408,-7/24E-5/17F,1/2)$& $(408,-5/17E-7/24F,1/2)$ & $(816,-239/816E-239/816F,332929/665856)$
\end{tabular}
\end{table}

\normalsize
Further still, one can easily replace the line bundle $\OO$ by an arbitrary line bundle $L$ and generate analogous infinite sequences of bad characters lying on the branch of the $\DLP$-surface given by $L$.

\end{example}

\begin{question}\label{question} Note that the characters from the previous three examples lie on a branch of the $\DLP^{<r}_{H_m}$-surface given by a line bundle. It remains an open question whether a character lying \begin{enumerate} \item  above the $\DLP^{<r}_{H_m}$-surface, or 
\item on the branch of the $\DLP^{<r}_{H_m}$-surface controlled by a higher-rank exceptional bundle,\end{enumerate}can be a bad character. The evidence coming from numerical calculations on a computer points to a negative answer to (1). The answer to (2) is most likely positive, though to construct an example one should search for a character ${\bf v}$ of a really high rank: computer calculations  show that the rank should be taken to be $r({\bf v})>4000$ to find such bad characters.\end{question}


\subsection{Irreducible families}{\label{irreducible families}} Let $\bv_1, \bv_2, ...,\bv_l$ be $H$-semistable Chern characters with $$p_{H,\bv_1}> p_{H,\bv_2} > ... > p_{H, \bv_l}.$$ We conclude this section by discussing how to build an {\it irreducible} family of sheaves containing {\it all} torsion-free sheaves whose quotients in the $H$-Harder-Narasimhan filtration have invariants $\bv_1, \bv_2,...,\bv_l$. We later use these results to show the irreducibility of Shatz strata in certain complete families of $\OO(1,1)$-prioritary sheaves. The statements of this subsection are briefly mentioned in \cite{Yosh96} without proof, and the outline of the proof of Proposition \ref{Yoshioka irreducibility} was communicated to us by Yoshioka directly (also see the Appendix to \cite{Yoshioka1995} for some similar constructions). 

Given $H$-semistable Chern characters $\bv_1, \bv_2, ...,\bv_l$ with $$p_{H,\bv_1}> p_{H,\bv_2} > ... > p_{H, \bv_l},$$ consider the family $\mathbb{F}(\bv_1, \bv_2, ...,\bv_l)$ of isomorphism classes of torsion-free sheaves $\cc{V}$ whose $H$-Harder-Narasimhan filtration $$0 \subset \cc{F}_1 \subset \cc{F}_2 \subset ... \subset \cc{F}_l=\VV$$ is of length $l$ and whose quotients satisfy ${\bf v}(\gr_i)={\bf v}_i$. Note that when $l=1$, $\mathbb{F}({\bf v}_1)$ is just the family of isomorphism classes of $H$-semistable sheaves with Chern character ${\bf v}_1$. 



We first recall how to construct irreducible families for $H$-semistable sheaves of Chern character ${\bf v}$. 
\begin{lemma}{\label{case one}} Let ${\bf v}=(r, \nu, \Delta)$ be an $H$-semistable Chern character. Then there exists a family $\VV_s/S$ of sheaves over an irreducible base $S$ with the following property: 
\begin{equation}\tag{P} \cc{V}_s \in  \mathbb{F}({\bf v}_1) \ \text{for every} \ s \in S, \ \text{and for any} \ \cc{V} \in  \mathbb{F}(\bv_1) \ \text{there exists} \ s \in S \ \text{with} \ \cc{V}=\cc{V}_s.\end{equation} 
\end{lemma}
\begin{proof} When $r=1$, the moduli space $M_{H}(\bv_1)$ for $\bv_1=(1,\nu, n)$ is a fine moduli space with a universal family $\cc{U}$.  For $\bv_1=(1,\nu, n)$ the moduli space $M_H(\bv_1)$ is isomorphic to the Hilbert scheme $X^{[n]}$ of $n$ points on $X$. Therefore, it is irreducible and we take $S:=M_H(\bv_1)$.  

When $r\geq 2$,    take $S$ to be an open subset of the Quot scheme parameterizing $H$-semistable quotients $V \otimes \OO(-NH_m) \surj \VV$ for $N \gg0$ as in the GIT construction of $M_H({\bf v})$ (\cite[paragraph 4.3]{HL10}). Walter shows in the proof of \cite[Theorem 1]{Walter1993IrreducibilityOM} that $S$ is irreducible as a consequence of his more general result which says that the stack of $\OO(0,1)$-prioritary sheaves is irreducible (see our discussion in \S \ref{quadric}). \end{proof}

Now we prove the analogous result for the family $\mathbb{F}(\bv_1, \bv_2, ...,\bv_l)$. 

\begin{proposition}{\label{Yoshioka irreducibility}} Let  $\bv_1, \bv_2, ..., \bv_l$ be $H$-semistable Chern characters with \begin{equation}{\label{slopes}}p_{H,\bv_1}> p_{H,\bv_2} > ... > p_{H, \bv_l}.\end{equation} 

Then there exists a family $\VV_s/S$ of sheaves over an irreducible base $S$ with the following property:
\small 
\begin{equation}\tag{P} \cc{V}_s \in  \mathbb{F}(\bv_1, \bv_2, ...,\bv_l) \ \text{for every} \ s \in S, \ \text{and for any} \ \cc{V} \in  \mathbb{F}(\bv_1, \bv_2, ...,\bv_l) \ \text{there exists} \ s \in S \ \text{with} \ \cc{V}=\cc{V}_s.\end{equation}
\end{proposition}
\normalsize

\begin{proof} We use induction on $l$. Case $l=1$ is Lemma \ref{case one}.

For $l\geq 2$, take $\cc{V}\in\mathbb{F}(\bv_1, \bv_2, ...,\bv_l)$. It fits into a short exact sequence \begin{equation}{\label{exts}} 0 \to \cc{F} \to \cc{V} \to \cc{E} \to 0\end{equation}with  $\cc{F} \in \mathbb{F}({\bf v}_1)$ and $\cc{E} \in\mathbb{F}(\bv_2, ...,\bv_l)$. By the induction assumption we have a family $\cc{F}_t/T$ over an irreducible base $T$ satisfying Property (P) with respect to $\mathbb{F}({\bf v}_1)$, and a family $\cc{E}_r/R$ over an irreducible base $R$ satisfying Property (P) with respect to $\mathbb{F}(\bv_2, ...,\bv_l)$. Intuitively, we want to build $S$ by taking all possible extensions of $\cc{E}_r$ by $\cc{F}_t$ for all possible $t\in T$ and $r \in R$. However, since $\ext^1(\cc{E}_r, \cc{F}_t)$ may not be constant for different $t\in T, \ r\in R$, we will have to enlarge $S$ in a certain sense.

To this end, since by the induction assumption $\mathbb{F}({\bf v}_1)$ and $\mathbb{F}(\bv_2, ...,\bv_l)$ are bounded families, we can choose $N \gg 0$ so that $$H^i(X, \cc{E}_r(N H_m))=0 \ \text{for} \  i>0 \ \text{and all} \  r \in R$$ and 
\begin{equation}{\label{vanishing}}H^i(X, \cc{F}_t(NH_m))=0 \ \text{for} \  i>0 \ \text{and all} \ t\in T.\end{equation} Taking $V$ to be a vector space of dimension $h^0(X, \cc{E}_r(N H_m))$, we have a surjection \begin{equation}{\label{surjection}}V \otimes \OO(-NH_m) \surj \cc{E}_r \ \text{for each} \ r\in R.\end{equation} Since $$\hom( V \otimes   \OO(-NH_m), \cc{E}_r)= h^0(X,V^{\dual} \otimes\cc{E}_r(N H_m))$$ is constant as a function of $r$, we get that $p_*(\sheafhom(V\otimes q^*\OO(-NH_m),\cc{E})$ is a vector bundle on $R$. 

Let $$\mathbb{V} \stackrel{\pi}{\to} R$$ be the corresponding geometric vector bundle. Note that $\mathbb{V}$ remains irreducible. On $\mathbb{V} \times X$ we have a universal morphism $$V \otimes \pi_X^* q^*\OO(-NH_m) \stackrel{\Phi}{\to} \pi_X^* \cc{E},$$and we let $\bb{U} \subset \mathbb{V}$ be the open subset parameterizing surjective morphisms. Observe that due to \eqref{surjection} $\bb{U} \stackrel{\pi}{\to} R$ remains surjective and that $\bb{U}$ is irreducible. Let $A=\ker(\Phi|_{\bb{U}\times X})$, \ $B=V \otimes \pi_X^* q^*\OO(-NH_m)$, and consider the exact sequence of sheaves over $\bb{U}\times X$: \begin{equation}{\label{shortsequence}}0 \to A \stackrel{\Xi}{\to} B \stackrel{\Phi}{\to} \pi^*_X\cc{E} \to 0.\end{equation}

By \eqref{vanishing} we have $\Ext^i(B_u, \cc{F}_t)=0$ for $i>0$ and all $u\in \bb{U}, t\in T$. By  \eqref{slopes} and semistability we also have that \begin{equation*}{\label{ext vanishing}}\Ext^2(\cc{E}_r, \cc{F}_t)=\Hom(\cc{F}_t, \cc{E}_r\otimes K_X))^{\dual}=0 \ \text{for all} \ r\in R, t\in T.\end{equation*}  Applying $\Hom(\_,\cc{F}_t)$ to \eqref{shortsequence} at point $u\in \bb{U}$ we get that $\Ext^i(A_u,\cc{F}_t)=0$ for $i>0$ and all $u\in \bb{U}, t\in T$. Thus $\hom(A_u, \cc{F}_t)$ is constant for all $u\in \bb{U}, t\in T$ and we have \small \begin{equation}{\label{covering of extensions}} 0 \to \Hom(\cc{E}_{\pi(u)}, \cc{F}_t) \to \Hom(B_u, \cc{F}_t) \to \Hom(A_u, \cc{F}_t) \to \Ext^1(\cc{E}_{\pi(u)}, \cc{F}_t) \to 0 \ \text{for all} \ u \in \bb{U}, t\in T. \end{equation} \normalsize 

Recall that in our intuitive explanation we mentioned that parameterising extensions of $\cc{E}_r$ by $\cc{F}_t$ might be problematic due to jumping values of $\ext^1(\cc{E}_r, \cc{F}_t)$ for different $r \in R, t \in T$. Now \eqref{covering of extensions} shows that  $\Hom(A_u, \cc{F}_t)$ is a vector space of constant dimension for different $u\in \bb{U}, t \in T$, so we can build an irreducible space parameterizing all homomorphisms $A_u \to \bb{F}_t$ for all $u\in \bb{U}, t \in T$. Since $\Hom(A_u, \cc{F}_t)$
 surjects onto $\Ext^1(\cc{E}_{\pi(u)}, \cc{F}_t)$, this  irreducible space will be a  "cover" for the naive "space of extensions." 

To this end, consider the natural projections \begin{equation*} \begin{aligned} &pr_{\bb{U}\times X}:&\bb{U} \times T \times X &\to \bb{U} \times X, \\ &pr_{T \times X}: &\bb{U} \times T \times X &\to T \times X, \\ &pr_{\bb{U} \times T}:&\bb{U} \times T \times X &\to \bb{U}  \times T. \end{aligned} \end{equation*} By the above discussion, $$pr_{\bb{U}  \times T, *} (\sheafhom(pr_{\bb{U}\times X}^* A, pr_{T \times X}^* \cc{F})$$ is a vector bundle over the irreducible base $\bb{U}  \times T$, therefore the associated geometric vector bundle $S \stackrel{\rho}{\to} \bb{U}  \times T$ is irreducible too. Consider the universal morphism over $S \times X$
$$\rho_X^* pr_{\bb{U} \times X}^*A \stackrel{\Psi}{\to} \rho_X^*pr_{T\times X}^*\cc{F},$$as well as the induced morphism $$\rho_X^* pr_{U \times X}^*A \inj{\rho_X^* pr_{U \times X}^*(\Xi)}\rho_X^* pr_{U \times X}^*B,$$where $\rho_X:=\rho \times Id_X$. Taking the direct sum of these maps and calling the resulting cokernel sheaf by $\mathcal{V}$, we obtain the following short exact sequence of sheaves on $S \times X$:
\begin{equation}{\label{finish}}0 \to \rho_X^* pr_{U \times X}^*A \xrightarrow{\rho_X^* pr_{U \times X}^*(\Xi) \oplus \Psi} \rho_X^* pr_{U \times X}^*B \oplus  \rho_X^*pr_{T\times X}^*\cc{F} \stackrel{\Omega}{\to} \cc{V} \to 0.\end{equation} For a point $s \in S$, this short exact sequence can be expanded into the following commutative diagram:
\begin{center}\begin{tikzcd}
            & 0 \arrow[d]           & & 0 \arrow[d]                    & 0 \arrow[d]           &   \\
0 \arrow[r] & 0 \arrow[rr] \arrow[d] & &\cc{F}_{pr_T(\rho(s))} \arrow[r] \arrow[d]          & \cc{F}_{pr_T(\rho(s))} \arrow[r] \arrow[d] & 0 \\
0 \arrow[r] & A_{pr_{\bb{U}} (\rho(s))} \arrow[rr,"\Xi_{pr_{\bb{U}} (\rho(s))}\oplus \Psi_s"] \arrow[d] &  & B_{pr_{\bb{U}} (\rho(s))} \oplus \cc{F}_{pr_T(\rho(s))} \arrow[r,"\Omega_s"] \arrow[d] & \cc{V}_s \arrow[r] \arrow[d] & 0 \\
0 \arrow[r] &  A_{pr_{\bb{U}} (\rho(s))} \arrow[rr,"\Xi_{pr_{\bb{U}} (\rho(s))}"] \arrow[d] & & B_{pr_{\bb{U}} (\rho(s))} \arrow[r,"\Phi_{pr_{\bb{U}} (\rho(s))}"] \arrow[d]          & \cc{E}_{\pi(pr_{\bb{U}} (\rho(s)))} \arrow[r] \arrow[d] & 0 \\
            & 0                     & & 0                              & 0                     &  
\end{tikzcd}
\end{center}

The row in the middle corresponds to \eqref{finish}, while the row at the bottom corresponds to \eqref{shortsequence}. The column in the middle is a trivial extension.
By construction, the fiber of $S$ over point $(u,t) \in \bb{U} \times T$ is the vector space $\Hom(A_u, \cc{F}_t)$ which by \eqref{covering of extensions} surjects onto $\Ext^1(\cc{E}_{\pi(u)}, \cc{F}_t)$. For a given $s \in S$ with corresponding $\Psi_s \in \Hom(A_u, \cc{F}_t)$, the resulting extension in $\Ext^1(\cc{E}_{\pi(u)}, \cc{F}_t)$ is displayed in the right column in the above diagram. This way, as $s$ varies over $S$, we parameterize all possible extensions \eqref{exts} and the Property (P) is satisfied.
\end{proof}

\section{Group actions and Gaeta-type resolutions}\label{Group act and Gaet} In this section, we recall some basic facts about the Picard group of $G$-linearized line bundles on a variety $Y$, and discuss how to compute with the Donaldson homomorphism when working with the family of $\OO(1,1)$-prioritary sheaves admitting an $L$-Gaeta type resolution constructed above in Propositions \ref{Gaeta parameterizations} and \ref{dual Gaeta family}.

\subsection{Characters and linearized line bundles}Let $G$ be an algebraic group acting on a variety $Y$. A \emph{crossed morphism} is a morphism of varieties $$\theta: G \times Y \to \bb{C}^*,$$ satisfying $$\theta(gg', y)=\theta(g, g'y) \theta(g', y) \ \ \text{for any} \ \ g,g'\in G, \  y \in Y.$$Crossed morphisms are in bijection with the linearizations of the trivial bundle $\OO_Y$. Indeed, given a crossed morphism $\theta$ define the action of $G$ on the total space $Y \times \bb{C}$ of $\OO_Y$ over the action of $G$ on $Y$ by $$g \cdot(y, a)=(g \cdot y, \theta(g,y)a).$$A crossed morphism $\theta$ is said to be \emph{principal} if there exists $f \in \OO^*(Y)$ such that $$\theta(g,y)=\frac{f(g\cdot y)}{f(y)} \ \ \text{for any} \ \ g\in G, y\in Y.$$Observe that for a principal crossed morphism $\theta$ coming from $f\in \OO^*(Y)$  the trivial line bundle $\OO_Y$ with a trivial linearization is isomorphic as $G$-bundles to  the bundle $(\OO_Y, \theta)$  via $$(y,a) \mapsto (y, f(y)a),$$which is easily seen to be a $G$-equivariant map. 

In summary, we get an exact sequence $$\OO^*(Y) \to \textnormal{CrMor}(Y,G) \to \Pic^G(Y) \to \Pic(Y)^G,$$where the second term is the group of crossed morphisms and the last term denotes $G$-invariant line bundles. 
Now note that any character $\eta\in \Char(G)$ can be viewed as a crossed morphism via $$\theta_{\eta}(g,y)=\eta(g).$$  Drez\'{e}t shows in \cite[Proposition 14]{drezet1987fibres} that for those algebraic groups $G$ for which any invertible function on $G$ can be written as a product of a constant and a character of $G$ we in fact have an isomorphism $$\Char(G) \to \textnormal{CrMor}(Y,G).$$ Therefore, for such groups we have the following result.

\begin{proposition}\label{crossed morphisms} Let $Y$ be an integral variety equipped with an action of an algebraic group $G$. Further suppose that any invertible function on $G$ can be written as a product of a constant and a character of $G$. Then we have the following exact sequence
\begin{equation}{\label{linearized line bundles}} \cc{O}^*(Y) \to \Char(G) \to \Pic^G(Y) \to \Pic(Y)^G. \end{equation}
\end{proposition}
We remark that in the first map $(f \mapsto \eta_f)$ the resulting character $\eta_f$ is described by the equality $$\eta_f(g)=\frac{f(g\cdot y)}{f(y)} \quad \text{for any} \ g\in G, y\in Y.$$ 

\subsection{Characters of the general linear group}In the context of the Gaeta-type resolutions we will be interested in the action of the  general linear group and groups closely related to it. These groups will satisfy the assumption of Proposition \ref{crossed morphisms}. In view of exact sequence \eqref{linearized line bundles}, we now recall how to describe characters for such groups.

For a fixed positive integer $n$, consider the homomorphism $$\bb{Z} \to \Char(GL(n))$$
$$a \mapsto [\eta_a:A \mapsto \det(A)^a].$$Since the coordinate ring of $GL(\alpha)$ is the localization $\bb{C}[ \{x_{ij} \}]_{\det}$, the only invertible functions mapping $Id\in GL(n)$ to $1 \in \bb{C}$ are of the form $[A \to \det(A)^a]$ for some $a \in \bb{Z}$. It follows that the above homomorphism is in fact an isomorphism.

More generally, for $k$ positive integers $n_1,n_2,...,n_k$ let $G:=GL(n_1) \times GL(n_2) \times ... \times GL(n_k)$. We have an isomorphism \begin{equation}{\label{char groups}} \bb{Z}^k \to \Char(G)\end{equation} given by $$(a_1,a_2,...,a_k) \mapsto [\eta_{(a_1,a_2,...,a_k)}:(A_1,A_2,...,A_k) \mapsto \det(A_1)^{a_1} \det(A_2)^{a_2} ... \det(A_k)^{a_k}].$$

Finally, let $\overline{G}:=\left( GL(n_1) \times GL(n_2) \times ... \times GL(n_k) \right) /\ \bb{C}^*(Id, Id, ..., Id)$. Under the above isomorphism $\Char(\overline{G})$ can be described as\begin{equation}{\label{characters}}\Char(\overline{G})=\{(a_1,a_2,...,a_k) \in \bb{Z}^k \ | \ a_1 n_1+a_2 n_2+...+a_kn_k=0 \} \subset \bb{Z}^k.\end{equation}

\subsection{Natural action of $G$ on Gaeta-type resolutions}We return back to the case $X=\XX$. Consider the family $\cc{V}_t/T$ of $\OO(1,1)$-prioritary sheaves of Chern character ${\bf v}$ with $r(\bv) \geq2$ over  ${T=U \subset \bb{H}}$ admitting an $L$-Gaeta type resolution \eqref{L Gaeta} $$0 \to L(-1,-1)^{\alpha} \stackrel{\psi_t}{\to} L(-1,0)^{\beta} \oplus L(0,-1)^{\gamma} \oplus L^{\delta} \to \VV_t \to 0,$$ as in Proposition \ref{Gaeta parameterizations}. We first treat the case where all integers $\alpha, \beta,\gamma,\delta$ are not zero and  say how to modify the argument when some of the exponents vanish later.

In this case, there is a natural group action of $$G= GL(\alpha) \times GL(\beta) \times GL(\gamma) \times GL(\delta)$$ on $T$: for point $\psi_t \in T$ the point $$(g_{\alpha},g_{\beta}, g_{\gamma},g_{\delta}) \cdot \psi_t$$corresponds to the morphism $$(g_{\beta} \oplus g_{\gamma} \oplus g_{\delta}) \circ \psi_t \circ (g_{\alpha})^{-1}.$$

Note that since $$c(Id,Id,Id,Id) \cdot \psi_t=\psi_t, \ c \in \bb{C},$$ there is also an induced action of $$\overline{G}=( GL(\alpha) \times GL(\beta) \times GL(\gamma)\times GL(\delta) )  / \ \bb{C}^*(Id, Id, Id, Id)$$ on $T$. 

We extend both actions  onto $T \times X$. On $T \times X$, there is a universal short exact sequence of sheaves\begin{equation}{\label{universal Gaeta}}0 \to q^*(L(-1,-1))^{\alpha} \stackrel{\psi}{\to}q^*(L(-1,0))^{\beta} \oplus q^* (L(0,-1))^{\gamma} \oplus q^*L^{\delta} \to \cc{V} \to 0.\end{equation} We endow the trivial families with a natural $G$-linearization as follows. Let $g=(g_{\alpha},g_{\beta},g_{\gamma},g_{\delta}) \in G.$ The action of $g$ is described as
\begin{equation*}
\begin{aligned}  (q^*(L(-1,-1))^{\alpha})_t  &=& L(-1,-1)^{\alpha} &\stackrel{g_{\alpha}}{\longrightarrow} &L(-1,-1)^{\alpha} &=  (q^*(L(-1,-1))^{\alpha})_{g \cdot t} \\
(q^*(L(-1,0))^{\beta})_t &=& L(-1,0)^{\beta} &\stackrel{g_{\beta}}{\longrightarrow} &L(-1,0)^{\beta} &=(q^*(L(-1,0))^{\beta})_{g \cdot t} \\
( q^*(L(0,-1))^{\gamma} )_t &=& L(0,-1)^{\gamma} &\stackrel{g_{\gamma}}{\longrightarrow}& L(0,-1)^{\gamma}&=(q^*(L(0,-1))^{\gamma})_{g\cdot t} \\
(q^*L^{\delta})_t &=& L^{\delta} &\stackrel{g_{\delta}}{\longrightarrow}& L^{\delta}&= (q^*L^{\delta})_{g \cdot t}
\end{aligned}
\end{equation*}
There is then a unique $G$-linearization of $\VV$ making \eqref{universal Gaeta} a short exact sequence of $G$-{\it linearized} sheaves. For $g\in G, \psi_t\in T$ as above, it is described as the unique isomorphism $\Phi_{(g,t)}$ completing the diagram
\begin{center}
\begin{tikzcd}
0 \arrow[r] & L(-1,-1)^{\alpha} \arrow[r, "\psi_t"] \arrow[d, "g_{\alpha}"] & L(-1,0)^{\beta} \oplus L(0,-1)^{\gamma} \oplus L^{\delta} \arrow[r] \arrow[d, "g_{\beta} \oplus g_{\gamma} \oplus g_{\delta}"] & \cc{V}_t \arrow[r] \arrow[d, "\Phi_{(g,t)}", dashed] & 0 \\
0 \arrow[r] & L(-1,-1)^{\alpha} \arrow[r, "\psi_{g \cdot t}"]                & L(-1,0)^{\beta} \oplus L(0,-1)^{\gamma} \oplus L^{\delta} \arrow[r]                & \cc{V}_{g\cdot t} \arrow[r]                        & 0.
\end{tikzcd}
\end{center}

This allows us to use Lemma \ref{computation} (3) and compute the Donaldson homomorphism $$\lambda_{\VV_t}: K(X) \to \Pic^G(T)$$explicitly, taking into account that both $K(X)$ and $\Pic^G(T)$ are {\it free} $\bb{Z}$-modules. Specifically, we identify $K(X) \cong \bb{Z}^4$ by choosing the following $\bb{Z}$-basis $${\bf e}_1:=[L^{\dual}(-1,-1)],{\bf e}_2:=[L^{\dual}(-1,0)], {\bf e}_3:=[L^{\dual}(0,-1)], {\bf e}_4:=[L^{\dual}].$$  On the other hand, by Proposition \ref{Gaeta parameterizations} $$\codim_{\bb{H}}(\bb{H}\setminus T) \geq 2$$ and, since $\bb{H}$ is an affine space, it follows that $$\cc{O}^*(T)=\bb{C}^* \ \text{and} \ \Pic(T)=0.$$ Note that $G$ satisfies the assumptions of Proposition \ref{crossed morphisms}, so  we get \begin{equation}{\label{Pic-Char}}\ \Char(G) \stackrel{\sim}{\to} \Pic^G(T) \end{equation} and the former group was shown to be $\bb{Z}^4$ in \eqref{char groups}. 

\begin{proposition}{\label{Donaldson on Gaeta}} Consider the family $\cc{V}_t/T$ of $\OO(1,1)$-prioritary sheaves of character $\bv$ with $r(\bv)\geq 2$ admitting an $L$-Gaeta type resolution $$0 \to L(-1,-1)^{\alpha} \stackrel{\psi_t}{\to} L(-1,0)^{\beta} \oplus L(0,-1)^{\gamma} \oplus L^{\delta} \to \VV_t \to 0$$ where $$T \subset \mathbb{H}=\Hom \left(L(-1,-1)^{\alpha},  L(-1,0)^{\beta} \oplus L(0,-1)^{\gamma} \oplus L^{\delta} \right)$$ is the open subset parameterizing injective sheaf maps with torsion-free cokernel and all the exponents $\alpha,\beta,\gamma,\delta$ are nonzero. 

Then the Donaldson homomorphism $$\lambda_{\VV_t}: K(X) \to \Pic^G(T)$$ is an isomorphism, and the image of ${\bf v}^{\perp}$ is equal to $\Char(\overline{G}) \subset \Char(G) \cong \Pic^G(T).$
\end{proposition}

\begin{proof} Let $$A=q^*(L(-1,1))^{\alpha}=q^*(L(-1,-1)) \otimes V_{\alpha}$$ and $$B=q^*(L(-1,0))^{\beta} \oplus q^* (L(0,-1))^{\gamma} \oplus q^*L^{\delta}=(q^*(L(-1,0))\otimes V_{\beta})\oplus (q^*(L(0,-1))\otimes V_{\gamma}) \oplus (q^*L\otimes V_{\delta}),$$where $V_{\alpha}, V_{\beta}, V_{\gamma}, V_{\delta}$ are vector spaces of dimension $\alpha, \beta, \gamma$ and $\delta$ respectively.  Since the universal short exact sequence \eqref{universal Gaeta} is a sequence of $G$-linearized sheaves, we have that for ${\bf u} \in K(X),$ $$\lambda_{\cc{V}}({\bf u})= \lambda_B ({\bf u}) \otimes \lambda_A({\bf u})^{\dual} \ \text{as elements of} \ \Pic^G(T),$$or
$$\lambda_{\cc{V}}({\bf u}) =\lambda_B ({\bf u}) - \lambda_A({\bf u}) \ \text{as elements of} \ \Char(G)$$under the isomorphism \eqref{Pic-Char}. 

Using this, one readily checks that $$p_! (B \otimes q^*L^{\dual}(-1,-1))=p_!((q^*\cc{O}(-1,-2)\otimes V_{\beta} )\oplus (q^*\cc{O}(-2,-1) \otimes V_{\gamma})\oplus (q^*\cc{O}(-1,-1) \otimes V_{\delta}))=0$$ and $$p_!(A \otimes q^*L^{\dual}(-1,-1))=p_!(q^*\cc{O}(-2,-2) \otimes V_{\alpha})=[\cc{O}_T \otimes V_{\alpha}],$$viewed as elements in  $K^G(X)$. Thus $$\lambda_{\cc{V}_t}({\bf e}_1)=\det(\cc{O}_T \otimes V_{\alpha})^{\dual},$$which corresponds to character $$\eta_{(-1,0,0,0)} \in \Char(G)$$under the isomorphism \eqref{Pic-Char}. Similar calculations show that $$ \begin{aligned} &\lambda_{\cc{V}}({\bf e}_2) \ \text{corresponds to}  \ \eta_{(0,-1,0,0)},\\ &\lambda_{\cc{V}}({\bf e}_3) \ \text{corresponds to}  \ \eta_{(0,0,-1,0)}, \\ &\lambda_{\cc{V}}({\bf e}_4) \ \text{corresponds to}  \ \eta_{(0,0,0,1)}.\end{aligned}  $$

In summary, the Donaldson homomorphism $\lambda_{\cc{V}_t}$ viewed as a map $K(X) \to \Char(G)$ is given by
\begin{equation}{\label{matrix}} {\bf u}=a_1 \be_1 + a_2 \be_2+a_3 \be_3+a_4\be_4  \mapsto \eta=\eta_{(-a_1, -a_2, -a_3,a_4)} \end{equation}
Alternatively, it has matrix $$\begin{bmatrix} -1 & 0 & 0 & 0 \\ 0 & -1 & 0 & 0 \\ 0 & 0 & -1 & 0 \\ 0 & 0 & 0 & 1 \end{bmatrix}$$ when viewed as a map $\bb{Z}^4 \to \bb{Z}^4$.

Now we turn to the second statement in the proposition. One checks that $${\bf v}=-\alpha [L(-1,-1)] +\beta[L(-1,0)]+\gamma[L(0,-1)]+\delta[L] \ \text{in}  \ K(X)$$ by applying $\chi(\_ \cdot \be_i)$ to both sides and using \eqref{Gaeta equation} along with the fact that $${\bf \overline{e}}_1=[L(-1,-1)], {\bf \overline{e}}_2=[L(-1,0)],{\bf \overline{e}}_3=[L(0,-1)],{\bf \overline{e}}_4=[L]$$ is a $\chi(\_ \cdot \_)$-orthogonal basis to ${\bf e}_1,{\bf e}_2,{\bf e}_3,{\bf e}_4$:$$ \chi({\bf \overline{e}}_i\cdot {\bf e}_j)=0$$ for $i\neq j$ and $$\chi({\bf \overline{e}}_1\cdot {\bf e}_1)=\chi({\bf \overline{e}}_4\cdot {\bf e}_4)=1,\chi({\bf \overline{e}}_2\cdot {\bf e}_2)=\chi({\bf \overline{e}}_3\cdot {\bf e}_3)=-1.$$ 

One further checks that the condition $${\bf u}\in {\bf v}^{\perp} \iff \chi ({\bf v} \cdot {\bf u})=0$$ is equivalent to $$-a_1\alpha  - a_2\beta  -a_3\gamma  + a_4 \delta =0.$$ By  \eqref{characters}, this last condition is precisely equivalent to   $\eta=\eta_{(-a_1, -a_2, -a_3,a_4)} \in \Char(\overline{G})$.  
\end{proof}

The above proof easily carries over to the case when one of the exponents in an $L$-Gaeta-type resolution is zero. In particular, we will later work with the case when  $\delta=0$. In this case, set \begin{equation}\label{missing exponent} \begin{aligned} G_{\hat{\delta}}&= GL(\alpha) \times GL(\beta) \times GL(\gamma) \\ \overline{G}_{\hat{\delta}}&=( GL(\alpha) \times GL(\beta) \times GL(\gamma) )  / \ \bb{C}^*(Id, Id, Id).\end{aligned}\end{equation}

\begin{proposition}\label{Donaldson on Gaeta 2} Consider the family $\VV_t/T$ of $\OO(1,1)$-prioritary sheaves of character $\bv$ with $r(\bv)\geq 2$ admitting an $L$-Gaeta-type resolution \begin{equation*}\label{seq} 0 \to L(-1,-1)^{\alpha} \stackrel{\phi_t}{\to} L(-1,0)^{\beta} \oplus L(0,-1)^{\gamma} \to \VV_t \to 0,\end{equation*}where  $$T\subset \mathbb{H}=\Hom \left(L(-1,-1)^{\alpha},  L(-1,0)^{\beta} \oplus L(0,-1)^{\gamma} \right)$$ is the open subset parameterizing injective sheaf maps with torsion-free cokernel and the exponents $\alpha, \beta, \gamma$ are nonzero. 

Then the Donaldson homomorphism $$\lambda_{\VV_t}: K(X) \to  \Pic^{G_{\hat{\gamma}}}(T)$$ is an epimorphism, and the image of ${\bf v}^{\perp}$ is equal to $\Char(\overline{G}_{\hat{\delta}}) \subset \Char(G_{\hat{\delta}}) \cong \Pic^{G_{\hat{\delta}}}(T).$
\end{proposition}

Finally, we can repeat the  discussion of this subsection for the "dual version"   of a Gaeta-type resolution. Consider the family $\cc{V}_t/T$ of $\OO(1,1)$-prioritary sheaves of Chern character ${\bf v}$ with $r(\bv) \geq2$ over  ${T=U \subset \bb{H}}$ admitting an $L$-Gaeta type resolution \eqref{dual Gaeta} $$0 \to  \VV_t \to  L(1,0)^{\alpha} \oplus L(0,1)^{\beta} \oplus L^{\gamma} \xrightarrow{\psi_t} L(1,1)^{\delta}\to 0$$ as in Proposition \ref{dual Gaeta family}. There is a natural action of  $$G= GL(\alpha) \times GL(\beta) \times GL(\gamma) \times GL(\delta) \ \text{and} \ \overline{G}=( GL(\alpha) \times GL(\beta) \times GL(\gamma)\times GL(\delta) )  / \ \bb{C}^*(Id, Id, Id, Id)$$  on $T$ and $T\times X$ if all the exponents $\alpha, \beta, \gamma, \delta$ are nonzero,  and of  $$G_{\hat{\gamma}}= GL(\alpha) \times GL(\beta) \times GL(\delta) \ \text{and} \ \overline{G}_{\hat{\gamma}}=( GL(\alpha) \times GL(\beta) \times GL(\delta) )  / \ \bb{C}^*(Id, Id, Id)$$  if $\alpha, \beta, \delta>0$, but $\gamma=0$.

As before, the action of $G$ (resp. $G_{\hat{\gamma}}$) on $T \times X$ lifts to a linearization of the universal families of sheaves and we have the obvious analogues of Propositions \ref{Donaldson on Gaeta} and \ref{Donaldson on Gaeta 2}.

\section{The Picard group of the moduli space}\label{main section}

\subsection{Associated exceptional bundles and the Donaldson homomorphism}\label{easy kernel} 

Exceptional bundles associated to an $H_m$-semistable character ${\bf v}$ (see Definition \ref{associated exceptional}) give rise to easy-to-describe classes in the kernel of the the Donaldson homomorphism $\lambda: {\bf v}^{\perp} \to \Pic(M_{H_m}({\bf v}))$. 

Specifically, suppose $E$ is associated to a nonsemiexceptional $H_m$-semistable character ${\bf v}$ and $\mu_{H_m}(E) \geq \mu_{H_m}({\bf v}).$ By semistability and Serre duality, for any semistable $\VV$ of character ${\bf v}$ we have $$\Hom(E,\VV)=\Ext^2(E,\VV)=0,$$and since $\Delta(\VV)=\DLP_{H_m, E}(\nu(\VV))$, we also have $$\chi(E,\VV)=0 \quad \text{and} \quad \Ext^1(E,\VV)=0.$$ This way, we see that if $\cc{U}_r/R$ is the family of $H_m$-semistable sheaves parameterized by a subset $R$ of the Quot scheme used in the GIT construction of $M_{H_m}({\bf v})$, then in the notation of \S \ref{Don homo} we have $$p_!( q^* [E^{\dual}] \cdot [\cc{U}])=0.$$Proposition \ref{Donaldson} then shows that $$\lambda([E^{\dual}])=0.$$

Similarly, if $\mu_{H_m}(E) < \mu_{H_m}({\bf v}),$ then $$\lambda([E^{\dual}\otimes K_X])=0.$$

For that matter, we introduce the following uniform notation: for an exceptional bundle $E$ associated to character ${\bf v}$, we define the following class in $K(X)$
$$[\overline{E}]=
        \begin{cases}
  [E^{\dual}]  &  \text{if} \ \mu_{H_m}(E)\geq \mu_{H_m}({\bf v}), \\ 
[ E^{\dual} \otimes K_X ] & \text{if} \ \mu_{H_m}(E)< \mu_{H_m}({\bf v}).              
        \end{cases}
$$

\subsection{The main theorem} 
Finally, we are ready to state our first main result about the Picard group of the moduli space $M_{H_m}({\bf v})$. We recall that $\lambda: {\bf v}^{\perp} \to \Pic(M_{H_m}({\bf v}))$ denotes the Donaldson homomorphism constructed in Proposition \ref{Donaldson}.

\begin{theorem}{\label{Main theorem}} Let ${\bf v}=(r, \nu, \Delta)\in K(X)$ be a character with $r \geq 2$ and $\Delta\geq \frac{1}{2}$. Let $\epsilon \in \bb{Q}$ be sufficiently small (depending on $r$), $0 < |\epsilon| \ll 1$, and set $m=1+\epsilon$.

\begin{enumerate}\item If    $\Delta> \frac{1}{2}, \ \Delta > \DLP_{H_m}^{<r}(\nu)$ and either \begin{enumerate} \item  $\Delta -\frac{1}{r} \geq \DLP_{H_m}^{<r}(\nu)$ and $\Delta - \frac{1}{r}>\frac{1}{2}$, or \item $\Delta -\frac{1}{r} \geq \DLP_{H_m}^{<r}(\nu)$ and ${\bf v'}=\left(r, \nu, \Delta-\frac{1}{r}\right) \ \text{is primitive} ,$ or \item ${\bf v}$ is a good character, \end{enumerate} then $$\Pic(M_{H_m}({\bf v}))\cong \ZZ^3$$ and $\lambda$ is an isomorphism.

\item If ${\bf v}$ is a good character with $\Delta=\DLP^{<r}_{H_m}(\nu)>\frac{1}{2}$ with a single exceptional bundle $E$ associated to ${\bf v}$, then $$\Pic(M_{H_m}({\bf v})) \cong \ZZ^2$$ and $\lambda$ is an epimorphism with $$\ker \lambda =\ZZ [\overline{E}].$$ 

\item \begin{enumerate} \item If $\Delta=\DLP^{<r}_{H_m}(\nu)>\frac{1}{2}$ with two exceptional bundles $E_1, E_2$ associated to ${\bf v}$ and $$\bv \ \text{is primitive or} \ (E_1,E_2) \ \text{is an exceptional pair},$$  then $$\Pic(M_H({\bf v})) \cong \ZZ$$ and $\lambda$ is an epimorphism with $$\ker \lambda=\ZZ [\overline{E_1}] + \ZZ [\overline{E_2}].$$

\item If $\Delta=\frac{1}{2}$, then $M_{H_m}({\bf v})$ is a projective space and $$\Pic(M_{H_m}({\bf v})) \cong \ZZ.$$  
\end{enumerate}
\end{enumerate}
\end{theorem}

We draw the reader's attention to statements (1.a) and (2) of the above theorem which use the notion of a \emph{good} Chern character from Definition \ref{bad definition}. This assumption is substantial: as we show in Theorem \ref{Main theorem 2} below for  certain bad characters lying on a single branch of the $\DLP$-surface the Picard number drops to $1$. We also emphasize that determining which statement of the theorem applies to a given character ${\bf v}=(r, \nu, \Delta)$ is a finite computational procedure and, therefore, can be implemented on a computer: the computation of $\DLP^{<r}_{H_m}(\nu)$ is finite, and to check whether ${\bf v}$ is a good  character one needs to test finitely many candidate characters with discriminant $\frac{1}{2}$ for whether they give a decomposition ${\bf v}={\bf v_1}+{\bf v_2}$ as in Definition \ref{bad definition}. 

Let us mention that the classification  in the case $r=1$ also fits the above pattern in a certain sense. For ${\bf v}=(1, aE+bF, n)$ the moduli space $M_{H_m}({\bf v})$ is isomorphic to the Hilbert scheme of $n$ points $X^{[n]}$. Therefore, when $$n=1=\DLP_{H_m,\OO(a,b)}(aE+bF)$$ we have  $X^{[1]} \cong X\ \text{and} \ \Pic(X^{[1]}) \cong \bb{Z}^2.$ When $$n>1=\DLP_{H_m,\OO(a,b)}(aE+bF)$$ we have $\Pic(X^{[n]})\cong \bb{Z}^3$ by the Theorem of Fogarty \cite{Fogarty}.

The proof of Theorem \ref{Main theorem} occupies the rest of this section. For the convenience of the reader, we will prove the theorem in a series of propositions   according to how the theorem is stated.  Cases (1.a), (1.b) and (3) of the theorem have relatively simple proofs. We then prove a part of case (2) so that at that point the theorem will be proved for characters ${\bf v}$ in a large  region in the $(r, \nu, \Delta)$-space. The remaining characters ${\bf v}$ have their discriminant in a narrow range $\frac{1}{2}<\Delta<1$ and have no line bundles associated to them. These conditions allow us to deal with  case (1.c) and the remainder of case (2)  in a uniform fashion though the proofs become considerably more involved.

\subsection{Proof of the main theorem} We start by proving case (3.b) of the theorem to be able to assume $\Delta>\frac{1}{2}$ in the rest of the proof and use the surjectivity of the Donaldson morphism from Theorem \ref{Yoshioka surjectivity}.

\begin{proposition}\label{first prop} Let ${\bf v}=(r, \nu, \Delta)\in K(X)$ be a character with $r \geq 2$ and $\Delta = \frac{1}{2}$. Let $\epsilon \in \bb{Q}$ be sufficiently small (depending on $r$), $0 < |\epsilon| \ll 1$, and set $m=1+\epsilon$.

If ${\bf v}$ is $H_m$-semistable, then $M_{H_m}({\bf v})$ is a projective space and $$\Pic(M_{H_m}({\bf v})) \cong \ZZ.$$
\end{proposition}

\begin{proof}  First, if ${\bf v}=(r,\nu, \frac{1}{2})=(r, c_1, \chi)$ is a primitive character, then for an appropriate (generic) choice of $m=\frac{p}{q}$ one checks that $$\gcd (r, c_1 \cdot (qH_m),\chi)=1.$$ In this case, $M_{H_m}({\bf v})=M_{H_m}^{s}({\bf v})$ is a smooth projective variety of dimension $$\dim M_{H_m}({\bf v})=\exp \dim M_{H_m}({\bf v})=1.$$Moreover, Walter shows in \cite{Walter1993IrreducibilityOM}  that $M^s_{H_m}({\bf v})$ is irreducible and unirational. It follows that in this case $M_{H_m}({\bf v}) \cong \PP^1$. 

Now, assume ${\bf v}$ is not primitive. We can write ${\bf v}=N{\bf v'}$ with $N \in \bb{N}$ and ${\bf v'}$ primitive. In this case $M_{H_m}({\bf v})$ consists of strictly semistable sheaves. But since $M_{H_m}({\bf v'})=M_{H_m}^s({\bf v'})$ carries a universal family of $H_m$-stable sheaves $\cc{U}$, we can take its $N$-fold sum to get a  morphism $$M_{H_m}({\bf v'}) \times M_{H_m}({\bf v'}) \times ... \times M_{H_m}({\bf v'}) \to M_{H_m}({\bf v}).$$ This morphism is surjective on closed points and invariant under permutation of factors, i.e. factors through the symmetric product \begin{equation}\label{symmetric} S^N(M_{H_m}({\bf v'})) \to M_{H_m}({\bf v}),\end{equation} which is now bijective at closed points. Note that since $M_{H_m}({\bf v'}) \cong \PP^1$, the symmetric product is just a projective space $$ S^N(M_{H_m}({\bf v'})) \cong \PP^N.$$ 

Now recall that $M_{H_m}({\bf v})$ is a good quotient  $R\sslash G$ of a smooth subvariety $R$ of the Quot scheme. In particular, $R$ is normal. Since normality is preserved under taking categorical quotients (see \cite[Page 5]{MR1304906}), $M_{H_m}({\bf v})$ is normal too. It follows that \eqref{symmetric} is an isomorphism. \end{proof}

From now on, we will be working with characters ${\bf v}$ with $\Delta(\bv)>\frac{1}{2}$ . Note that by Proposition \ref{changing polarization} (1) the stable locus $M_{H_m}^{s}(\bv)$ will be nonempty for $H_m$-semistable Chern characters with $\Delta(\bv)>\frac{1}{2}$.  Applying Theorem \ref{Yoshioka surjectivity}, we know that the Donaldson homomorphism is surjective  $$\lambda: {\bf v}^{\perp} \surj \Pic (M_{H_m}({\bf v})),$$
and we need to study its kernel. 

The next proposition corresponds to cases (1.a) and (1.b) of Theorem \ref{Main theorem}.

\begin{proposition}\label{second prop}Let ${\bf v}=(r, \nu, \Delta)\in K(X)$ be a character with $r \geq 2$ and $\Delta> \frac{1}{2}$. Let $\epsilon \in \bb{Q}$ be sufficiently small (depending on $r$), $0 < |\epsilon| \ll 1$, and set $m=1+\epsilon$.

If    $\Delta - \frac{1}{r} \geq \DLP_{H_m}^{<r}(\nu)$ and either $$\left(\Delta - \frac{1}{r}>\frac{1}{2}\right) \quad \text{or} \quad \left({\bf v'}=\left(r, \nu, \Delta-\frac{1}{r}\right) \ \text{is primitive} \right),$$ then $$\Pic(M_{H_m}({\bf v}))\cong \ZZ^3$$ and $\lambda$ is an isomorphism.
\end{proposition}

\begin{proof}  Since $ \Delta - \frac{1}{r} \geq \DLP_{H_m}^{<r}(\nu)$, Theorem \ref{existence theorem} implies that for ${\bf v'}=(r, \nu, \Delta-\frac{1}{r})$ the moduli space $M_{H_m} ({\bf v'})$ is non-empty. If character ${\bf v}$ is primitive, then as in the proof of Proposition \ref{first prop}  above we get that $M_{H_m}({\bf v'})=M_{H_m}^s({\bf v'})$ and there are $H_m$-stable sheaves of character ${\bf v'}$. If 

$$\Delta({\bf v'})=\Delta-\frac{1}{r}>\frac{1}{2},$$
then Proposition  \ref{changing polarization} (1)  guarantees the existence of $H_m$-stable sheaves. In both cases, Proposition \ref{changing polarization} (2) implies that there are $\mu_{H_m}$-stable sheaves of character ${\bf v'}$.  By the results of Walter discussed in \S \ref{quadric}, we can find a  $\mu_{H_m}$-stable {\it vector bundle} $\cc{V}'$ of character ${\bf v'}$. 

Now, taking elementary modifications of $\VV'$ as described in \cite[Example 8.1.7]{HL10} one can show that $\Pic(M_{H_m}({\bf v}))$ contains $\ZZ\oplus \Pic(X)\cong \ZZ^3$. It follows that $$ \ZZ^3 \cong {\bf v}^{\perp} \stackrel{\lambda}{\surj} \Pic (M_{H_m}({\bf v}))$$is an isomorphism, for if it had a nontrivial kernel, Theorem \ref{Yoshioka surjectivity} would imply that the Picard number $\rho(M_{H_m}({\bf v}))\leq 2$, yielding a contradiction. \end{proof}

The next proposition corresponds to case (3.a) of Theorem \ref{Main theorem}.

\begin{proposition}\label{third prop}Let ${\bf v}=(r, \nu, \Delta)\in K(X)$ be a character with $r \geq 2$ and $\Delta> \frac{1}{2}$. Let $\epsilon \in \bb{Q}$ be sufficiently small (depending on $r$), $0 < |\epsilon| \ll 1$, and set $m=1+\epsilon$. 

If $\Delta=\DLP^{<r}_{H_m}(\nu)>\frac{1}{2}$ with two exceptional bundles $E_1, E_2$ associated to ${\bf v}$ and $$\bv \ \text{is primitive or} \ (E_1,E_2) \ \text{is an exceptional pair},$$ then $$\Pic(M_H({\bf v})) \cong \ZZ$$ and $\lambda$ is an epimorphism with $$\ker \lambda=\ZZ [\overline{E_1}] + \ZZ [\overline{E_2}].$$
\end{proposition}

\begin{proof}  In this case, the discussion in \S \ref{easy kernel} shows that the subgroup $\ZZ [\overline{E_1}] + \ZZ[\overline{E_2}]$ lies in the kernel of the Donaldson homomorphism, which now factors as $$ {\bf v}^{\perp} / \ZZ [\overline{E_1}] \oplus \ZZ[\overline{E_2}] \surj \Pic(M_{H_m}({\bf v})).$$  Since the ample bundle generates a free $\bb{Z}$-submodule inside $\Pic({M_{H_m}}({\bf v}))$, it follows that the Picard number is equal to one $$\rho(M_{H_m}({\bf v}))=1.$$ 

If $\bv$ is primitive, then for a generic choice of $m=\frac{p}{q}$ we have$$\gcd (r, c_1 \cdot (qH_m),\chi)=1.$$ Applying Theorem \ref{Maiorana} we get that $\Pic(M_{H_m}({\bf v}))$ is torsion-free and, therefore, $$\Pic(M_{H_m}({\bf v})) \cong \ZZ.$$

Now assume $(E_1,E_2)$ or $(E_2,E_1)$ forms an exceptional pair. One checks that $(\overline{E}_1, \overline{E}_2)$ or $(\overline{E_2},\overline{E_1})$ is still an exceptional pair. Zyuzina \cite{1994IzMat..42..163Z} shows that any exceptional pair on $\XX$ can be completed to a full exceptional collection.
Since a full exceptional collection forms a $\ZZ$-basis for $K(X)$, we see that  $\ZZ [\overline{E_1}] \oplus \ZZ[\overline{E_2}]$ is a primitive lattice inside ${\bf v}^{\perp} \subset K(X)$. This way, the Donaldson homomorphism induces $$\ZZ \cong {\bf v}^{\perp} / \ZZ [\overline{E_1}] \oplus \ZZ[\overline{E_2}] \surj \Pic(M_{H_m}({\bf v})).$$The result follows. \end{proof}

The arguments above worked equally well for both good and bad $H_m$-semistable Chern characters. However, for the remaining cases (1.c) and (2) of Theorem \ref{Main theorem} the assumption that character ${\bf v}$ is good is essential.  

It will be convenient to separate the proof of case (2) of Theorem \ref{Main theorem} into the following two subcases (keeping the notation and the assumptions of the theorem):

(2.a) \emph{If  ${\bf v}$ is a good character with  $\Delta=\DLP^{<r}_{H_m}(\nu)>\frac{1}{2}$ with a single exceptional bundle $L$ associated to ${\bf v}$ and $r(L)=1$, then $$\Pic(M_{H_m}({\bf v})) \cong \ZZ^2$$ and $\lambda$ is an epimorphism with $$\ker \lambda =\ZZ [\overline{L}].$$}

(2.b) \emph{If ${\bf v}$ is a good character with $\Delta=\DLP^{<r}_{H_m}(\nu)>\frac{1}{2}$ with a single exceptional bundle $E$ associated to ${\bf v}$ and $r(E)>1$, then $$\Pic(M_{H_m}({\bf v})) \cong \ZZ^2$$ and $\lambda$ is an epimorphism with $$\ker \lambda =\ZZ [\overline{E}].$$}

We prove the case (2.a) first. 
\begin{proposition}\label{fourth prop}Let ${\bf v}=(r, \nu, \Delta)\in K(X)$ be a character with $r \geq 2$ and $\Delta> \frac{1}{2}$. Let $\epsilon \in \bb{Q}$ be sufficiently small (depending on $r$), $0 < |\epsilon| \ll 1$, and set $m=1+\epsilon$.

If ${\bf v}$ is a good character with $\Delta=\DLP^{<r}_{H_m}(\nu)$ with a single exceptional bundle $L$ associated to ${\bf v}$ and $r(L)=1$, then $$\Pic(M_{H_m}({\bf v})) \cong \ZZ^2$$ and $\lambda$ is an epimorphism with $$\ker \lambda =\ZZ [\overline{L}].$$

\end{proposition}

\begin{proof}We treat the case $$\mu_{H_m}({\bf v}) \leq \mu_{H_m}(L)$$and say how to modify the argument in the other case at the end of the proof.

In this situation, the general $H_m$-semistable sheaf $\VV$ admits an $L$-Gaeta type resolution with exponents        \begin{equation*}
  \begin{aligned}
  \alpha=-\chi({\bf v}\otimes L^{\dual}(-1,-1)) > 0 \\
  \beta=-\chi({\bf v}\otimes L^{\dual}(-1,0)) > 0 \\
  \gamma=-\chi({\bf v}\otimes L^{\dual}(0,-1)) > 0 \\
  \delta=\chi({\bf v}\otimes L^{\dual}) = 0   
  \end{aligned} 
\end{equation*}Note that none of $\alpha, \beta, \gamma$ can be equal to $0$. For otherwise, one checks using the semistability of $\VV$ that one of the bundles $L(1,1), L(1,0), L(0,1)$ or their Serre twists would also be associated to ${\bf v}$,  contradicting our assumption.  

Consider the family $\VV_t/T$ of  $\OO(1,1)$-prioritary sheaves admitting  the $L$-Gaeta type resolution \begin{equation}\label{seq} 0 \to L(-1,-1)^{\alpha} \stackrel{\psi_t}{\to} L(-1,0)^{\beta} \oplus L(0,-1)^{\gamma} \to \VV_t \to 0,\end{equation}where  $$T \subset \mathbb{H}=\Hom \left(L(-1,-1)^{\alpha},  L(-1,0)^{\beta} \oplus L(0,-1)^{\gamma} \right)$$ is the open subset parameterizing injective sheaf maps with torsion-free cokernel from Proposition \ref{Gaeta parameterizations}. Since we are assuming that ${\bf v}$ is a good character, there is no $\Delta_i=\frac{1}{2}$ strata in this family by Proposition \ref{Shatz2}. The other potential divisorial Shatz stratum should consist of sheaves $\VV_t$ admitting the $H_m$-Harder-Narasimhan filtration $$0 \subset L \subset \VV_t.$$ Applying $\Hom(L, \_)$ to the short exact sequence \eqref{seq}, we see that $$\Hom(L,\VV_t)=0$$ for all $\psi_t \in T$, so this potential stratum is empty. Combined with Proposition \ref{Gaeta parameterizations}, we get that $$\codim_T \ (T \setminus T^{ss})\geq 2.$$ 

Functorial properties of the Donaldson homomorphism from Lemma \ref{computation} give the following commutative diagram (recall our notation from \eqref{missing exponent} and Proposition \ref{Donaldson on Gaeta 2}):
\begin{center}
\begin{tikzcd}
{\bf v}^{\perp} \arrow[d, hook] \arrow[rr,two heads, "\lambda"]& & \Pic(M_{H_m}({\bf v})) \arrow[d,"\phi_{\cc{V}_t|_{T^{ss}}}^*"] \\
K(X) \arrow[rrd, two heads, "\lambda_{\cc{V}_t}"'] \arrow[rr, two heads, "\lambda_{\cc{V}_t|_{T^{ss}}}"]  & & \Pic^{G_{\hat{\delta}}}(T^{ss})  \\
     &   & \Pic^{G_{\hat{\delta}}}(T).    \arrow[u,"\cong","res"']              
\end{tikzcd} 
\end{center}  
By Proposition \ref{Donaldson on Gaeta 2}  we know that the image of $(res^{-1} \circ \phi_{\cc{V}_t|_{T^{ss}}}^*)$ is $$\Char(\overline{G}_{\hat{\delta}})\subset \Char(G_{\hat{\delta}}) \subset \Pic^{G_{\hat{\delta}}}(T),$$which is a free $\ZZ$-module of rank $2$. On the other hand, by the discussion in \S \ref{easy kernel} we know that $[\overline{L}]$ lies in the kernel of the Donaldson homomorphism $\lambda$. Putting these together, we get that $\phi_{\cc{V}_t|_{T^{ss}}}^* \circ \lambda$ factors as $$\ZZ^2 \cong{\bf v}^{\perp}/\ZZ[\overline{L}] \surj \Pic(M_{H_m}({\bf v})) \surj \ZZ^2, $$so  both maps are isomorphisms and $\Pic(M_{H_m}({\bf v})) \cong \ZZ^2$.

In the other case when  $$\mu_{H_m}({\bf v}) >\mu_{H_m}(L),$$ one modifies the above proof by using the dual version of a Gaeta-type resolution $$0 \to \VV \to L(1,0)^{\beta}\oplus L(0,1)^{\gamma} \to L(1,1)^{\delta} \to 0$$and replacing Propositions \ref{Gaeta parameterizations} and \ref{Donaldson on Gaeta 2} by Proposition \ref{dual Gaeta family} and the dual version of Proposition \ref{Donaldson on Gaeta 2}.  \end{proof} 

At this point let us make a couple of useful observations. First, note that the previous cases fully establish Theorem \ref{Main theorem} for characters ${\bf v}=(r,\nu,\Delta)$ with  $r=2$.
Indeed, $(r=2, \Delta=\frac{1}{2})$ was covered by Proposition \ref{first prop} and for $(r=2, \Delta>\frac{1}{2})$  we have four different cases:
\begin{enumerate}
\item ${\bf v}=(2,\nu,\Delta)=(2,\varepsilon E+ \varphi F,\Delta)$ with $ \varepsilon, \varphi \in \ZZ$ and $$\Delta>\DLP^{<2}_{H_m}(\varepsilon E+ \varphi F)=\DLP_{H_m, \OO(\varepsilon, \varphi)}(\varepsilon E+ \varphi F)=1. $$ The inequality implies that $\Delta\geq \frac{3}{2}$. Therefore,$$\Delta-\frac{1}{2} \geq \DLP_{H_m, \OO(\varepsilon, \varphi)}(\varepsilon E+ \varphi F)=1$$and  this is covered by Proposition \ref{second prop}. 
\item ${\bf v}=(2,\nu,\Delta)=(2,\varepsilon E+ \varphi F,\Delta)$ with $ \varepsilon, \varphi \in \ZZ$ and $$\Delta=\DLP^{<2}_{H_m}(\varepsilon E+ \varphi F)=\DLP_{H_m, \OO(\varepsilon, \varphi)}(\varepsilon E+ \varphi F)=1.$$The line bundle $\OO(\varepsilon, \varphi)$ is the only exceptional bundle associated to ${\bf v}$ in this case. Since characters ${\bf v}$ with $r=2$ are always good characters (see Definition \ref{bad definition}),  this case is covered by Proposition \ref{fourth prop} above.  

\item ${\bf v}=(2,\nu,\Delta)=(2,\varepsilon E+ \varphi F,\Delta)$ with $\varepsilon \in (\ZZ[\frac{1}{2}] \setminus \ZZ)$, or $\varphi  \in (\ZZ[\frac{1}{2}] \setminus \ZZ)$ (or both) and $$\Delta>\DLP^{<2}_{H_m}(\varepsilon E+ \varphi F). $$ As before, one shows that then $$\Delta-\frac{1}{2} \geq \DLP^{<2}_{H_m}(\varepsilon E+ \varphi F) \quad \text{and} \quad \DLP^{<2}_{H_m}(\varepsilon E+ \varphi F) \geq \frac{1}{2}.$$ In case one of the inequalities is strict, we have $\Delta-\frac{1}{2}>\frac{1}{2}$. If $\Delta-\frac{1}{2}=\frac{1}{2}$, then ${\bf v}'=(2, \nu, \frac{1}{2})=(2, c_1, \chi)$ is primitive, because by our assumption $c_1=(2\varepsilon) E+(2\varphi )F$ has an odd component  and therefore is not divisible by $2$. We conclude that this case is covered by Proposition \ref{second prop} above.

\item ${\bf v}=(2,\nu,\Delta)=(2,\varepsilon E+ \varphi F,\Delta)$ with $\varepsilon \in (\ZZ[\frac{1}{2}] \setminus \ZZ)$, or $\varphi  \in (\ZZ[\frac{1}{2}] \setminus \ZZ)$ (or both) and $$\Delta=\DLP^{<2}_{H_m}(\varepsilon E+ \varphi F).$$ From Figure \ref{fig:example}, one sees that ${\bf v}$ lies on two branches of the $\DLP^{<2}_{H_m}$-surface $$\Delta=\DLP_{H_m, L_1}(\varepsilon E+ \varphi F)=\DLP_{H_m, L_2}(\varepsilon E+ \varphi F)=\frac{3}{4}$$with $L_1=\OO((\floor{\varepsilon}+1)E+\floor{\varphi}F)), L_2=\OO(\floor{\varepsilon}E+(\floor{\varphi}+1)F)$, so  there are two line bundles associated to ${\bf v}$ and they form an exceptional pair. This is covered by Proposition \ref{third prop} above.
\end{enumerate}

It remains to prove statements (1.c) and (2.b). By the above discussion we can assume that $r({\bf v}) \geq 3$.  We show next that for the remaining characters ${\bf v}=(r, \nu, \Delta)$ there is a strong upper bound on the discriminant $\Delta$.

\begin{lemma} Let ${\bf v}=(r, \nu, \Delta)\in K(X)$ be a character with $r \geq 3$ and $\Delta> \frac{1}{2}$. Let $\epsilon \in \bb{Q}$ be sufficiently small (depending on $r$), $0 < |\epsilon| \ll 1$, and set $m=1+\epsilon$. 

If ${\bf v}$ is a character satisfying either

\begin{itemize}
\item $ \Delta>\DLP^{<r}_{H_m}(\nu)$, and $$\left(\Delta-\frac{1}{r} <\DLP^{<r}_{H_m}(\nu)\right) \ \text{or} \ \left( \Delta-\frac{1}{r} \leq \frac{1}{2}\right),$$
\end{itemize}
or

\begin{itemize}
\item $\Delta=\DLP^{<r}_{H_m}(\nu)$ with a single exceptional bundle $E$ associated to ${\bf v}$ with $r(E)>1$,

\end{itemize}
then $\Delta<1.$

\end{lemma}

\begin{proof}

Suppose first $$\Delta-\frac{1}{r} \leq \frac{1}{2}.$$
We immediately get $$\Delta\leq \frac{1}{2}+\frac{1}{r}\leq \frac{1}{2}+\frac{1}{3}=\frac{5}{6}<1 \quad (\text{recall} \ r\geq 3).$$

Now suppose that $$\Delta>\DLP^{<r}_{H_m}(\nu)  \quad \text{and} \quad \left(\Delta-\frac{1}{r} <\DLP^{<r}_{H_m}(\nu)\right).$$ Note that then $\DLP^{<r}_{H_m}(\nu)$ cannot be equal to $\DLP_{H_m, L}(\nu)$ for a line bundle $L$. Indeed,  $$\Delta>\DLP_{H_m, L}(\nu)$$ is equivalent to $$ r(P(\nu - \nu(L))-\Delta)<0 \ \text{or} \ \ r(P(\nu(L) - \nu)-\Delta)<0.$$This implies $$r(P(\nu - \nu(L))-\Delta+\frac{1}{r})\leq0 \ \text{or} \ \ r(P(\nu(L) - \nu)-\Delta+\frac{1}{r})\leq0, $$which is equivalent to $$ \Delta - \frac{1}{r} \geq \DLP^{<r}_{H_m, L}(\nu),$$contradicting our assumption. Thus $\DLP^{<r}_{H_m}(\nu)=\DLP_{H_m, E}(\nu)$ with $r(E)>1$. By Lemma \ref{exceptional stuff} (1) \begin{equation}\label{est10} r(E)\geq 3 \quad \text{and} \quad  \DLP_{H_m, E}(\nu)\leq \DLP_{H_m, E}(\nu(E))=1-\Delta(E)=\frac{1}{2}+\frac{1}{2r(E)^2}\leq \frac{5}{9}, \end{equation}which  implies $$\Delta<\DLP^{<r}_{H_m}(\nu)+\frac{1}{r}=\DLP_{H_m, E}(\nu)+\frac{1}{r}\leq \frac{5}{9}+\frac{1}{3}=\frac{8}{9}<1.$$

Finally, suppose character ${\bf v}=(r,\nu,\Delta)$ satisfies  $$\Delta=\DLP^{<r}_{H_m}(\nu)=\DLP_{H_m, E}(\nu),$$where $r(E) \geq 3$. From \eqref{est10} we get $$\Delta=\DLP_{H_m, E}(\nu) \leq \frac{5}{9}<1.$$
\end{proof}

Let us make one last observation: for the remaining cases we can assume that in $\nu({\bf v})=\varepsilon E + \varphi F$ we have $$\varepsilon+\varphi\neq \floor{\varepsilon}+\floor{\varphi}+1.$$Indeed, otherwise the total slope $\nu$ would lie on the "antidiagonal" in the $(\varepsilon, \varphi)$-plane of total slopes along which the $\DLP$-surface is given by Line bundles (see Figure \ref{fig:example}). This is covered by Propositions \ref{second prop}, \ref{third prop} and \ref{fourth prop}. 
 
So in the proof of the remaining cases (1.c) and (2.b) of the theorem we can assume that $\bv=(r,\nu,\Delta)$ satisfies
\begin{equation}{\label{assumptions}} \begin{aligned} 
      (1)&  &r &\geq 3, \\
       (2)& &\frac{1}{2}<& \ \Delta<1, \\
        (3)& &(\floor{\varepsilon}+\floor{\varphi}\leq \varepsilon+\varphi< \floor{\varepsilon}+\floor{\varphi}+1) \ &\text{or} \ (\floor{\varepsilon}+\floor{\varphi}+1< \varepsilon+\varphi< \floor{\varepsilon}+\floor{\varphi}+2), \\
        (4)& &\text{There is no line} \ &\text{bundle associated to} \ {\bf v}. \end{aligned}\end{equation}

These assumptions make it possible to resolve a general $H_m$-semistable sheaf of character $\bv$ via  a Gaeta-type resolution with all exponents nonzero. \begin{lemma}\label{aux lemma}Let ${\bf v}=(r, \nu, \Delta)=(r, \varepsilon E+\varphi F, \Delta)\in K(X)$ be a character satisfying conditions \eqref{assumptions}. Let $\epsilon \in \bb{Q}$ be sufficiently small (depending on $r$), $0 < |\epsilon| \ll 1$, and set $m=1+\epsilon$. 

 If $$\floor{\varepsilon}+\floor{\varphi}\leq \varepsilon+\varphi< \floor{\varepsilon}+\floor{\varphi}+1,$$then a general $H_m$-semistable sheaf $\VV$ of character ${\bf v}$ admits an $L$-Gaeta-type resolution $$0 \to L(-1,-1)^{\alpha} \to L(-1,0)^{\beta} \oplus L(0,-1)^{\gamma} \oplus L^{\delta} \to \VV \to 0$$ where $L:=L_{\floor{\epsilon},\floor{\phi}}$ and all the exponents are nonzero.

If $$\floor{\phi}+1< \epsilon+\phi< \floor{\epsilon}+\floor{\phi}+2,$$then a general $H_m$-semistable sheaf $\VV$ of character ${\bf v}$ admits a dual $L$-Gaeta-type resolution $$0 \to \VV \to   L(1,0)^{\alpha}  \oplus L(0,1)^{\beta} \oplus L ^{\gamma}   \to   L (1, 1)^{\delta} \to 0$$ where $L:=L_{\floor{\epsilon},\floor{\phi}}$ and all the exponents are nonzero.
\end{lemma}

\begin{proof} We prove the case \begin{equation}\label{cond}\floor{\varepsilon}+\floor{\varphi}\leq \varepsilon+\varphi< \floor{\varepsilon}+\floor{\varphi}+1,\end{equation} and say how to modify the argument for the other case at the end of the proof.

 With the notation of \S \ref{Gaeta section}, we use the bound \eqref{assumptions} (2) and formally compute using Riemann-Roch: $$\chi({\bf v}(-L_{\varepsilon, \varphi}))=r(1-\Delta)>0.$$Thus, in the $(a,b)$-plane $\bb{R}^2$ the point $(\varepsilon, \varphi)$ lies below the lower-left branch $Q$ of the hyperbola $\chi({\bf v}(-L_{a,b}))=0$. Therefore, the integral point $(\floor{\varepsilon}, \floor{\varphi})$ also lies below $Q$ and we have that \begin{equation*}{\label{k l zero}}\chi({\bf v}(-L_{\floor{\varepsilon},\floor{\varphi}}))>0.\end{equation*} 

For a sufficiently small $\epsilon$ condition \eqref{cond} translates into a condition on $\mu_{H_m}$-slopes: $$\mu_{H_m}(L_{\floor{\varepsilon},\floor{\varphi}})\leq\mu_{H_m}({\bf v})<\mu_{H_m}(L_{\floor{\varepsilon}+k,\floor{\varphi}+l})+1 \ \text{for} \ (k,l)\in \{(1,0),(0,1),(1,1)\}.$$Therefore, for $\cc{V} \in M_{H_m}({\bf v})$  we have $$\Hom(L_{\floor{\varepsilon}+k,\floor{\varphi}+l},\cc{V})=\Ext^2(L_{\floor{\varepsilon}+k,\floor{\varphi}+l},\cc{V})=0,$$resulting in $$-\chi({\bf v}(-L_{\floor{\varepsilon}+k,\floor{\varphi}+l}))=-\chi(L_{\floor{\varepsilon}+k,\floor{\varphi}+l},{\bf v})=\ext^1(L_{\floor{\varepsilon}+k,\floor{\varphi}+l},\cc{V})\geq0.$$In fact, the inequalities are strict, for otherwise $L_{\floor{\varepsilon}+k,\floor{\varphi}+l}$ (or their Serre twists) would be associated to ${\bf v}$ contradicting assumption \eqref{assumptions} (4).  This shows that the line bundle $L:=L_{\floor{\varepsilon},\floor{\varphi}}$ satisfies \eqref{Gaeta equation} of Theorem \ref{Gaeta theorem} with all  integers $\alpha, \beta, \gamma,\delta$ being nonzero. 

In the case $$\floor{\varepsilon}+\floor{\varphi}+1< \varepsilon+\varphi< \floor{\epsilon}+\floor{\phi}+2,$$
one can first pass to the dual character $\bv'$, resolve a generic $\mu_{H_m}$-stable vector bundle by a Gaeta-type resolution with all exponents nonzero as above, and then take the dual of the whole resolution. Here we use Proposition \ref{changing polarization} to guarantee the existence $\mu_{H_m}$-stable bundles.
\end{proof}

Thus, to study $M_{H_m}({\bf v})$ for $\bv$ satisfying conditions \eqref{assumptions} with $$\floor{\varepsilon}+\floor{\varphi}\leq \varepsilon+\varphi< \floor{\varepsilon}+\floor{\varphi}+1$$we consider the complete family $\VV_t/T$ of $\OO(1,1)$-prioritary sheaves admitting an $L$-Gaeta type resolution \begin{equation}{\label{working Gaeta}} 0 \to L(-1,-1)^{\alpha} \stackrel{\psi_t}{\to} L(-1,0)^{\beta} \oplus L(0,-1)^{\gamma} \oplus L^{\delta} \to \VV_t \to 0,\end{equation}where $L=L_{\floor{\varepsilon},\floor{\varphi}}$,  all the exponents are nonzero, and \begin{equation}\label{family}T\subset \mathbb{H}=\Hom \left(L(-1,-1)^{\alpha},  L(-1,0)^{\beta} \oplus L(0,-1)^{\gamma} \oplus L^{\delta} \right)\end{equation} is the open subset parameterizing injective sheaf maps with torsion-free cokernel. By Proposition \ref{Gaeta parameterizations} \begin{equation}\label{codim} \codim_{\bb{H}} (\bb{H} \setminus T)\geq 2.\end{equation}

Likewise, to study $M_{H_m}({\bf v})$ for $\bv$ satisfying conditions \eqref{assumptions} with $$\floor{\varphi}+1< \varepsilon+\phi< \floor{\varepsilon}+\floor{\varphi}+2$$we consider the complete family $\VV_t/T$ of $\OO(1,1)$-prioritary vector bundles admitting the dual version of an $L$-Gaeta type resolution \begin{equation*} 0 \to  \VV_t \to  L(1,0)^{\alpha} \oplus L(0,1)^{\beta} \oplus L^{\gamma} \xrightarrow{\psi_t} L(1,1)^{\delta}\to 0,\end{equation*}where $L=L_{\floor{\varepsilon},\floor{\varphi}}$,  all the exponents are nonzero, and \begin{equation*}T\subset \mathbb{H}=\Hom \left(L(1,0)^{\alpha} \oplus L(0,1)^{\beta} \oplus L^{\gamma} ,  L(1,1)^{\delta} \right)\end{equation*} is the open subset parameterizing surjective sheaf maps. By the dual version of Proposition \ref{dual Gaeta family} \begin{equation*}\codim_{\bb{H}} (\bb{H} \setminus T)\geq 2.\end{equation*}

Next, we analyze the Shatz stratification of $T$ and apply the irreducibility results of \S \ref{irreducible families} to compute the kernel of the Donaldson homomorphism in the remaining cases.  In what follows, we treat the case   $$\floor{\varepsilon}+\floor{\varphi}\leq \varepsilon+\varphi< \floor{\varepsilon}+\floor{\varphi}+1,$$ leaving the necessary modifications of the proof in the other case to the reader.
 
The next proposition proves case (1.c) of Theorem \ref{Main theorem}.

\begin{proposition}\label{fifth prop} Let ${\bf v}=(r, \nu, \Delta)\in K(X)$ be a character with $r \geq 3$ and $\Delta> \frac{1}{2}$. Let $\epsilon \in \bb{Q}$ be sufficiently small (depending on $r$), $0 < |\epsilon| \ll 1$, and set $m=1+\epsilon$.

If ${\bf v}$ is a good character with  $$\Delta>  \DLP^{<r}_{H_m}(\nu)$$ and $$\left(\Delta-\frac{1}{r} <\DLP^{<r}_{H_m}(\nu)\right) \ \text{or} \ \left(\Delta-\frac{1}{r} \leq \frac{1}{2}\right),$$ then $$\Pic(M_H({\bf v}))\cong \ZZ^3$$ and $\lambda$ is an isomorphism.

\end{proposition}

\begin{proof} By the discussion above, we can assume that character ${\bf v}$ satisfies \eqref{assumptions} with  $$\floor{\varepsilon}+\floor{\varphi}\leq \varepsilon+\varphi< \floor{\varepsilon}+\floor{\varphi}+1.$$
Consider the family $T$ from \eqref{family}. Proposition \ref{no divisorial Shatz strata} says that potential divisorial Shatz strata of $T$ are given by the $\Delta_i=\frac{1}{2}$-strata. Since we are assuming that ${\bf v}$ is a good Chern character, it follows that there is no Shatz strata of codimension $1$ in $T$. Therefore, by \eqref{codim} the semistable locus $T^{ss}\subset T \subset \bb{H}$ satisfies \begin{equation*}{\label{codimension}}\codim_{\bb{H}}(\bb{H} \setminus T^{ss}) \geq 2.\end{equation*} 

The Donaldson morphism fits into the following commutative diagram 
\begin{center}
\begin{tikzcd}
{\bf v}^{\perp} \arrow[d, hook] \arrow[rr,two heads, "\lambda"]& & \Pic(M_H({\bf v})) \arrow[d,"\phi_{\cc{V}_t|_{T^{ss}}}^*"] \\
K(X) \arrow[rrd, "\lambda_{\cc{V}_t}"'] \arrow[rr, "\lambda_{\cc{V}_t|_{T^{ss}}}"]  & & \Pic^G(T^{ss})  \\
     &   & \Pic^G(T).    \arrow[u,"\cong","res"']              
\end{tikzcd} 
\end{center}
By Proposition \ref{Donaldson on Gaeta}, the bottom map is an isomorphism and it follows that $\lambda$ is injective. \end{proof}

Finally, we finish proving Theorem \ref{Main theorem} by considering the last remaining case (2.b).

\begin{proposition}\label{sixth prop}Let ${\bf v}=(r, \nu, \Delta)\in K(X)$ be a character with $r \geq 3$ and $\Delta> \frac{1}{2}$. Let $\epsilon \in \bb{Q}$ be sufficiently small (depending on $r$), $0 < |\epsilon| \ll 1$, and set $m=1+\epsilon$.

If ${\bf v}$ is a good character with $$\Delta=\DLP^{<r}_{H_m}(\nu)$$ with a single exceptional bundle $E$ associated to ${\bf v}$ and $r(E)>1$, then $$\Pic(M_{H_m}({\bf v})) \cong \ZZ^2$$ and $\lambda$ is an epimorphism with  $$\ker \lambda=\ZZ [\overline{E}].$$

\end{proposition}

\begin{proof}  By the discussion above, we can again assume that character ${\bf v}$ satisfies \eqref{assumptions} with  $$\floor{\varepsilon}+\floor{\varphi}\leq \varepsilon+\varphi< \floor{\varepsilon}+\floor{\varphi}+1.$$We further assume $$ \mu_{H_m}({\bf v})\leq\mu_{H_m}(E),$$ leaving the necessary modifications of the proof in the other case to the reader. By the discussion in \S \ref{easy kernel}, we know that $[\overline{E}]$ lies in the kernel of the Donaldson homomorphism. We show that in fact $$\ker(\lambda)=\bb{Z} [E].$$

\emph{Step 1}. Once again we analyze the Shatz stratification of the family $\VV_t/T$ from \eqref{family} which parameterizes $\O(1,1)$-prioritary sheaves admitting an $L$-Gaeta type resolution \eqref{working Gaeta}: \begin{equation}\tag{\ref{working Gaeta}}0 \to L(-1,-1)^{\alpha} \stackrel{\psi_t}{\to} L(-1,0)^{\beta} \oplus L(0,-1)^{\gamma} \oplus L^{\delta} \to \VV_t \to 0.\end{equation}This time, Proposition \ref{Shatz2} says that there is at most one possible divisorial Shatz stratum $S_T$ since $\Delta_i=\frac{1}{2}$-strata are excluded by the assumption that ${\bf v}$ is good. In fact, it must be nonempty for our family $\VV_t/T$. For otherwise arguing as in the Proposition \ref{fifth prop}, we would show that $\lambda$ is injective, in contradiction to $\bb{Z}[\overline{E}] \subset \ker(\lambda).$ 

According to Proposition \ref{Shatz2}, this stratum $S_T=S_{T,H_m}({\bf v}_1,{\bf v}_2)$ consists of points $\psi_t\in T$ such that the corresponding sheaf $\VV_t$ admits the $H_m$-Harder-Narasimhan filtration of length $l=2$ \begin{equation}{\label{bad stratum}} 0 \subset E \subset \cc{V}_t,\end{equation} where $E$ is the exceptional bundle associated to ${\bf v}$ and ${\bf v}_1:={\bf v}(E), {\bf v}_2:={\bf v}(\cc{V}_t/E)$. 

\emph{Step 2}. We claim that this Shatz stratum is irreducible. 
Set $B= L(-1,0)^{\beta} \oplus L(0,-1)^{\gamma} \oplus L^{\delta}$. Consider the Quot scheme $Quot(B, {\bf v})$ parameterizing quotients $$q=[B \surj \cc{E}_q], \quad q \in Quot(B,{\bf v}),$$where ${\bf v}(\cc{E}_q)={\bf v}$. First, restrict the family $\{\cc{E}_q \}$ to the open subset of $Quot(B,{\bf v})$ parameterizing torsion-free  $\cc{E}_q$. Over this subset we have a universal family of sheaves $$0 \to K \to q^*B \to \cc{E} \to 0.$$ Note that by \eqref{working Gaeta} $${\bf v}(K_q)={\bf v}(L(-1,-1)^{\alpha}).$$We further restrict to the open subset $Q \subset Quot(B,{\bf v})$ parameterizing those quotients $q$ for which $K_q$ is a semistable vector bundle. Since ${\bf v}(L(-1,-1)^{\alpha})$ is a semiexceptional Chern character, $$K_q \cong L(-1,-1)^{\alpha} \ \ \text{for each} \ q \in Q.$$ 

By the universal property of Quot schemes the family of quotients $\{ B \to \cc{V}_t \}_{\psi_t \in  T}$ of \eqref{working Gaeta} gives a surjective morphism $$T \stackrel{\Omega}{\surj} Q \subset Quot(B, {\bf v}),$$whose fibers are isomorphic to $GL(\alpha)$. Denote the Shatz stratum of  points $q \in Q$ such that the corresponding $\cc{E}_q$ has the $H_m$-Harder-Narasimhan filtration $$0 \subset E \subset \cc{E}_q$$ by  $S_Q=S_{Q,H_m}({\bf v}_1,{\bf v}_2)$. From the Cartesian diagram 
\begin{center}
\begin{tikzcd}
S_T \arrow[r, hook] \arrow[d, two heads] & T \arrow[d,two heads, "\Omega"] \\
S_Q \arrow[r, hook]           & Q,         
\end{tikzcd}
\end{center}
it follows that  the irreducibility of $S_T$ is equivalent to the irreducibility of $S_Q$. Indeed,  $S_T$ is equidimensional and the fibers of $S_T \surj S_Q$ are all irreducible and isomorphic to $GL(\alpha)$. We can apply the following version of the irreducibility criterion: if  $Y \to X$ is a finite type  surjective morphism from an equidimensional Noetherian $\bb{C}$-scheme to an irreducible Noetherian $\bb{C}$-scheme, and all fibers  over the closed points are irreducible of the same dimension, then $Y$ is irreducible.

To show the irreducibility of $S_Q$, consider the family $\cc{W}_s/S$ over irreducible $S$ having Property (P) with respect to $\bb{F}({\bf v}_1,{\bf  v}_2)$ that was constructed in Proposition \ref{Yoshioka irreducibility} (we denote the sheaves in this family by $\cc{W}_s$ instead of $\VV_s$ to avoid confusion with the sheaves $\VV_t$ from \eqref{working Gaeta}). Recall that heuristically $\cc{W}_s/S$ parameterizes all torsion-free sheaves whose $H_m$-Harder Narasimhan filtration is of lenght $2$ and has quotients of characters ${\bf v}_1,{\bf  v}_2$, possibly with repetition. Intuitively, we are going to build a family of quotients over an irreducible base out of $\cc{W}_s/S$ that will surject onto $S_Q$ under the universal morphism to the Quot scheme $Q$.  

Note that for $\psi_t\in T$ the Gaeta-type resolution \eqref{working Gaeta} implies $$\Ext^i(B, \cc{V}_t)=0 \ \text{for} \ i >0 \implies \hom(B, \cc{V}_t)=\chi(B, {\bf v}).$$
Thus, consider the open subset $U \subset S$ parameterizing those $\cc{W}_s$ for which $$\Ext^i(B, \cc{W}_s)=0 \ \text{for} \ i>0.$$ It is non-empty because we concluded above that $S_T$ is non-empty, and irreducible. It follows that $$(p_U)_* \sheafhom(q^* B, \cc{W})$$is a vector bundle on $U$. Denote the corresponding geometric vector bundle by $$\bb{V} \stackrel{\pi}{\to} U,$$ so that over $\bb{V} \times X$ we have a universal morphism $$\pi^*q^*B \stackrel{\Psi}{\to} \pi^*\cc{W}.$$ 

We further restrict to an open subset $\bb{U} \subset \bb{V}$ parameterizing surjective maps with an $H_m$-semistable kernel, so that for $u \in \bb{U}$ we have an exact sequence $$0 \to L(-1,-1)^{\alpha} \to B \stackrel{\Psi_u}\to \cc{W}_{\pi(u)} \to 0.$$  By the universal property of Quot schemes, we obtain a map $$\bb{U} \to Q,$$whose image is equal to $S_Q$ because of the Property (P). Since $\bb{U}$ is irreducible, it follows that $S_Q$ is irreducible too.  We summarize the discussion in the following diagram
\begin{center}
\begin{tikzcd}
                       & GL(\alpha)                 & GL(\alpha) \\
                       & S_T \arrow[d, two heads] \arrow[loop above,distance=30]  \arrow[r, hook] & T \arrow[d, two heads] \arrow[loop above,distance=30] \\
\bb{U} \arrow[r, two heads] & S_Q \arrow[r, hook]           & Q.          
\end{tikzcd}
\end{center}

\emph{Step 3}. We return to the problem of describing the kernel of the Donaldson homomorphism. Let $T'=T \setminus \overline{S_T}$ and note that this is a $G$-invariant  open subset of $T$. Because of the irreducibility proved at the previous step, $\overline{S_T}=V(f)$ for some irreducible polynomial \begin{equation}{\label{poly}}f \in \bb{C}[\{x_{ij}\}],\end{equation} where $\bb{C}[\{x_{ij}\}]$ is the coordinate algebra of $\bb{H}$. The sequence \eqref{linearized line bundles} for $Y=T'$ \begin{equation}{\label{eta}}\{a f^k \ | \ a \in \bb{C}^*, k\in \bb{Z}\}=\OO^*(T') \xrightarrow{af^k \mapsto k \eta_f} \Char(G) \to \Pic^G(T') \to 0\end{equation}now implies that $$\Pic^G(T')\cong\Char(G) / \bb{Z}\cdot\eta_{f}.$$
Since $S_T$ was the only divisorial Shatz stratum, we have that $$\codim_{T'} (T' \setminus T^{ss}) \geq 2$$and we obtain the following commutative diagram

\begin{center}
\begin{tikzcd}
{\bf v}^{\perp} \arrow[d, hook] \arrow[rr,two heads, "\lambda"]& & \Pic(M_H({\bf v})) \arrow[d,"\phi_{\cc{V}_t|_{T^{ss}}}^*"] \\
K(X)  \arrow[rr, "\lambda_{\cc{V}_t|_{T^{ss}}}"] \arrow[d, equal]  & & \Pic^G(T^{ss})  \\
 K(X)  \arrow[rr, "\lambda_{\cc{V}_t|_{T'}}"] \arrow[d, equal] &   & \Pic^G(T').    \arrow[u,"\cong","res"']      \\
 
 K(X) \arrow[rr, "\lambda_{\cc{V}_t}","\cong"']  & & \Pic^G(T).  \arrow[u,two heads, "res"'] \end{tikzcd} 
\end{center}

Chasing this diagram shows that integer multiples of $[\overline{E}] \in {\bf v}^{\perp}$ are the only elements in the kernel of $\lambda$. For if there was ${\bf u} \in ({\bf v}^{\perp} \setminus \bb{Z} [\overline{E}])$ with $$\lambda({\bf u})=0,$$then going around the outer lower  part of the diagram would imply that two $\bb{Z}$-linearly independent elements, $\lambda_{\cc{V}_t}({\bf u})$ and $\lambda_{\cc{V}_t}([\overline{E}])$, lie in the kernel of the restriction $$\Pic^G(T) \stackrel{res}{\longrightarrow} \Pic^G(T').$$ But this then contradicts the fact that the kernel of this restriction is a cyclic subgroup, that could be seen from looking at sequence \eqref{linearized line bundles} for the inclusion $T'\subset T$:
 \begin{center}
\begin{tikzcd}
\bb{C}^*=\OO^*(T) \arrow[rr,"0"] \arrow[d, hook] & &\Char(G) \arrow[r] \arrow[d, equal] & \Pic^G(T)  \arrow[r] \arrow[d, two heads, "res"] & 0 \\
\{a f^k \ | \ a \in \bb{C}^*, k\in \bb{Z}\}=\OO^*(T') \arrow[rr,"af^k \mapsto k\eta_f"] & & \Char(G) \arrow[r]  & \Pic^G(T')  \arrow[r] & 0. 
\end{tikzcd}
\end{center}

This finishes the proof of statement (3.b) of the theorem.  \end{proof}

The main Theorem \ref{Main theorem} is now fully proved.

\begin{remark}\label{explicit character} Note that we can describe the polynomial $f$ appearing in \eqref{poly} in such a way so that one can explicitly compute the character  $\eta_f$ appearing in  \eqref{eta}. Recall that for $\psi_t \in S_T$ the corresponding sheaf $\cc{V}_t$ comes equipped with a filtration\begin{equation*} 0 \subset E \subset \cc{V}_t,\end{equation*} while for an $H_m$-semistable $\cc{V}_{\tau}$ we have \begin{equation*}\Hom(E, \cc{V}_{\tau})=\Ext^1(E, \VV_{\tau})=0.\end{equation*} Therefore, $S_T \subset \{ \psi_t\in T | \ \Hom(E, \cc{V}_t)\neq 0 \}.$ The long exact sequence in cohomology coming from \eqref{working Gaeta} $$0 \to \Hom(E, \VV_t) \to \Ext^1(E, L(-1,-1)^{\alpha}) \xrightarrow{(\psi_t)_*}  \Ext^1(E, B) \to \Ext^1(E, \cc{V}_t) \to 0$$ shows that $\{ \psi_t \in T | \ \Hom(E, \cc{V}_t)\neq 0 \}$ is a determinantal divisor given as the vanishing locus of $$\psi_t \mapsto \det((\psi_t)_*).$$ 

As $\overline{S_T}=\{\psi_t \in T | \ \Hom(E, \cc{V}_t)\neq 0 \}$, this describes $f$ as
\begin{equation*}\label{poly2}f(\psi_t)=\det((\psi_t)_*).\end{equation*}Since $\eta_f$ is defined by the equation $$f(g\cdot \psi_t)=\eta_f(g) f(\psi_t),$$ we can  explicitly recover $\eta_f$ from the following computation 
\small
\begin{equation*}
\begin{aligned} f(g \cdot \psi_t)&= \det((g \cdot \psi_t)_*) \\&= \det \left[ \left( (g_{\beta} \oplus g_{\gamma} \oplus g_{\delta}) \circ \psi_t \circ (g_{\alpha})^{-1} \right)_* \right] \\   &=\left[ \det((g_{\beta})_*)\det((g_{\gamma})_*)\det((g_{\gamma})_*)\det((g_{\alpha})_*)^{-1} \right] \det((\psi_t)_*) \\ &= \left[ \det (g_{\beta})^{-\chi(E, L(-1,0)^{\beta})}\det (g_{\gamma})^{-\chi(E, L(0,-1)^{\gamma})} \det (g_{\delta})^{-\chi(E, L^{\delta})} \det (g_{\alpha})^{\chi(E, L(-1,1)^{\alpha})} \right]\det((\psi_t)_*).
\end{aligned}
\end{equation*}
\normalsize
Thus $$\eta_f=\eta_{a,b,c,d}$$with 
\begin{equation*}
\begin{aligned} a&=\chi(E, L(-1,1)^{\alpha}),\\ b&=-\chi(E, L(-1,0)^{\beta}),\\ c& =-\chi(E, L(0,-1)^{\gamma}), \\ d&=-\chi(E, L^{\delta}).
\end{aligned}
\end{equation*}
\end{remark}

\subsection{Corollaries of Theorem \ref{Main theorem}}We conclude this section by exploring some immediate corollaries of  Theorem \ref{Main theorem}.

First, we can get rid of some of the assumptions in Proposition \ref{third prop} at the expense of loosing the information about  torsion in $\Pic(M_{H_m}(\bv))$.

\begin{corollary}Let ${\bf v}=(r, \nu, \Delta)\in K(X)$ be a character with $r \geq 2$ and $\Delta> \frac{1}{2}$. Let $\epsilon \in \bb{Q}$ be sufficiently small (depending on $r$), $0 < |\epsilon| \ll 1$, and set $m=1+\epsilon$. 

If $\Delta=\DLP^{<r}_{H_m}(\nu)$ with two exceptional bundles $E_1, E_2$ associated to ${\bf v}$, then $$\rho(M_H({\bf v})) \cong \ZZ$$ and $\lambda$ is an epimorphism.
\end{corollary}

The next corollary is concerned with the position of certain good $H_m$-semistable characters $\bv$ relative to the branches of the $\DLP$-surface given by exceptional bundles of rank higher than the rank of $\bv$.

\begin{corollary} Let ${\bf v}=(r, \nu, \Delta)$ be a good $H_m$-semistable Chern character with $\Delta > \frac{1}{2}$, where $m=1+\epsilon$ and $\epsilon \in \bb{Q}$ is a  sufficiently small number depending on $r$, $0 < |\epsilon| \ll 1$. 

If ${\bf v}$  lies above the $\DLP^{<r}_{H_m}$-surface or has a single exceptional bundle $E$ associated to $\bv$, then for any exceptional bundle $F$  with $r(F)>r$ satisfying $$|(\nu-\nu(F))\cdot H_m |  \leq -\frac{1}{2}K_X\cdot H_m$$ we have $$\Delta>\DLP_{H_m,F}(\nu).$$
\end{corollary}
In other words, such $H_m$-semistable character ${\bf v}$ of rank $r$ does not lie on any branch of the DLP-surface given by an exceptional bundle of rank \emph{higher} than $r$. 
\begin{proof} Suppose on the contrary that $$\Delta=\DLP_{H_m, F}(\nu)$$for an exceptional bundle $F$ with $r(F)>r$ and  $$|(\nu-\nu(F))\cdot H_m |  \leq -\frac{1}{2}K_X\cdot H_m.$$  By \S \ref{easy kernel}, the class $[\overline{F}]$ lies in the kernel of the Donaldson homomorphism. 

\emph{Case 1}: $\Delta> \DLP^{<r}_{H_m}(\nu)$. Existence of $F$ as above then  contradicts the fact that by cases (1.c)  Theorem \ref{Main theorem} the Donaldson homomorphism $\lambda$ has a trivial kernel.

\emph{Case 2}:  $\Delta=\DLP^{<r}_{H_m}(\nu)$ with a single exceptional bundle $E$ associated to ${\bf v}$. According to case (2) of Theorem \ref{Main theorem} we get that $$[\overline{F}]=n [\overline{E}], \ \   r(\overline{E})<r.$$ Since $r(\overline{F})>r$, we conclude $n>1$. This shows that $[\overline{F}]$ is not a primitive Chern character. This is a contradiction, since $\overline{F}$ is an exceptional bundle and characters of exceptional bundles are primitive by Lemma \ref{exceptional stuff} (2).
\end{proof}

Finally, our last corollary states the conditions on character ${\bf v}$ under which the set of $H_m$-semistable sheaves of character ${\bf v}$ admitting a Gaeta-type resolution forms an open subset of the moduli space $M_{H_m}({\bf v})$ whose complement has codimension at least $2$.


\begin{corollary}\label{cor2} Let ${\bf v}=(r, \nu, \Delta)=(r, \varepsilon E + \varphi F, \Delta)\in K(X)$ be a good $H_m$-semistable Chern character with $r\geq 2, \Delta>\frac{1}{2}$ satisfying either \begin{itemize} \item $\Delta=\DLP^{<r}_{H_m}(\nu)$ with a single exceptional bundle $L$ associated to ${\bf v}$ with $r(L)=1$ and ${\mu_{H_m}(L)\geq \mu_{H_m}({\bf v})}$,
or

\item conditions \eqref{assumptions} with $$\floor{\varepsilon}+\floor{\varphi}\leq \varepsilon+\varphi< \floor{\varepsilon}+\floor{\varphi}+1$$
\end{itemize}
where $m=1+\epsilon$ and $\epsilon \in \bb{Q}$ is a  sufficiently small number depending on $r$, $0 < |\epsilon| \ll 1$. 

Then one can choose a line bundle $L$ such that for the complete family $\cc{V}_t/T$ of $\OO(1,1)$-prioritary sheaves admitting an $L$-Gaeta type resolution \eqref{L Gaeta} \begin{equation*}0 \to L(-1,-1)^{\alpha} \stackrel{\psi_t}{\to} L(-1,0)^{\beta} \oplus L(0,-1)^{\gamma} \oplus L^{\delta} \to \VV_t \to 0,\end{equation*}over the open subset  $$T \subset \mathbb{H}=\Hom \left(L(-1,-1)^{\alpha},  L(-1,0)^{\beta} \oplus L(0,-1)^{\gamma} \oplus L^{\delta} \right)$$ parameterizing injective sheaf maps with torsion-free cokernel we have  $T^{ss}\neq \emptyset$ and the image of the classifying morphism $$T^{ss} \xrightarrow{\phi_{\cc{V}_t|_{T^{ss}}}} M_{H_m}({\bf v})$$ is an open set whose complement has codimension $\geq 2$. 

Similarly, if ${\bf v}=(r,\nu,\Delta)$  satisfies either \begin{itemize} \item $\Delta=\DLP^{<r}_{H_m}(\nu)$ with a single exceptional bundle $L$ associated to ${\bf v}$ with $r(L)=1$ and ${\mu_{H_m}(L)< \mu_{H_m}({\bf v})}$,
or

\item conditions \eqref{assumptions} with   $$\floor{\epsilon}+\floor{\phi}+1< \epsilon+\phi < \floor{\epsilon}+\floor{\phi}+2,$$
\end{itemize} then one can choose a line bundle $L$ such that for the complete family $\cc{V}_t/T$ of  $\OO(1,1)$-prioritary vector bundles admitting the dual version of  an $L$-Gaeta type resolution \eqref{dual Gaeta} \begin{equation*} 0 \to  \VV_t \to  L(1,0)^{\alpha} \oplus L(0,1)^{\beta} \oplus L^{\gamma} \xrightarrow{\psi_t} L(1,1)^{\delta}\to 0,\end{equation*}over the open subset \begin{equation*}T\subset \mathbb{H}=\Hom \left(L(1,0)^{\alpha} \oplus L(0,1)^{\beta} \oplus L^{\gamma} ,  L(1,1)^{\delta} \right)\end{equation*}  parameterizing surjective sheaf maps we have $T^{ss}\neq \emptyset$ and  the image of the classifying morphism $$T^{ss} \xrightarrow{\phi_{\cc{V}_t|_{T^{ss}}}} M_{H_m}({\bf v})$$ is an open set whose complement has codimension $\geq 2$. 
\end{corollary}

\begin{proof}  Choose $L$ and $\VV_t/T$ as in the  proof of Proposition \ref{fourth prop} or \ref{fifth prop} or \ref{sixth prop} depending on which case we are considering. Suppose there is an irreducible Weil divisor $Z$ in the complement of the image. Then the corresponding line bundle $\OO(Z)$ lies in the kernel of $$\Pic(M_{H_m}({\bf v})) \xrightarrow{\phi^*_{\cc{V}_t|_{T^{ss}}}} \Pic^G(T^{ss}).$$ But in  the  proofs of Propositions \ref{fourth prop}, \ref{fifth prop}, \ref{sixth prop} we showed that the above map is injective, a contradiction. 
\end{proof}

\begin{remark} The restrictions on the numerical invariants in Corollary \ref{cor2} are  substantial conditions. When these conditions are not satisfied some of the exponents in a Gaeta-type resolution may become zero. As a result, we can no longer ensure that we can find $L$ such that for the resulting complete family $\VV_t/T$ of $\OO(1,1)$-prioritary sheaves admitting an $L$-Gaeta-type resolution we have $$\rk (\Pic^G(T^{ss}))\geq\rho(M_{H_m}(\bv)).$$This way, the homomorphism $$\Pic(M_{H_m}({\bf v})) \xrightarrow{\phi^*_{\cc{V}_t|_{T^{ss}}}} \Pic^G(T^{ss})$$ may no longer be injective.

\end{remark}

\section{Bad Chern characters}\label{last section}
In this section we show that when an $H_m$-semistable character ${\bf v}$ is \emph{bad}, the Picard number of $M_{H_m}({\bf v})$ is no longer controlled only by  the position of $\bv$ relative to the $\DLP^{<r}$-surface. One also needs to take into account the presence of the the $\Delta_i=\frac{1}{2}$-strata in complete families that force additional characters to be in the kernel of the Donaldson homomorphism $\lambda$.

We start with a continuation of Example \ref{bad character}.
\begin{example}\label{bad example} Let $m=1+\epsilon,$ where $\epsilon \in \bb{Q}$ and $0<\epsilon \ll 1$. Consider  character ${\bf v}=(4,-\frac{1}{4}E-\frac{1}{4}F, \frac{9}{16})$ from Example \ref{bad character}. In that example we considered the  family $\VV_t/T$ of  $\OO(1,1)$-prioritary sheaves of character ${\bf v}$ admitting an $\OO$-Gaeta-type resolution \begin{equation}\label{Gaeta3}0 \to \OO(-1,-1)^{2} \xrightarrow{\psi_t} \OO(-1,0)^{ 3} \oplus \OO(0,-1)^{ 3} \to \VV_t \to 0,\end{equation} where $$T \subset \mathbb{H}=\Hom \left(\OO(-1,-1)^{2},  \OO(-1,0)^{3} \oplus \OO(0,-1)^{3}  \right)$$ is the open subset parameterizing injective sheaf maps with torsion-free cokernel. We showed that $T$ is not empty, $\codim_{\bb{H}} (\bb{H} \setminus T)\geq 2$, the family $\cc{V}_t/T$ is complete  and \emph{any} $H_m$-semistable $\VV \in M_{H_m}({\bf v})$ is equal to $\cc{V}_t$ for some $t \in T$. This last property implies that the classifying morphism $$T^{ss} \xrightarrow{\phi_{\cc{V}_t|_{T^{ss}}}} M_{H_m}({\bf v})$$realizes $M_{H_m}({\bf v})=M^{s}_{H_m}({\bf v})$ as a geometric quotient of $T^{ss}$ under the action of $$\overline{G}=(GL(2)\times GL(3) \times GL(3))/ \bb{C}^*(Id, Id, Id)=G/\bb{C}^*(Id, Id, Id),$$
see \cite[Proposition 2.6]{DL85}.
Thus, by \cite[p. 32]{MR1304906} $$\Pic(M_{H_m}({\bf v}))\cong\Pic^{\overline{G}}(T^{ss}).$$

As before, we can compute the latter group  using the exact sequence from Proposition \ref{crossed morphisms}: \begin{equation}\label{crossed crossed}\OO^*(T^{ss})\to \Char(\overline{G}) \to \Pic^{\overline{G}}(T^{ss}) \to 0.\end{equation} We claim that the first map is not zero.  Take the $\Delta_i=\frac{1}{2}$-stratum $S_T=S_{H_m}({\bf v}_1,{\bf v}_2)$ described in Example \ref{bad character}.  Its closure $\overline{S_T}$ is a Weil divisor in $\bb{H}$,  so it is given by a  polynomial $f$: $$\overline{S_T}=V(f), \ f \in \bb{C}[\{ x_{ij} \} ].$$ Since $\overline{S_T}$ is $\overline{G}$-invariant, the complement $\bb{H} \setminus \overline{S_T}$ is  $\overline{G}$-invariant too, and the polynomial $f$ defines an invertible function on it, which by a remark after Proposition \ref{crossed morphisms} satisfies  $$f(\overline{g} h)=\eta_f(\overline{g}) f(h) \ \text{for some} \ \eta_f \in \Char(\overline{G}) \ \text{and any} \ \overline{g} \in \overline{G}, \ h \in \bb{H} \setminus \overline{S_T}.$$ Note that since $f(h)=f(\overline{g}h)=0$ for $h \in \overline{S}$ the above equation in fact holds for all $h \in \bb{H}$. 

We  show that $\eta_f$ is a nontrivial character, which would establish our claim. Assume, on the contrary, that $\eta_f$ is a trivial character so that $f$ is $\overline{G}$-invariant and, consequently, $G$-invariant:$$f \in \bb{C}[\{ x_{ij} \}]^G.$$ As the closure of any $G$-orbit contains the zero morphism $0 \in \bb{H}$, all $G$-invariant functions are constant  $$\bb{C}[\{ x_{ij} \}]^G = \bb{C}^*.$$But $f$ defines a non-empty divisor, so this is a contradiction.

Now, sequence \eqref{crossed crossed} gives $$\bb{Z}^2/\bb{Z}\eta_f \surj  \Pic^{\overline{G}}(T^{ss}) \cong \Pic(M_{H_m}({\bf v})).$$ Since the ample bundle generates a free $\bb{Z}$-submodule inside $\Pic({M_{H_m}}({\bf v}))$, it follows that the Picard number is equal to one $$\rho(M_{H_m}({\bf v}))=1.$$

Note, that an explicit computation of $\eta_f$ along the lines of remark \ref{explicit character} does not work in this case. The closure of the divisorial Shatz stratum is now described as
$$\overline{S_T} = \{ \psi_t \in T \ | \ \Hom(F_1, \cc{V}_t)\neq 0 \ \text{for some} \ F_1 \in M_{H_m}({\bf v}_1) \},$$ and compared to Remark \ref{explicit character} the computation is obstructed by the fact that $F_1$ is not a fixed bundle, but varies along its one-dimensional moduli space. 

However, note that by Remark \ref{primitive}, character ${\bf v}=(r, c_1, \chi)$ is primitive,
so for a generic choice of $m=\frac{p}{q}$ we have$$\gcd (r, c_1 \cdot (qH_m),\chi)=1.$$ Applying Proposition \ref{Maiorana} we get that $\Pic(M_{H_m}({\bf v}))$ is torsion-free and therefore $$\Pic(M_{H_m}({\bf v})) \cong \ZZ.$$ 

Let us also remark that using Proposition \ref{Yoshioka irreducibility} we can argue as in the proof of Proposition \ref{sixth prop} and show that $S_T$ is an irreducible subvariety of $T$.

\end{example}

\begin{example}\label{bad example 2} Let $m=1+\epsilon,$ where $\epsilon \in \bb{Q}$ and $0<\epsilon \ll 1$. Let $\{{\bw}_k\}_{k\in \bb{N}}$ be one of the infinite sequences of bad Chern characters constructed in Examples \ref{other bad characters} and \ref{other bad characters 2}. The same argument can be applied verbatim to the complete family $\cc{W}_t/T$ from Example \ref{other bad characters} to conclude that $$\Pic(M_{H_m}(\bw_k))\cong \ZZ$$ for any $k\in \bb{N}$.\end{example}  

It turns out that the techniques of the previous two examples allow us to tackle bad $H_m$-semistable Chern characters whenever they lie on a branch of the $\DLP$-surface given by a \emph{line bundle}. Note that $H_m$-semistable characters ${\bf v}$ with $r=2$ are always good (see Definition \ref{bad character}), so we can assume $r\geq 3$.

\begin{theorem}\label{Main theorem 2} Let ${\bf v}=(r, \nu, \Delta)\in K(X)$ be a character with $r \geq 3, \Delta>\frac{1}{2}$. Let $\epsilon \in \bb{Q}$ be sufficiently small (depending on $r$), $0 < |\epsilon| \ll 1$, and set $m=1+\epsilon$.

If $\bv$ is a bad character with $\Delta=\DLP^{<r}_H(\nu)$ with a single exceptional bundle $L$ associated to ${\bf v}$ with $r(L)=1$, then $$\Pic(M_{H_m}({\bf v})) \cong \ZZ,$$and $\lambda$ is an epimorphism with $$\ker \lambda \supsetneq \ZZ [\overline{L}].$$

\end{theorem}

\begin{proof} Assume first that $\mu_{H_m}(L) \geq \mu_{H_m}({\bf v})$. For a semistable $\VV$ of character $\bv$ we have $$\Ext^i(L,\VV)=0, \ i=0,1,2,$$by semistability and the fact that $L$ is associated to $\bv$.

Using this, one checks that  the  Beilinson-type resolution from  \cite[Proposition 5.1]{drezet:hal-01175951} coincides  with the $L$-Gaeta-type resolution and   \emph{every} $H_m$-semistable sheaf $\VV$ of character ${\bf v}$ is resolved as $$0 \to L(-1,-1)^{\alpha} \to L(-1,0)^{\beta} \oplus L(0,-1)^{\gamma} \to \VV \to 0.$$ 
We can then repeat the argument of Example \ref{bad example} to conclude \begin{equation*}\Pic(M_{H_m}({\bf v}))\cong \ZZ.\end{equation*}

When $\mu_{H_m}(L)<\mu_{H_m}({\bf v})$, the Beilinson-type resolution coincides with the dual version of the $L$-Gaeta type resolution and \emph{every} $H_m$-semistable sheaf $\VV$ of character ${\bf v}$ is resolved as $$0 \to  \VV \to  L(1,0)^{\alpha} \oplus L(0,1)^{\beta} \to L(1,1)^{\delta}\to 0 .$$ We can also repeat the argument of Example \ref{bad example} with straightforward modifications.
\end{proof}

It is interesting to further explore the geometry of $M_{H_m}(\bv)$ for bad characters $\bv$ as in the previous theorem, taking into account that these are unirational varieties (see \S2.6) with Picard number $\rho=1$. As a step in this direction, we consider the character $\bv$ from example \ref{bad example}.   
\begin{example}\label{bad example 3} We claim that for ${\bf v}=(4,-\frac{1}{4}E-\frac{1}{4}F, \frac{9}{16})$ and $m=1+\epsilon$ with $\epsilon \in \bb{Q}$, \  $0<\epsilon \ll 1$, we in fact have $$M_{H_m}(\bv) \cong \PP^3.$$

First, note that for a generic $\VV \in M_{H_m}(\bv)$ with the corresponding Gaeta-type resolution $$0 \to \OO(-1,-1)^{2} \xrightarrow{\psi} \OO(-1,0)^{ 3} \oplus \OO(0,-1)^{ 3} \to \VV \to 0$$the map $$pr_{\OO(0,-1)^{ 3}} \circ \psi: \OO(-1,-1)^{ 2} \to \OO(-1,0)^{ 3}$$ is an injective map of vector bundles. Therefore, we can expand the Gaeta-type resolution of $\VV$ into the following commutative diagram
\begin{center}
\begin{tikzcd}
            &                       & 0 \arrow[d]           & 0 \arrow[d]           &   \\
            &                       & \OO(0,-1)^3 \arrow[d] \arrow[r, equal] & \OO(0,-1)^3  \arrow[d]           &   \\
0 \arrow[r] & \OO(-1,-1)^2 \arrow[r, "\psi"] \arrow[d, equal] & \OO(-1,0)^3 \oplus \OO(0,-1)^3 \arrow[r] \arrow[d, "pr_{\OO(-1,0)^3}"] & \VV \arrow[r] \arrow[d] & 0 \\
0 \arrow[r] & \OO(-1,-1)^2 \arrow[r]           & \OO(-1,0)^3 \arrow[r] \arrow[d] & \OO(-1,2) \arrow[r] \arrow[d] & 0. \\
            &                       & 0                     & 0                     &  
\end{tikzcd}
\end{center}

So, next we consider extensions $$\xi:=[0 \to \OO(0,-1)^3 \to \VV_{\xi} \to \OO(-1,2) \to 0 ]$$ with $\xi=(\xi_1,\xi_2,\xi_3) \in \Ext^1(\OO(-1,2),\OO(0,-1))^{\oplus 3}.$ We assert that $\VV_{\xi}$ is $H_m$-semistable if and only if the corresponding vectors $\xi_1,\xi_2,\xi_3$ are linearly independent.

Indeed, suppose without loss of generality that $\xi_1=a\xi_2+b\xi_3.$ Consider the morphism $$\OO(0,-1)^2 \xrightarrow{A} \OO(0,-1)^3$$ given by the matrix $$A=\begin{bmatrix} a & 1 & 0 \\ b & 0 & 1 \end{bmatrix}.$$Then the induced map $$\Ext^1(\OO(-1,2),\OO(0,-1))^{\oplus 2} \xrightarrow{A_*}\Ext^1(\OO(-1,2),\OO(0,-1))^{\oplus 3}$$ sends $(\xi_2,\xi_3)$ to $(\xi_1, \xi_2,\xi_3)$. This fact translates into the following commutative diagram

\begin{center}
\begin{tikzcd}
            
0 \arrow[r] & \OO(0,-1)^2 \arrow[r] \arrow[d, hook, "A"] & \VV_{(\xi_2,\xi_3)} \arrow[r] \arrow[d, hook] & \OO(-1,2) \arrow[r] \arrow[d, equal] & 0 \\
0 \arrow[r] & \OO(0,-1)^3 \arrow[r]           & \VV_{(\xi_1,\xi_2,\xi_3)} \arrow[r] & \OO(-1,2) \arrow[r]  & 0.   
\end{tikzcd}
\end{center}
Since $$\mu_{H_m}(\VV_{(\xi_2,\xi_3)})=-\frac{1}{3}-\frac{\varepsilon}{3}>-\frac{1}{2}-\frac{\varepsilon}{4}=\mu_{H_m}(\VV_{(\xi_1,\xi_2,\xi_3)}),$$we conclude that $\VV_{(\xi_1,\xi_2,\xi_3)}$ is unstable.

Conversely, suppose $\VV_{\xi}$ is unstable. A rank $4$ bundle can be destabilized by subbundles or quotient bundles of rank $1$ or $2$. We will only sketch the argument in the case of a destabilizing subbundle of rank $2$ and leave the similar routine checks for the other cases to the reader. Suppose there is a destabilizing subbundle $$\cc{W} \subset \VV_{\xi}$$ with $r(\cc{W})=2.$ Since \begin{equation}\label{ineq} \mu_{H_m}(\cc{W})\geq \mu_{H_m}(\VV_{\xi})>\mu_{H_m}(\OO(0,-1)^3),\end{equation} there is no maps $\cc{W} \to \OO(0,-1)^3$ and, therefore, the composition $\cc{W} \to \VV_{\xi} \to \OO(-1,2)$ is not zero. We will further assume that this composition is surjective, leaving the check in the other case to the reader. In this case we have the following commutative diagram

\begin{center}
\begin{tikzcd}
    
0 \arrow[r] & \OO(a,b) \arrow[r] \arrow[d, hook, "B"] & \cc{W} \ \arrow[r] \arrow[d, hook] & \OO(-1,2) \arrow[r] \arrow[d, equal] & 0 \\
0 \arrow[r] & \OO(0,-1)^3 \arrow[r]           & \VV_{\xi} \arrow[r]  & \OO(-1,2) \arrow[r]  & 0 
\end{tikzcd}
\end{center}
with $a\leq 0$ and  $b\leq -1$. Furthermore, one checks that \eqref{ineq} is satisfied only if   $(a,b)=(0,-1)$. In this case, denote the extension defining $\cc{W}$ by $\zeta$ and write $$B=\begin{bmatrix} b_1 \\ b_2 \\ b_3 \end{bmatrix}, \quad b_i \in \bb{C}.$$ The induced map $$\Ext^1(\OO(-1,2),\OO(0,-1)) \xrightarrow{B_*}\Ext^1(\OO(-1,2),\OO(0,-1))^{\oplus 3}$$ sends $\zeta$ to $\xi=(\xi_1, \xi_2,\xi_3)=(b_1 \eta, b_2 \eta, b_3 \eta)$. Thus, we see that $\xi_1, \xi_2, \xi_3$ are linearly dependent.

Denote the locus of $\xi$ with linearly independent component vectors $\xi_1,\xi_2,\xi_3$ by $$U \subset \Ext^1(\OO(-1,2),\OO(0,-1))^{\oplus 3}.$$ By the above discussion, the universal extension over $U \times X$ defines a dominant morphism $$U \to M_{H_m}(\bv).$$Note that the isomorphism class of $\VV_{\xi}$ only depends on the hyperplane spanned by $\xi_1,\xi_2, \xi_3$ in the four-dimensional space $\Ext^1(\OO(-1,2),\OO(0,-1))$, so the above map factors through $$\PP(\Ext^1(\OO(-1,2),\OO(0,-1))^{\dual})\to M_{H_m}(\bv),$$as claimed. 

It is also interesting to note that the extensions $$0 \to F_2 \to \VV \to F_1 \to 0$$with $F_i \in M_{H_m}(\bv_i), \ \bv_1=(2,-\frac{1}{2}F, \frac{1}{2}), \ \bv_2=(2, -\frac{1}{2}E, \frac{1}2{}),$ give an embedding of a quadric into the moduli space $M_{H_m}(\bv)\cong \PP^3$: $$\PP^1 \times \PP^1 \hookrightarrow \PP^3.$$ Indeed, $M_{H_m}(\bv_i)\cong \PP^1$ by Theorem \ref{Main theorem} (3.b). The isomorphism class of $\VV$ is uniquely determined by $F_1$ and $F_2$ because  $\ext^1(F_1, F_2)=1$. Finally, the stability of $\VV$ follows from \cite[Lemma 10.8]{coskun2019existence} and $$\mu_{H_m}(F_2)<\mu_{H_m}(\VV)<\mu_{H_m}(F_1).$$
\end{example}

\begin{question}\label{interesting question} One can also show that $M_{H_m}(\bv)$ is isomorphic to a projective space for the bad characters $\bv$ of small rank listed in Example \ref{other bad characters} using the same method as in the previous example. However, it takes more and more work to directly check semistability for characters of higher and higher rank. An interesting question is whether $M_{H_m}(\bv)$ is isomorphic to a projective space for \emph{all} bad characters in the infinite sequence of Example \ref{other bad characters}. One can pose the same question for the bad characters constructed  in Example \ref{other bad characters 2} and, more generally, for \emph{all bad characters lying on a single branch of the $\DLP^{<r}$-surface} (keeping in mind our discussion in Question \ref{question}).  

\end{question}

We finish this paper by making the following conjecture about $\Pic(M_{H_m}(\bv))$ for \emph{all} characters $\bv$ with positive-dimensional moduli space, taking into account our remarks in Question \ref{question}. 
\begin{conjecture}\label{conjecture}Let ${\bf v}=(r, \nu, \Delta)\in K(\XX)$ be a character with $r \geq 2$ and $\Delta\geq \frac{1}{2}$. Let $\epsilon \in \bb{Q}$ be sufficiently small (depending on $r$), $0 < |\epsilon| \ll 1$, and set $m=1+\epsilon$.

\begin{enumerate}\item If    $\Delta > \DLP_{H_m}^{<r}(\nu)$,  then $$\Pic(M_{H_m}({\bf v}))\cong \ZZ^3$$ and $\lambda$ is an isomorphism.

\item \begin{enumerate} \item If ${\bf v}$ is a good character with  $\Delta= \DLP^{<r}_{H_m}(\nu)$ with a single exceptional bundle $E$ associated to $\bv$, then $$\Pic(M_{H_m}({\bf v})) \cong \ZZ^2$$ and $\lambda$ is an epimorphism with $$\ker \lambda \cong\ZZ [\overline{E}].$$

\item If ${\bf v}$ is a bad character with  $\Delta=  \DLP^{<r}_{H_m}(\nu)$ with a single exceptional bundle $E$ associated to $\bv$, then $$\Pic(M_{H_m}({\bf v})) \cong \ZZ$$ and $\lambda$ is an epimorphism with $$\ker \lambda\supsetneq\ZZ [\overline{E}].$$ 
\end{enumerate}

\item \begin{enumerate} \item If $\Delta=\DLP^{<r}_{H_m}(\nu)>\frac{1}{2}$ with at least two different exceptional bundles $E_1, E_2$ associated to $\bv$, then $$\Pic(M_{H_m}({\bf v})) \cong \ZZ$$ and $\lambda$ is an epimorphism with $$\ker \lambda\cong\ZZ [\overline{E_1}] + \ZZ [\overline{E_2}].$$

\item If $\Delta=\frac{1}{2}$, then $M_{H_m}({\bf v})$ is a projective space and$$\Pic(M_{H_m}({\bf v})) \cong \ZZ.$$  
\end{enumerate}
\end{enumerate}
\end{conjecture}
In other words, this conjecture states that the Picard number of $M_{H_m}(\bv)$ is determined by the position of $\bv$ relative to the $\DLP^{<r}_{H_m}$-surface and by whether character $\bv$ is good or bad.

It is likely that in order to verify this conjecture one needs to study some fine properties of full exceptional collections on $\XX$ along the lines of \cite{Rudakov_1989}, which could allow one to build Beilinson-type resolutions better suited for studying semistable sheaves of a given Chern character $\bv$. Another interesting problem in this direction is the problem of classifying $H_m$-semistable Chern characters with $\Delta=\frac{1}{2}$. This could further lead to a classification of bad $H_m$-semistable Chern characters, thus answering  Question \ref{question}. Finally, it remains an open question how to explicitly describe the second generator of $\ker \lambda$ in Theorem \ref{Main theorem 2} and  Conjecture \ref{conjecture} (2.b).

\newpage
\bibliographystyle{alpha}
\bibliography{References}{}

\begin{thebibliography}{{Zyu}94}

\bibitem[CH18]{coskun_huizenga_2018}
Izzet Coskun and Jack Huizenga.
\newblock Brill{-}{N}oether theorems and globally generated vector bundles on
  {H}irzebruch surfaces.
\newblock {\em Nagoya Mathematical Journal}, 238:1--36, 2018.

\bibitem[CH19]{coskun2019existence}
Izzet Coskun and Jack Huizenga.
\newblock Existence of semistable sheaves on {H}irzebruch surfaces.
\newblock {\em arXiv:1907.06739}, 2019.

\bibitem[CHW17]{67c4a1ff6f134980a50953a83ab223e7}
Izzet Coskun, Jack Huizenga, and Matthew Woolf.
\newblock The effective cone of the moduli space of sheaves on the plane.
\newblock {\em Journal of the European Mathematical Society}, 19(5):1421--1467,
  2017.

\bibitem[DLP85]{DL85}
Jean-Marc Dr{\'e}zet and Joseph Le~Potier.
\newblock {Fibr{\'e}s stables et fibr{\'e}s exceptionnels sur $\mathbb{P}^2$}.
\newblock {\em {Annales Scientifiques de l'{\'E}cole Normale Sup{\'e}rieure}},
  18:193--244, 1985.

\bibitem[Dr{\'e}87]{drezet1987fibres}
Jean-Marc Dr{\'e}zet.
\newblock Fibr{\'e}s exceptionnels et vari{\'e}t{\'e}s de modules de faisceaux
  semi-stables sur $\mathbb{P}^2$.
\newblock {\em Journal f{\"u}r die reine und angewandte Mathematik},
  380:14--58, 1987.

\bibitem[Dre88]{Dr88}
Jean-Marc Drezet.
\newblock Groupe de {Picard} des vari\'et\'es de modules de faisceaux
  semi-stables sur $\mathbb{P}(\mathbb{C})^2$.
\newblock {\em Annales de l'Institut Fourier}, 38(3):105--168, 1988.

\bibitem[Dr{\'e}91]{drezet:hal-01175951}
Jean-Marc Dr{\'e}zet.
\newblock {Points non factoriels des vari{\'e}t{\'e}s de modules de faisceaux
  semi-stables sur une surface rationnelle}.
\newblock {\em {Journal f{\"u}r die Reine und Angewandte Mathematik}},
  290:99--127, 1991.

\bibitem[Fog73]{Fogarty}
John Fogarty.
\newblock {Algebraic families on an algebraic surface II: the Picard scheme of
  the punctual Hilbert scheme}.
\newblock {\em American Journal of Mathematics}, 95(3):660--687, 1973.

\bibitem[Gor89]{Gorodentsev_1989}
Alexey Gorodentsev.
\newblock Exceptional bundles on surfaces with a moving anticanonical class.
\newblock {\em Mathematics of the {USSR}-Izvestiya}, 33(1):67--83, 1989.

\bibitem[HL10]{HL10}
Daniel Huybrechts and Manfred Lehn.
\newblock {\em The Geometry of Moduli Spaces of Sheaves}.
\newblock Cambridge University Press, 2 edition, 2010.

\bibitem[Hui17]{huizenga2017birational}
Jack Huizenga.
\newblock Birational geometry of moduli spaces of sheaves and {Bridgeland}
  stability.
\newblock {\em Surveys on Recent Developments in Algebraic Geometry, Proc.
  Symp. Pure Math}, 95:101--148, 2017.

\bibitem[LP97]{LP97}
Joseph Le~Potier.
\newblock {\em Lectures on vector bundles}.
\newblock Cambridge University Press, 1997.

\bibitem[MFK94]{MR1304906}
David Mumford, John Fogarty, and Frances Kirwan.
\newblock {\em Geometric invariant theory}, volume~34 of {\em Ergebnisse der
  Mathematik und ihrer Grenzgebiete (2)}.
\newblock Springer-Verlag, Berlin, third edition, 1994.

\bibitem[Nak93]{Nak93}
Tohru Nakashima.
\newblock On the moduli of stable vector bundles on a {Hirzebruch} surface.
\newblock {\em Mathematische Zeitschrift}, 212:211--222, 1993.

\bibitem[Qin92]{Qin92}
Zhenbo Qin.
\newblock Moduli of stable sheaves on ruled surfaces and their {P}icard groups.
\newblock {\em Journal f\"{u}r die reine und angewandte Mathematik},
  433:201--219, 1992.

\bibitem[Rud89]{Rudakov_1989}
Alexei Rudakov.
\newblock Exceptional vector bundles on a quadric.
\newblock {\em Mathematics of the {USSR}-Izvestiya}, 33(1):115--138, 1989.

\bibitem[Rud94]{Rudakov94}
Alexei Rudakov.
\newblock A description of {C}hern classes of semistable sheaves on a quadric
  surface.
\newblock {\em Journal f\"{u}r die reine und angewandte Mathematik},
  453:113--136, 1994.

\bibitem[Wal93]{Walter1993IrreducibilityOM}
Charles Walter.
\newblock Irreducibility of moduli spaces of vector bundles on birationally
  ruled surfaces.
\newblock {\em Algebraic Geometry, Catania/Barcelona, Lecture Notes in Pure and
  Appl. Math}, 200:201--211, 1993.

\bibitem[Yos95]{Yoshioka1995}
K\={o}ta Yoshioka.
\newblock The {B}etti numbers of the moduli space of stable sheaves of rank 2
  on a ruled surface.
\newblock {\em Mathematische Annalen}, 302(3):519--540, 1995.

\bibitem[Yos96a]{Yos96}
K\={o}ta Yoshioka.
\newblock A note on a paper of {J.-M.} {D}r\'{e}zet on the local factoriality
  of some moduli spaces.
\newblock {\em Internat. J. Math}, 7(6):843--858, 1996.

\bibitem[Yos96b]{Yosh96}
K\={o}ta Yoshioka.
\newblock The {Picard} group of the moduli space of stable sheaves on a ruled
  surface.
\newblock {\em J. Math. Kyoto Univ.}, 36(2):279--309, 1996.

\bibitem[{Zyu}94]{1994IzMat..42..163Z}
S.~{Zyuzina}.
\newblock {Constructibility of exceptional pairs of vector bundles on a
  quadric}.
\newblock {\em Russian Acad. Sci. Izv. Math}, 42(1):163--171, 1994.

\end{thebibliography}

\end{document}